\documentclass[11pt]{amsart}
\usepackage{geometry}                
\usepackage{graphicx}
\usepackage{amssymb}
\usepackage{epstopdf}
\theoremstyle{plain}
\newtheorem{theorem}{Theorem}[section]
\newtheorem{corollary}[theorem]{Corollary}
\newtheorem{lemma}[theorem]{Lemma}
\newtheorem{proposition}[theorem]{Proposition}

\theoremstyle{definition}
\newtheorem{definition}[theorem]{Definition}
\newtheorem{example}[theorem]{Example}

\theoremstyle{remark}
\newtheorem{remark}[theorem]{Remark}

\usepackage{hyperref} 

\newcommand{\ZZ}{\ensuremath{{\mathbb Z}}}
\newcommand{\DD}{\ensuremath{{\mathbb D}}}
\newcommand{\CC}{\ensuremath{{\mathbb C}}}
\newcommand{\CW}{\ensuremath{{\widehat{\mathbb C}}}}
\newcommand{\RR}{\ensuremath{{\mathbb R}}}
\newcommand{\NN}{\ensuremath{{\mathbb N}}}

\newcommand{\PSL}{\ensuremath{PSL(2,\CC)}}
\newcommand{\R}{\ensuremath{{\mathcal R}}}
\newcommand{\A}{\ensuremath{{\mathcal A}}}
\newcommand{\M}{\ensuremath{{\mathcal M}}}
\newcommand{\Z}{\ensuremath{{\mathcal Z}}}
\newcommand{\Po}{\ensuremath{{\mathcal P}}}
\newcommand{\G}{\ensuremath{{\mathbf G}}}
\newcommand{\CP}{\ensuremath{{\mathbb{CP}}}}
\font\myfont=cmr10 at 12pt
\newcommand{\e}{{\text{\myfont e}}}
\newcommand{\zbar}{{\bar{z}}}
\newcommand{\abs}[1]{\left\lvert#1\right\rvert}

\renewcommand{\Re}[1]{{\mathfrak{Re}\left(#1\right)}}

\renewcommand{\bar}[1]{{\overline{#1}}}
\newcommand{\del}[2]{\frac{\partial #1}{\partial #2}}

\title[Classification of rational 1--forms on \CW, via their isotropy group]{Classification of rational differential forms on the Riemann sphere, via their isotropy group}

\author[A. Alvarez--Parrilla, M.E. Fr\'ias--Armenta, C. Yee--Romero]{Alvaro Alvarez--Parrilla}
\address{Grupo Alximia SA de CV, 
M\'exico}
\email{alvaro.uabc@gmail.com}

\author[]{Mart\'in Eduardo Fr\'ias--Armenta}
\address{Departamento de Matem\'aticas, 
Universidad de Sonora, Hermosillo, Sonora, M\'exico}
\email{martineduardofrias@gmail.com}

\author[]{Carlos Yee--Romero}
\address{Facultad de Ciencias, Universidad Aut\'onoma de Baja California, 
M\'exico}
\email{carlos.yee@uabc.edu.mx}

\date{November 2, 2017}

\begin{document}

\begin{abstract}
We classify the rational differential 1--forms with simple poles and simple zeros on the Riemann sphere 
according to their isotropy group;
when the 1--form has exactly two poles the isotropy group is 
isomorphic to $\CC^{*}$, namely $\{z\mapsto az\ \vert\ a\in\CC, a\neq0\}$, and when the 1--form has 
$k\geq 3$ poles the isotropy group is finite.
\\
In particular we show that all the finite subgroups of $\PSL$ are realizable as isotropy groups for a rational 
1--form on $\CW$. 
We also present local and global geometrical conditions for their classification.  
The classification result enables us to describe the moduli space of rational 1--forms with finite isotropy that 
have exactly $k$ simple poles and $k-2$ simple zeros on the Riemann sphere. 
Moreover, we provide sufficient (geometrical) conditions for when the 
1--forms are isochronous. 
\\
Concerning the recent work of J.C.~Langer, we reflect on the strong relationship between our work and his and 
provide a partial answer regarding polyhedral geometries that arise from rational quadratic differentials on the 
Riemann sphere.
\end{abstract}


\maketitle

\section{Introduction}

The study of meromorphic 1--forms dates back to N.~H.~Abel and B.~Riemann who classified them as first, 
second and third type according to their regularity: whether they are holomorphic, they have zero residue poles 
or they have non-zero residue poles \cite{Brieskorn}. Later on, F.~Klein (see \cite{Klein}), describes 
geometrically the integrals of meromorphic 1--forms, in his personal memories (\emph{The Klein Protokols}) he 
further presents images of tessellations of the Riemann sphere related to the 1--forms he studies.

More recently R.~S.~Kulkarni \cite{KULKARNI} treats pseudo--Riemannian space forms of positive constant 
sectional curvature and studied the subgroups of isotropy under the orthogonal transformations.
In \cite{Adem} A.~Adem et al.~consider the problem of characterization of  finite groups that act freely on 
products of spheres.

In \cite{Mucino2}, M.~E.~Frias and J.~Muci\~no--Raymundo, study quotient spaces of holomorphic 1--forms 
over the Riemann sphere under the action of different groups. 
One of the most important groups they consider is $\PSL$ since it is the group of automorphisms of the 
Riemann sphere.
Also some continuous subgroups of $\PSL$ that appear as 
isotropy groups of rational 1--forms are studied.

On the other hand, A.~Alvarez--Parrilla and J.~Muci\~no--Raymundo, see \cite{AlvarezMucino2}, while 
studying (complex) analytic 1--forms over the Riemann sphere that have $r$ zeros and either a pole of 
order $-(r+2)$ or an essential singularity (satisfying certain requirements) at $\infty\in\CW$, classify their 
isotropy subgroups; showing that exactly the cyclic groups $\ZZ_{s}$ appear as non--trivial isotropy groups.

In \cite{Julio}, J. Maga\~na shows that there are three equivalent complex structures on the space 
$\Omega^{1}(-s)$ of rational 1--forms on the sphere with exactly $s\geq2$ simple poles: coefficients, 
residues--poles and zeros--poles of the 1--forms 
(note that in the case of the characterization of $\eta\in\Omega^{1}(-s)$ by its zeros--poles, it is also 
necessary to specify the principal coefficient).
He proves 
that the subfamily of rational isochronous 1--forms 
$\mathcal{RI}\Omega^{1}(-s)$ is a $(3s-1)$--dimensional real analytic sub--manifold of $\Omega^{1}(-s)$. 
Since the complex Lie group $\PSL$ acts holomorphically on $\Omega^{1}(-s)$ with the action being proper 
for $s\geq3$, an understanding of the non--trivial isotropy groups for $\eta\in\Omega^{1}(-s)$ allows him to 
prove that $\mathcal{RI}\Omega^{1}(-s)/\PSL$ is a stratified manifold; with the singular orbits arising precisely 
from the 1--forms with non--trivial isotropy. 
Moreover, 
he also shows that every finite subgroup of $\PSL$ appears as an isotropy subgroup for some 
isochronous $\eta\in\mathcal{RI}\Omega^{1}(-s)$.

The Lie group $PSL(2,\CC)$ acting on the space of $1$-forms 
$\Omega^1\{\CC\}=\bigcup_{s}\Omega^1(-s)$ 
leaves invariant the residues and the associated metric.
A natural question is to consider the quotient space 
\\ \centerline{
$\Omega^1\{\CC\}/PSL(2,\CC)$}
and ask when the fiber is not 
$PSL(2,\CC)$.
A quotient space has singularities when the isotropy group is not the identity. 
Hence this last question is related to which 1--forms do not have trivial isotropy group, particularly which 
1--forms have finite isotropy group under the action of $PSL(2,\CC)$.

\medskip
\noindent
However the classification question for isotropy groups of $\eta\in\Omega^{1}\{\CC\}$, \emph{how do the finite 
subgroups of $\PSL$ realize as isotropy groups of rational 1--forms over the Riemann sphere} is still 
unanswered.

\medskip
\noindent
Since rational 1--forms with simple poles and simple zeros on the Riemann sphere are an open and dense\footnote{
The set of polynomials $P$ of degree at most $k$ with at least one multiple root can be characterized as the algebraic variety given by discriminant of $P$ and $P'$ equal to cero (the discriminant being an algebraic equation of the $k+1$ coefficients of $P$), which shows that this set is closed and not dense in the vector space (of dimension $k+1$) of polynomials of degree at most $k$. Thus the polynomials of degree at most $k$ with simple roots are an open and dense set in the vector space of polynomials of degree at most $k$.
\\
Considering now rational 1--forms, apply the above to the numerator and denominator.
} 
set in the (vector) space of rational 1--forms on the Riemann sphere, we shall from hereafter concern 
ourselves with rational 1--forms on the Riemann sphere with simple poles and simple zeros, unless we specify 
otherwise.
In this paper we:
\begin{enumerate}
\item Show that all finite subgroups of $PSL(2,\CC)$ are realizable  as isotropy groups of some 
$1$-form (not necessarily isochronous).
\item 
Classify the rational
$1$-forms that have finite isotropy group $G$.
This is done first, in Theorem \ref{caracterizacion1}, by considering the complex structure arising from the 
location of poles and zeros and requiring that the sets of poles and zeros be $G$--invariant. 
However since this is not enough, we provide an easy to check \emph{local--geometric condition} that states 
that every non--trivial element of the group $g\in G$ has two fixed points and that these must be poles or zeros 
of the $G$ invariant 1--form.

\noindent
From this main theorem we then prove classification theorems based on whether $G$ is: a platonic subgroup 
(Theorem \ref{main2}), a dihedral group (Theorem \ref{teoremaDihedrico} and Theorem \ref{D22}), or a cyclic 
group (Theorem \ref{Z2} and Theorem \ref{main3b}).

\item
We summarize the above theorems in terms of \emph{global--geometric conditions} as 
Corollary \ref{genclass}.

\noindent
The main result: \emph{the classification of rational 1--forms on the Riemann sphere with simple poles and simple zeros according to their isotropy group} follows immediately as Theorem \ref{mainresult}.

\item 
In \S\ref{sec:variedad} we describe the moduli space of rational 1--forms with finite isotropy that have exactly $k$ 
simple poles and $k-2$ simple zeros. This is done by placing $\ell_{1}$ zeros and $\ell_{2}$ poles in a 
quasi--fundamental region $\widehat{\R}_{G}$, where the quasi--fundamental region is a simply connected set 
containing one representative of those orbits that have the same number of elements as $G$, 
Theorem \ref{variedad}.

\item In Theorem \ref{sufficientIsochronous}, we provide sufficient conditions for 
isochronicity with finite isotropy. 

\item In \S\ref{secLanger} we reflect on the strong relationship between the work of J.~C.~Langer, 
see \cite{Langer},
and our work; partially answering a question on polyhedral geometries associated to rational quadratic 
differentials, Proposition \ref{geomPolyIsocrono}.

\item In \S\ref{ejemplos}, we provide interesting and beautiful examples of realizations of 1--forms with finite 
isotropy group $G$ for each finite subgroup of $\PSL$: in \S\ref{a4} the case of $A_{4}$ is considered, 
in \S\ref{s4} the case of $S_{4}$ is presented, in \S\ref{a5} the case of $A_{5}$, and finally in \S\ref{cyclic} 
and \S\ref{dihedric} the cases of the cyclic and dihedral groups, respectively, are given.

\end{enumerate}

As a first observation, note that rational 1--forms $\eta$, with finite isotropy group, must have at least one zero 
and three poles (by Gauss--Bonnet $Card(\Po_{\eta})-Card(\Z_{\eta})$ $=2$, where 
$\Po_{\eta}$ and $\Z_{\eta}$ are the set of poles and set of zeros of $\eta$ respectively), otherwise the 
isotropy group is continuous.
Also, given a 1--form  
$\eta\in\Omega^{1}(-s)$ and a subgroup $G<\PSL$, 
the two obvious and natural conditions are that $\Po_{\eta}$ and $\Z_{\eta}$  be invariant under 
the action of $G<\PSL$. These conditions are necessary but not sufficient for invariance of 
the 1--form $\eta$.

\noindent
In fact there is an additional obstruction which can be observed in the following examples.

\begin{example}\label{ejemplo1}
Consider the  case of the cyclic group $\ZZ_4$ generated by $T(z) = i z$. 
Then the 1--form:
$$\eta=f(z)dz= \lambda\ \frac{z\ dz}{(z-1)(z-i)(z+1)(z+i)},$$   
has poles at $1$,$-1$, $i$ and $-i$, and has a  zeros at infinity and origin.
It is clear that $T$ fixes the set of poles and the set of zeros of $\eta$.
The  push--forward  of $\eta$ via $T$ is:
$$T_{*}\eta=\lambda\ \frac{-w\ dw}{(w-1)(w-i)(w+1)(w+i)}.$$ 
We observe that  
$T_{*}\eta=-\eta$, hence $T^{2}(z)=-z$ and of course $T^{2}_{*}\eta=\eta$, so the isotropy group of $\eta$ is $\ZZ_{2}$ generated by $T^{2}$.
\end{example}

\begin{example}\label{ejemplo2}
Another example is 
$$\eta=\frac{(z^2-4) dz}{z^4-1}.$$ 
Here the set of zeros and the set of poles are both invariant under the group $\ZZ_{2}$ generated by 
$T(z)=-z$, but $\eta$
has trivial isotropy group.
\end{example}

These two examples are the minimum (in terms of the number of poles) such that hypothesis 1 and 2 of 
Theorem \ref{caracterizacion1} are satisfied (the set of poles and zeros are invariant), but condition 3  of 
Theorem \ref{caracterizacion1} is not satisfied: in 
Example \ref{ejemplo1}, $T(z)=i z$ is of order 4 and fixes a zero 
(since $T$ is of order $\geq3$, the fixed points of $T$ should be poles);
in Example \ref{ejemplo2}, $T(z)=-z$ is of order 2, but 
$0,\infty\in\CW$ 
are fixed points of $T$ that are not zeros or poles of $\eta$.

\begin{example}\label{ejemplo3}
A non--trivial example is 
$$\eta=\frac{(z^3-27)(z^3-1/27) dz}{z(z^3-8)(z^3-1/8)}.$$ 
Here the set of zeros and the set of poles are both invariant under the group $\DD_{3}$ generated by 
$\{T_1(z)=1/z,T_2(z)=e^{2\pi i/3}z\}$, but $\eta$ has  isotropy group $\ZZ_3$.
Of course condition 3 of Theorem \ref{caracterizacion1} is also satisfied with $G=\ZZ_{3}$, 
but not with $G=\DD_{3}$.
\end{example}


\section{Background}
\smallskip
\noindent
First recall the classification of the finite subgroups of $PSL(2,\mathbb{C})$ (see for instance \cite{Klein1}, 
\cite{Toth} chapter 1 and \cite{Du}).
The resulting possible cases are the conjugacy classes (in $\PSL$) of the: 
\begin{itemize}
\item group of isometries of the tetrahedron, isomorphic to $A_4$, 
\item group of isometries of the cube, or of the octahedron, both isomorphic to $S_4$,
\item group of isometries of the dodecahedron, or of the icosahedron, both isomorphic to $A_5$,
\item cyclic groups ($\ZZ_{n}$, $n>1$) and
\item dihedral groups ($\DD_{n}$, $n>1$).
\end{itemize}

On the other hand, the
notion of {\it center} goes back to Poincar\'{e} (see \cite{Po}). He
defined it for differential systems on the real plane; 
{\it i.e.} given a vector field $X$, a \emph{center for $X$} is a
singular point surrounded by a neighborhood filled by closed
orbits of $X$ with the unique exception of the singular point.
An \emph{isochronous center} is a center all of whose orbits have the same period. 
In particular for the case at hand, that is for complex analytic vector fields $X$, a simple \emph{flow--box} 
argument shows that if $X$ has a center then it is an isochronous center. 

The centers and particularly the isochronous centers have been studied widely in Hamiltonian systems, see for 
example \cite{CLV2}, \cite{Llibre}; in holomorphic systems see for example \cite{Mucino1}, \cite{Mucino2}; and for a wide 
survey see \cite{Chavarriga-Sabatini}. 

An \emph{isochronous field} 
is a vector field such that all the zeros are centers. 
We shall say that \emph{a 1--form is isochronous} if its associated vector field (given by the 
correspondence \eqref{correspondencia} below) is isochronous.
In \cite{Mucino1} the isochronous fields arising from polynomial fields are classified and studied.

On the Riemann sphere infinity is a regular point, thus meromorphic is equivalent to (complex) analytic  
\cite{AlvarezMucino1}, \cite{Mucino1}, \cite{Mucino2}. Rational functions have only zeros and poles, so in our context 
(complex) analytic, meromorphic and rational \emph{functions} are all equivalent.

\noindent
Moreover, as is explained in \cite{AlvarezMucino1}, \cite{AlvarezMucino2}, 
on any Riemann surface $M$ there is a one to one canonical correspondence between:
\begin{enumerate}
\item 
Singular analytic \emph{vector fields} $X=f\del{}{z}$.
\item 
Singular analytic \emph{differential forms} $\omega=\frac{dz}{f}$.
\item 
Global singular analytic (additively automorphic, probably multivalued) 
\emph{distinguished parameters (functions)}\\
\centerline{ $
\Psi (z)= \int^z \omega .
$}
\end{enumerate}

\noindent
This correspondence can be represented by the following diagram 
(see  \cite{AlvarezMucino1}, \cite{AlvarezMucino2}, \cite{Mucino1}, \cite{Mucino3} for the complete details of the diagram 
and further correspondences):
\medskip
\begin{equation}\label{correspondencia}
\begin{array}{cccc}
&\qquad \omega_{X}=\frac{dz}{f}  \\
&\qquad \qquad \qquad \qquad \nwarrow \searrow \\
& \downarrow \uparrow &  X = f \del{}{z}, \\ 
&\qquad \qquad \qquad \qquad \nearrow \swarrow \\
&\qquad \Psi_X(z)= \int^z \omega_X 
\end{array}
\end{equation}

\medskip
\noindent
where the subindex $X$ recalls the dependence on the original vector field, 
which we omit when it is unnecessary. 

In terms of pullbacks and push--forwards, if $T\in PSL(2,\mathbb{C})$ and $X$ is the singular analytic vector 
field associated to the 1--form $\omega_{X}$ then $T^*X$ is the singular analytic vector field corresponding to 
the 1--form $T_*\omega_{X}$, see \cite{AlvarezMucino1}, \cite{Mucino1}, \cite{Mucino2}. 
In fact,  as shown in \cite{AlvarezMucino1}:

\noindent
\emph{`Every singular analytic vector field $X$ on $M$ can be expressed as the pullback, 
via certain singular analytic probably multivalued maps $\Psi$ and $\Phi$, of the 
simplest analytic vector fields $\del{}{t}$ or $-w\del{}{w}$ on 
the Riemann sphere $\CW$.'} 

\smallskip
\noindent 
In other words the following commutative diagram holds true
\begin{center}
\begin{picture}(210,55)
\put(-75,20){\vbox{\begin{equation}\label{diagrama-basico}\end{equation}}}

\put(17,8){$ (\CW,\del{}{t}) $}
\put(50,12){\vector(1,0){80}}
\put(73,0){$\exp(-t)$}

\put(135,8){$ (\CW,-w\del{}{w}) $,}
\put(122,34){$\Phi $}
\put(113,42){\vector(1,-1){20}}

\put(76,45){$(M,X)$}
\put(52,34){$\Psi$}
\put(70,42){\vector(-1,-1){20}}

\end{picture}
\end{center}

\noindent
where $\Phi = \exp \circ (-\Psi)$. 
In the language of differential equations:
\begin{enumerate}

\item [$\bullet$]
$X=\Psi^{*}(\del{}{t})$ means that $X$ has 
a \emph{global flow--box}, {\it i.e.} the local rectifiability can be continued 
analytically to $M$ minus the singular set of $X$.

\item [$\bullet$]
$X=\Phi^{*}(-w\del{}{w})$ states that $X$ is the  \emph{global Newton vector field 
of $\Phi$}, 
{\it i.e.} $X$ has sinks exactly at the zeros of $\Phi$.
\end{enumerate} 

\noindent
In particular, the fact that every singular analytic vector field $X$ is a global Newton vector field, is used 
to visualize $X$, 
and hence the associated 1--form $\omega_{X}$, see 
\cite{AlvarezMucino1}, \cite{Alvaro-Carlos-Selene-Jesus} for further details. 
\begin{remark}
Because of the duality between vector fields $X$ and the associated 1--form $\omega_{X}$, 
the poles of $X$ are the zeros of $\omega_{X}$, and the zeros of $X$ are the poles of $\omega_{X}$. 
In this work we will agree to speak of poles and zeros of the \emph{1--form} unless explicitly stated.
\end{remark}

\section{Classification of rational 1--forms on the Riemann sphere with simple poles and simple zeros according to their isotropy group}
It is clear that the number $k$ of poles of a rational 1--form on the Riemann sphere is at least 2.

\begin{remark}\label{separacionviak}
1. Any rational 1--form that has exactly 2 poles on the Riemann sphere is conjugate, via an element of $\PSL$, 
to $\eta=\frac{\lambda}{z}\,dz$, for some $\lambda\in\CC^{*}$, and thus its isotropy group is isomorphic 
to $\CC^{*}=\{z\mapsto az\ \vert\ a\in\CC^{*}\}$, see \cite{Julio} lemma 3.17 pp.~44.

\noindent
2. The rational 1--forms that have at least 3 poles have finite 
isotropy group, see \cite{Julio} corollary 3.6 pp~34 
(the idea being that an element $g$ of the isotropy group will permute the poles of the 1--form, since the 1--form has a finite number of poles the result follows). 
\end{remark}

Hence, in what follows we shall classify the rational 1--forms whose isotropy group are non--trivial, in the understanding that any other rational 1--form with at least 3 poles has trivial isotropy group.

\subsection{Local--geometric characterization of rational 1--forms with finite isotropy}
In this section we provide analytical classification results for all possible rational 1--forms with simple poles and simple zeros on $\CW$ that have non--trivial finite isotropy groups. 

As previously mentioned, $\mathcal{P}_{\eta}$ and $\mathcal{Z}_{\eta}$ are, respectively, the set of poles and the set of zeros of the 1--form $\eta$.

\begin{lemma}\label{afirmacion0}
If $T\in\PSL$ has invariant set with cardinality at least 3, then $T$ is an elliptic transformation.
\\
In particular if $\eta$ is a rational 1-form and $T\in\PSL$ leaves invariant the 1--form, then $T$ is an elliptic 
transformation.
\end{lemma}
\begin{proof}
If $T$ is a homothecy, a translation (or a conjugate of either), then it's invariant set has at most two elements 
in $\CW$. On the other hand since $T_{*}\eta=\eta$ then $\mathcal{P}_{\eta}$ must be $T$--invariant. 
Thus $Card(\mathcal{P}_{\eta})\leq 2$, contradiction. Hence $T$ must be an elliptic transformation. 
\end{proof}

The following is a very simple result which will be useful.
\begin{lemma}\label{pushForma}
Let $\eta$ be a 1--form whose isotropy group is $G<\PSL$. If $T\in\PSL$ then $T_{*}\eta$ has isotropy group 
$T G T^{-1}$.
\end{lemma}
\begin{proof}
Let $g\in G$, then a simple calculation using the fact that $g_{*}\eta=\eta$ shows that 
\begin{equation}
\label{pullbackdeconjugacion}
(TgT^{-1})_{*}(T_{*}\eta)=T_{*} g_{*} T^{-1}_{*}(T_*\eta)=T_*\eta.
\end{equation}
 
\end{proof}

\begin{remark}\label{todaselipticas}
Given a non--trivial $g\in G$, by Lemma \ref{afirmacion0}, $g$ must be an 
elliptic transformation. Denote the order of $g$ by $k\geq 2$.
There exists $T\in \PSL$ such that $T(x)=0$, $T(y)=\infty$ where $\{x,y\}$ are the fixed  points of $g$, hence 
$\widehat{g}=T g T^{-1}$ fixes $\{0,\infty\}\subset\CW$, in fact $\widehat{g}(z)= \e^{i 2\pi/k}z$ 
(in fact, there are an infinitude of such $T\in\PSL$: three points completely determine a unique $T\in\PSL$).
\end{remark}

\begin{theorem}[Characterization of rational 1--forms with finite isotropy]\label{caracterizacion1}
\hfill\\
Let $G$ be a finite subgroup of $\PSL$, and  let $\eta$ be a $1$-form with simple poles and zeros. 
\\
$\eta$ is $G$--invariant if and only if the following three conditions are met:

\begin{enumerate}
\item[1)] $\Po_{\eta}$ is $G$--invariant.

\item[2)] $\Z_{\eta}$ is $G$--invariant.

\item[3)] For each non--trivial $g\in G$, let $\{x,y\}$ be the set of fixed points of $g$. 
One of the next statements is satisfied:
\begin{enumerate}
\item[a)] $g^2$ is the identity and $\{x,y\}\subset \Po_{\eta}\cup \Z_{\eta}$,
\item[b)] $g$ is of order greater than 2 and $\{x,y\}\subset \Po_{\eta}$.
\end{enumerate}
\end{enumerate}
Moreover, $G$ is the maximal group, as a subgroup of $\PSL$, satisfying (1)--(3) if and only if  $\eta$ has 
isotropy $G$.  
\end{theorem}
\begin{proof}
($\Rightarrow$)\\
Since $\eta$ is $G$--invariant, it is clear that (1) and (2) hold.

\noindent
Let $T$ and $\widehat{g}$ be as in Remark \ref{todaselipticas}.
To prove condition (3), consider the orbits under the action of $\widehat{g}(z)=\e^{i 2\pi/k}z$. 
For $z_{0}\in\CW\backslash\{0,\infty\}$, the orbit of $z_{0}$ has $k$ elements, so $T_{*}\eta$ has an 
expression of the form
\begin{equation}\label{Tforma}
T_*\eta=\lambda\frac{\prod_{\iota=1}^{\ell_2}(z^k-q_\iota^k)}{z^{d_1}\prod_{\iota=1}^{\ell_1}(z^k-p_\iota^k)}dz,
\quad\text{for }\lambda\in\CC^{*},
\end{equation}
where $\{q_{\iota}\}$ are zeros of $T_{*}\eta$ and similarly $\{p_{\iota}\}$ are poles of $T_{*}\eta$. 
Note that 
\begin{itemize}
\item if the origin is a pole then $d_1=1$,
\item if the origin is a zero then $d_1=-1$,
\item if the origin is a regular point then $d_1=0$.
\end{itemize}
Of course at $\infty\in\CW$ there could also be a pole, zero or a regular point, hence we shall have the 
existence of $d_2\in\{-1,0,1\}$ with $d_{2}$ following the same conventions as $d_{1}$ but at $\infty\in\CW$.
\\
By Gauss--Bonnet: 
\\
\centerline{$k (\ell_1- \ell_2)+d_1+d_2=2$.}
\\
Let $GB=k (\ell_1- \ell_2)$, then $GB=(2-d_1-d_2)\in\{0,1,2,3,4\}$. 
We examine these cases.
\begin{description}
\item [$GB=0$:] implies that $d_1=d_2=1$, $\ell_1=\ell_2$, for arbitrary $k\geq 2$ and condition (3.a) follows. 
\item [$GB=1$:] implies that $k=1$, which leads to a contradiction.
\item [$GB=2$:] implies that $k=2$, $d_1=-d_2=\pm 1$, and $\ell_1=\ell_2+1$ so condition (3.b)  holds:
either $0$ is a pole and $\infty$ is a zero, or viceversa.
\item [$GB=3$:] implies that $k=3$,  $\ell_1=\ell_2+1$,  so it follows that $d_1=-1$ and $d_2=0$, 
or $d_1=0$ and $d_2=-1$, but $\widehat{g}_*T_*\eta=e^{i 2 \theta \pi/3 }T_*\eta\neq T_*\eta$ with 
$\theta\in \{1,2\}$, which is a contradiction.
\item [$GB=4$:] we have two sub cases
\begin{enumerate}
\item[a)] $k=2$, $\ell_1=\ell_2+2$, $d_1=d_2=-1$ so the condition (3.a)  holds:
both fixed points $\{0,\infty\}$ are zeros.
\item[b)] $k=4$, $\ell_1=\ell_2+2$, $d_1=d_2=-1$ but $\widehat{g}_*T_*\eta=e^{i  \pi/2 }T_*\eta$ contradiction.
\end{enumerate}
\end{description}

\smallskip
\noindent
($\Leftarrow$)
Because of (1) and (2) Remark \ref{todaselipticas} holds.
\\
Assume that given a non--trivial $g\in G$, conditions (3.a) or (3.b) are met:
\begin{description}
\item [(3.a)] Thus by Gauss--Bonnet equation \eqref{Tforma} is:
$$
T_*\eta=\lambda\frac{\prod_{\iota=1}^{\ell}(z^k-q_\iota^k)}{z\prod_{\iota=1}^{\ell}(z^k-p_\iota^k)}dz.
$$
\item [(3.b)] In this case $\widehat{g}(z)=-z$ and by Gauss-Bonet we have three sub--cases for 
equation \eqref{Tforma}:
\begin{itemize}
\item[] $$
T_*\eta=\lambda\frac{z \prod_{\iota=1}^{\ell}(z^2-q_\iota^2)}{\prod_{\iota=1}^{\ell+1}(z^2-p_\iota^2)}dz,
$$
\item[] $$
T_*\eta=\lambda\frac{ \prod_{\iota=1}^{\ell}(z^2-q_\iota^2)}{z\prod_{\iota=1}^{\ell+1}(z^2-p_\iota^2)}dz,
$$
\item[] $$
T_*\eta=\lambda\frac{z \prod_{\iota=1}^{\ell}(z^2-q_\iota^2)}{\prod_{\iota=1}^{\ell+2}(z^2-p_\iota^2)}dz.
$$
\end{itemize}

\end{description} 
In all the cases $\widehat{g}_*T_*\eta=T_*\eta$ and so $g_*\eta=\eta$. 
\end{proof}
\begin{remark}\label{elmeollodelasunto}
Condition (3) of Theorem \ref{caracterizacion1} is key, in words it states that the fixed points for the 
non--trivial elements of the group must be zeros or poles.
\end{remark}

\begin{remark}\label{contraejemplos}
Condition (3) is a non--trivial condition.
For examples of rational 1--forms with simple poles and simple zeros that satisfy conditions (1) and (2) of 
Theorem \ref{caracterizacion1} but are not invariant under the action of $G$, see Examples \ref{ejemplo1}, 
\ref{ejemplo2}, \ref{ejemplo3}, \ref{ejemploTetra} and \ref{ejemploLangerOcta}.
\end{remark}
\noindent
It is to be noted that for $G\cong A_{5}$ condition (3) is automatically satisfied.
The statement and proof is presented in \S\ref{MobiusPolyhedra} as Proposition \ref{caracterizacionA5}.

\smallskip
Note that even though the classification result given by Theorem \ref{caracterizacion1} is quite general, it has a 
``local--geometric'' nature (in the sense that one needs to check a condition for each non--trivial element 
of $G$), a natural question is to ask whether there is a more ``global--geometric'' characterization. As we will 
see in the next section this indeed turns out to be the case.

\subsubsection{The case of $G\subset\PSL$ finite and not isomorphic to $\ZZ_{n}$.}\label{MobiusPolyhedra}

Recalling that the platonic polyhedra, namely the tetrahedra, octahedra (cube), icosahedra (dodecahedra) 
have isotropy group isomorphic to the finite subgroups $A_{4}$, $S_{4}$, $A_{5}$ respectively; 
a natural question is to ask whether there exist polyhedra whose isotropy groups are isomorphic to the cyclic 
and the dihedric groups, $\ZZ_{n}$ and $\DD_{n}$ respectively. 

\noindent
The answer is no, however by allowing \emph{spherical 
polyhedra} we obtain a positive answer in the case of $\DD_{n}$.

\begin{definition}
We will say that $\A\subset\CW$ is a \emph{spherical polyhedra} if $\A$ is a tiling of the sphere in which the 
sphere is partitioned by great arcs into spherical polygons. 
\end{definition}

\begin{definition}
Let $A$ be a polyhedra or a spherical polyhedra, an embedding $H:A\longrightarrow\A\subset\CW$ is a 
\emph{conformal embedding} if the image of every edge of $A$ is an arc of a circle in $\CW$ and the angle 
formed by any two edges of $A$ is preserved by $H$.
\\
Moreover, we shall say that $\A$ is a \emph{platonic polyhedra embedded in $\CW$} if it is a conformal 
embedding of a platonic polyhedra.
\end{definition}

\begin{definition}
We shall say that a \emph{(regular) $n$--gonal hosohedron $\mathfrak{H}_{n}$} is the spherical polyhedra 
obtained by embedding, via the inverse of a stereographic projection $\Psi^{-1}:\CC\longrightarrow\CW$, the 
set $\mathfrak{H}_{n}$ formed by the $n$ straight line segments $\{L_{j}\}_{j=1}^{n}$ that start at the origin 
(with angle $2\pi j/n$ respectively, $j=1,\ldots,n$) together with $0,\infty\in\CW$.
\\
We shall say that a \emph{(regular) $n$--gonal dihedron $\mathfrak{D}_{n}$} is the spherical polyhedra 
obtained by embedding, via the inverse of a stereographic projection $\Psi^{-1}:\CC\longrightarrow\CW$, the 
set $\mathfrak{D}_{n}$  formed by a regular $n$--sided polygon with vertices at 
$\{\e^{i 2\pi j/n}\}_{j=1}^{n}\subset\CC$ on the unit circle.
\end{definition}
It is immediately clear that the $n$--gonal hosohedron $\mathfrak{H}_{n}$ and the $n$--gonal dihedron 
$\mathfrak{D}_{n}$,
\begin{enumerate}
\item[1)] are spherical polyhedra, 
\item[2)] are duals of each other, moreover the $2$--gonal hosohedron and the $2$--gonal dihedron are self duals,
\item[3)] the isotropy group of either is precisely the dihedric group $\DD_{n}$.
\end{enumerate}
See figure \ref{HosohedraDihedra}.

\begin{figure}[htbp]
\begin{center}
\includegraphics[width=.95\textwidth]{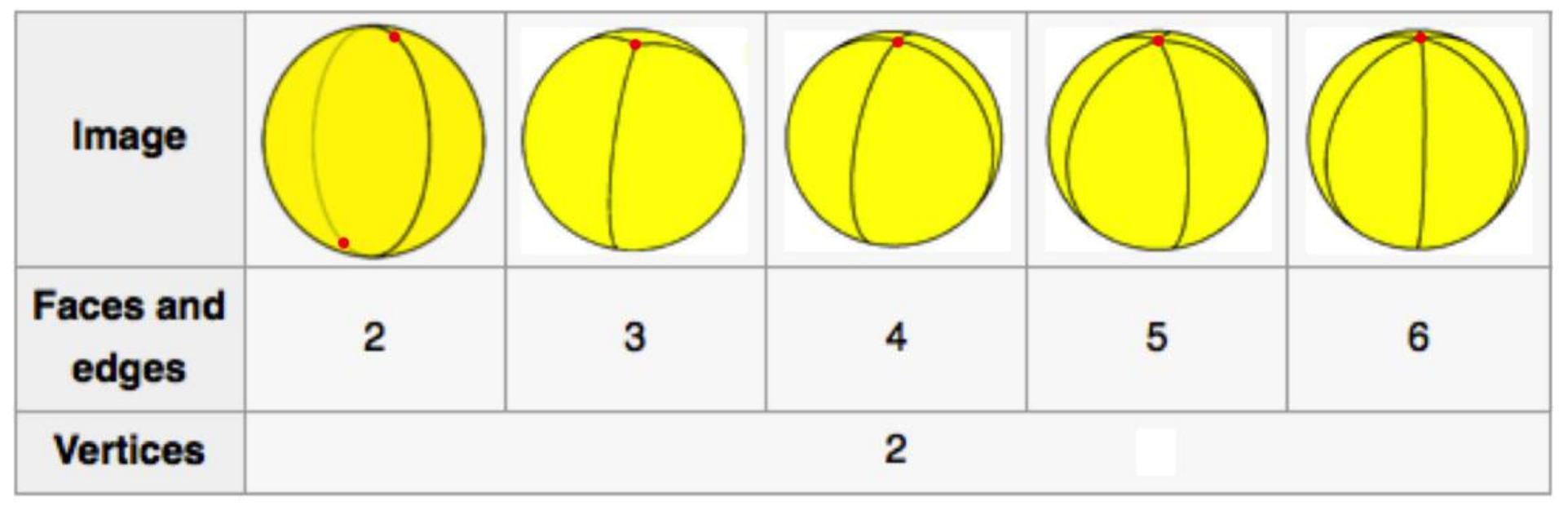}
\\
\includegraphics[width=.95\textwidth]{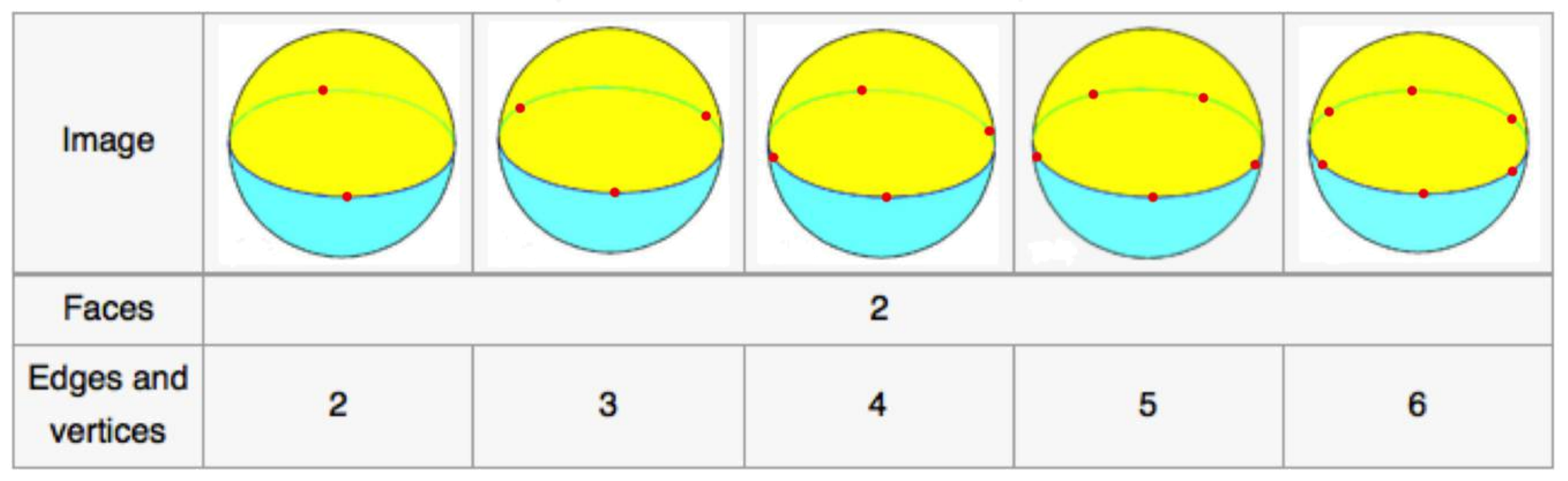}
\caption{
Hosohedra $\mathfrak{H}_{n}$, top figures, and dihedra $\mathfrak{D}_{n}$, bottom figures, are spherical 
polyhedra, duals of each other and with isotropy groups being the dihedral groups $\DD_{n}$.
}
\label{HosohedraDihedra}
\end{center}
\end{figure}

\begin{definition}
We shall say that the subset of spherical polyhedra comprised of the platonic polyhedra embedded in $\CW$, 
the $n$--gonal hosohedra and the $n$--gonal dihedra is the set of \emph{M\"obius polyhedra}. 
\end{definition}

\begin{remark}
Note that all M\"obius polyhedra can be obtained as the image of a conformal embedding 
$H:A\longrightarrow\A\subset\CW$ with $A$ being a platonic polyhedra, $\mathfrak{H}_{n}$ or 
$\mathfrak{D}_{n}$.
\end{remark}

\begin{definition}\label{antipodas}
Let $A$ be a platonic polyhedra, $\mathfrak{H}_{n}$ or $\mathfrak{D}_{n}$.
Let $H:A\longrightarrow\A\subset\CW$ be a conformal embedding.
\\
Let $p\in\A\subset\CW$, then $\widehat{p}\in\A\subset\CW$ is the \emph{antipode (in $\A$) of $p$} if 
$H^{-1}(\widehat{p})\in A$ is the antipode of $H^{-1}(p)\in A$.
\end{definition}
\begin{remark}
For the platonic polyhedra the antipode is as usual, in the case of $\mathfrak{H}_{n}$ and $\mathfrak{D}_{n}$
in $\CC$ define the antipode as the image via $z\mapsto -1/z$, for $z\neq0$, the antipode of $z=0$ is clear 
since  $H(\CC)=\CW\backslash\{point\}$.
\end{remark}
\begin{remark}
Center of a face, center of an edge, for a M\"obius polyhedra, are defined similarly.
\end{remark}

Note that if $\A$ is a spherical polyhedra and $H:\A\longrightarrow\CW$ is a conformal embedding, 
then $H\in\PSL$.
Given two conformal embeddings $\A_{1}$ and $\A_{2}$ of a M\"obius polyhedra $A$, 
the actual conformal embeddings will not be relevant (since $\A_{1}=T(\A_{2})$, for 
$T=H_{1}\circ H_{2}^{-1}\in\PSL$), hence, when not explicitly needed, we shall omit the reference of the 
conformal embedding.

Given a M\"obius polyhedra $\A$ 
denote by $V(\A)$ the vertices of $\A$, $E(\A)$ the centers of the edges of $\A$, and by 
$F(\A)$ the centers of the faces of $\A$. The cardinalities of these sets will be denoted by:
\\
$v:=Card(V(\A))$, $e:=Card(E(\A))$ and $f:=Card(F((\A))$. 
 
\begin{proposition}\label{encajeconforme}
Let $\A_{1}, \A_{2}$ be two M\"obius polyhedra with isotropy group $G$. 
Let $H_{1}:A\longrightarrow\A_{1}\subset\CW$ and $H_{2}:A\longrightarrow\A_{2}\subset\CW$ be two 
conformal embeddings of $A$ in the Riemann sphere $\CW$.
Let $\Z$ and $\Po$ be two subsets of $A$.
Let $\eta_{j}$, for $j=1,2$, be 1--forms with zeros in $H_{j}(\Z)$ and poles in $H_{j}(\Po)$.
\\
Then $\eta_{1}$ has isotropy group isomorphic to $G$ if and only if $\eta_{2}$ has isotropy group isomorphic 
to $G$.
Moreover, there exists a 
$T\in\PSL$ and $\lambda\in\CC^{*}$ such that $\eta_2=\lambda T_{*}\eta_1$.
\end{proposition}
\begin{proof}
Consider $T=H_{1}\circ H_{2}^{-1}$ in Lemma \ref{pushForma}.
\end{proof}

\begin{remark}
As it turns out, the case of the cyclic group $\ZZ_{n}$ for $n\geq 2$ will be different.
In fact all the non--trivial finite groups $G\subset \PSL$, with the exception of the cyclic groups, have a corresponding 
M\"obius polyhedra with isotropy group $G$.
What follows will apply for all non--trivial finite groups $G\subset\PSL$ except $\ZZ_{n}$. 
The case of $\ZZ_{n}\subset\PSL$ will be treated in \S\ref{ciclico}.
\end{remark}

\begin{lemma}
\label{classFinitos}
There exist
\begin{enumerate}
\item[1)] realizations $G_{1}$ of $A_4$, $G_{2}, G_{3}$ of $S_4$ and $G_{4}, G_{5}$ of $A_5$,
$G_{6}, G_{7}$ of $\DD_{ n}$ as subgroups of $\PSL$, 
and
\item[2)] embeddings of a regular tetrahedra $\A_{1}$, regular octahedra $\A_{2}$, cube $\A_{3}$, 
regular icosahedra $\A_{4}$, regular dodecahedra $\A_{5}$, the dihedron $\A_{6}=\mathfrak{D}_{n}$, 
and the hosohedron $\A_{7}=\mathfrak{H}_{n}$,
\end{enumerate}
such that the isotropy group of $\A_{i}$ is $G_{i}$ for $i=1,\ldots,7$.
\end{lemma}
\noindent
\emph{Proof.} We present explicit examples for the pairs $(G_{i}, \A_{i})$, for $i=1,\ldots,7$.
\begin{description}
\item[$(G_{1}, \A_{1})$]
We embed a tetrahedron $\A_{1}$ in the Riemann sphere $\CW$ in such way that  
$$\left\{\frac{1}{\sqrt{2}},\ \frac{e^{i\frac{2\pi}{3}}}{\sqrt{2}},\ \frac{e^{i\frac{4\pi}{3}}}{\sqrt{2}},\ \infty\right\}$$ 
are its vertices. The six edges are segments of circles (great arcs) on $\CW$ from each vertex to the three 
adjacent vertices. And the four faces are the (open) triangles formed by removing the vertices and edges from 
$\CW$.

\noindent
The transformations 
\begin{equation}
T_1(z)=e^{i\frac{2\pi}{3}} z \quad \text{ and }\quad
T_2(z)=\frac{(\sqrt{2} +i \sqrt{6}) z+2+2 i \sqrt{3}}{2 \sqrt{2}-4 z},
\end{equation}
generate the tetrahedron's isometry\footnote{
Since $G_{1}$ is the isometry group of the tetrahedron $\A_{1}$ then it also is the isotropy group of the 
tetrahedron. The same is true for the other cases.
} 
group $G_{1}$ which is isomorphic to $A_{4}$.
\\
The orbit of $\frac{1}{\sqrt{2}}$ under $T_2$ is 
$\left\{\frac{1}{\sqrt{2}}, \infty, \frac{e^{i\frac{4\pi}{3}}}{\sqrt{2}}\right\}$.
\\
The triangle formed by the vertices 
$\left\{\frac{1}{\sqrt{2}},\: \frac{e^{i\frac{2\pi}{3}}}{\sqrt{2}},\: \frac{e^{i\frac{4\pi}{3}}}{\sqrt{2}}\right\}$ 
is a face and its center is $0$.
The midpoint of the edge with vertices $\frac{1}{\sqrt{2}}$ and $\frac{e^{i\frac{2\pi}{3}}}{\sqrt{2}}$ is 
$b=\frac{\sqrt{6}}{3+\sqrt{3}}e^{\frac{i \pi}{3}}$. 
\\
The orbit of $0$ (under the whole group) is the set of centers of the four faces, and the orbit of $b$ 
is the set of the midpoints of the six edges.

\item[$(G_{2}, \A_{2})$]
The origin, the fourth roots of unity and $\infty\in\CW$ are the six vertices of an octahedron $\A_{2}$. 
The twelve edges of $\A_{2}$ are the segments of circle $(0,1)$, $(0,-1)$, $(0,i)$, $(0,-i)$, $(1,\infty)$, 
$(-1,\infty)$, $(i,\infty)$, $(-i,\infty)$, and the segments on the unit circle between $1$, $i$, $-1$ and $-i$. The 
eight faces are the (open) triangles formed by removing the vertices and edges from $\CW$.
The isometry group $G_{2}$ of $\A_{2}$ is generated by 
$$T_3(z)=i z \quad \text{and} \quad T_4(z)=\frac{z+1}{-z+1}$$ 
and is isomorphic to $S_{4}$.

\item[$(G_{3}, \A_{3})$] 
For $\A_{3}$ consider the dual of $\A_{2}$. $G_{3}=G_{2}$.

\item[$(G_{4}, \A_{4})$]
Let 
\begin{multline}
T_5(z)=e^{2 \pi i/5} z   \quad \text{and} \\
T_6(z)=\frac{(\sqrt{5}+1) z- 2 e^{i\frac{2\pi}{5}}}
{\left(1-e^{i\frac{2\pi}{5}}+e^{i\frac{4\pi}{5}}\right) \left(3+\sqrt{5}\right) z-e^{i\frac{4\pi}{5}} \left(1+\sqrt{5}\right)}.
\end{multline}
Then $G_{4}$ generated by $T_5$ and $T_6$ is the isometry group of the icosahedra $\A_{4}$ whose twelve 
vertices are the orbit of $0$. The thirty edges of $\A_{4}$ are the orbit under $G_{4}$ of the segment 
$\overline{0\ T_{6}(0)}$. The twenty faces are as usual obtained by removing the vertices and edges from 
$\CW$. Finally note that $G_{4}\cong A_{5}$.
\item[$(G_{5}, \A_{5})$] 
For $\A_{5}$ consider the dual of $\A_{4}$. $G_{5}=G_{4}$.
\item[$(G_{6}, \A_{6})$] 
In this case the M\"obius polyhedra is the dihedron $\A_{6}=\mathfrak{D}_{n}$. The vertices are the $n$--th 
roots of unity, the respective segments of the unit circle are the edges and the faces are the upper and lower 
hemispheres. The isometry group of the dihedron $\A_{6}$ is generated by 
$$T_7(z)=e^{2 \pi i/n} z\quad \text{and}\quad T_8(z)=\frac{1}{z}$$ 
and is isomorphic to $\DD_n$.
\item[$(G_{7}, \A_{7})$] 
For $\A_{7}$ consider the dual of $\A_{6}$. $G_{7}=G_{6}$.\hfill\qed
\end{description}

\begin{proposition}\label{embedding}
Let $G$ be a finite subgroup of $\PSL$. If $G$ is isomorphic to $A_4$, $S_4$, $A_5$, or $\DD_{ n}$ then 
there exists an embedding, in the usual Riemann sphere \CW, of the tetrahedra, octahedra/cube, 
icosahedra/dodecahedra, or dihedron/hosohedron respectively, whose isotropy group is $G$.
\end{proposition}
\begin{proof}
From Klein's classical result on the classification of finite subgroups of $\PSL$, 
there exists $T\in\PSL$ such that $G=TG_{k} T^{-1}$, for some $k\in\{1,\ldots,7\}$, 
where $G_{k}$ is as in Lemma \ref{classFinitos}. 
Clearly $G$ fixes $T(\A_{k})$, and since $T$ is an isometry of the Riemann sphere $\CW$, then $T(\A_{k})$ is 
the sought after embedding.
\end{proof}

\begin{lemma}\label{ptosfijossonverticesaristascaras}
Let $\A$ be a M\"obius polyhedra with isotropy group $G$.

\noindent
Then 
\\ \centerline{
$\{$fixed points of non--trivial elements of $G\} = V(\A)\cup E(\A)\cup F(\A)$.}

\end{lemma}
\begin{proof}
Let $x\in\CW$ a fixed point for a non--trivial $g\in G$. Since $g$ is elliptic there is another fixed point of $g$, 
namely $y\in\CW$, $x\neq y$. This pair $\{x,y\}$ defines a symmetry axis. 
On the other hand since $\A$ is a M\"obius polyhedra with $G$ as its isometry group, then 
through each element of $V(\A)\cup E(\A)\cup F(\A)$ there is a symmetry axis going through it, 
hence $\{x,y\}\subset V(\A)\cup E(\A)\cup F(\A)$.

\noindent
Now let $q\in V(\A)\cup E(\A)\cup F(\A)$, since through each element of $V(\A)\cup E(\A)\cup F(\A)$ there is a symmetry axis going through it, it follows that there is a non--trivial $g\in G$ with fixed point $q$.
\end{proof}

\begin{lemma}\label{VerticesAristasCaras}
Let $\A$ be a M\"obius polyhedra with isotropy group $G$, and let $\eta$ be a $G$--invariant rational 1--form.
Then $V(\A)\cup E(\A)\cup F(\A)\subset \Po_\eta\cup\Z_\eta$.
\end{lemma}
\begin{proof}
By Lemma \ref{ptosfijossonverticesaristascaras}, given $x\in V(\A)\cup E(\A)\cup F(\A)$ there exists non--trivial $g\in G$ with fixed points $\{x,y\}$, for some $y\in\CW$. 
By Lemma \ref{afirmacion0} $g$ is an elliptic element and by Theorem \ref{caracterizacion1}.3 the fixed points $\{x,y\}\in\Po_\eta\cup\Z_\eta$. 
\end{proof}

\begin{definition}\label{defRA}
1. The \emph{fundamental region for $\A$} denoted by $\mathcal{R}_{\A}$ 
will be 
the interior of the triangle formed by the center of a face and the two vertices of an edge of the same face; 
half of the interior of said edge; 
the segment that joins one of the vertices of the edge to the center of the face; 
one of the two vertices of the edge and the center of the face (see figure \ref{regionfundamental}). 

\noindent
2. Let $\widehat{\mathcal{R}}_{\A}=\mathcal{R}_{\A} \backslash\{V(\A) \cup E(\A) \cup F(\A)\}$, we shall call 
this is a \emph{quasi--fundamental region of $\A$}.

\noindent
3. A \emph{fundamental region} $\mathcal{R}_{G}$ for the action of the group $G$, is a maximal connected 
region on $\CW$ such that for $a\in\mathcal{R}_{G}$ the orbit $\mathcal{O}(a)$ of $a$ only has one element 
in $\mathcal{R}_{G}$, that is $\mathcal{O}(a)\cap\mathcal{R}_{G}=\{a\}$.
\end{definition}

\begin{figure}[htbp]
\begin{center}
\includegraphics[width=\textwidth]{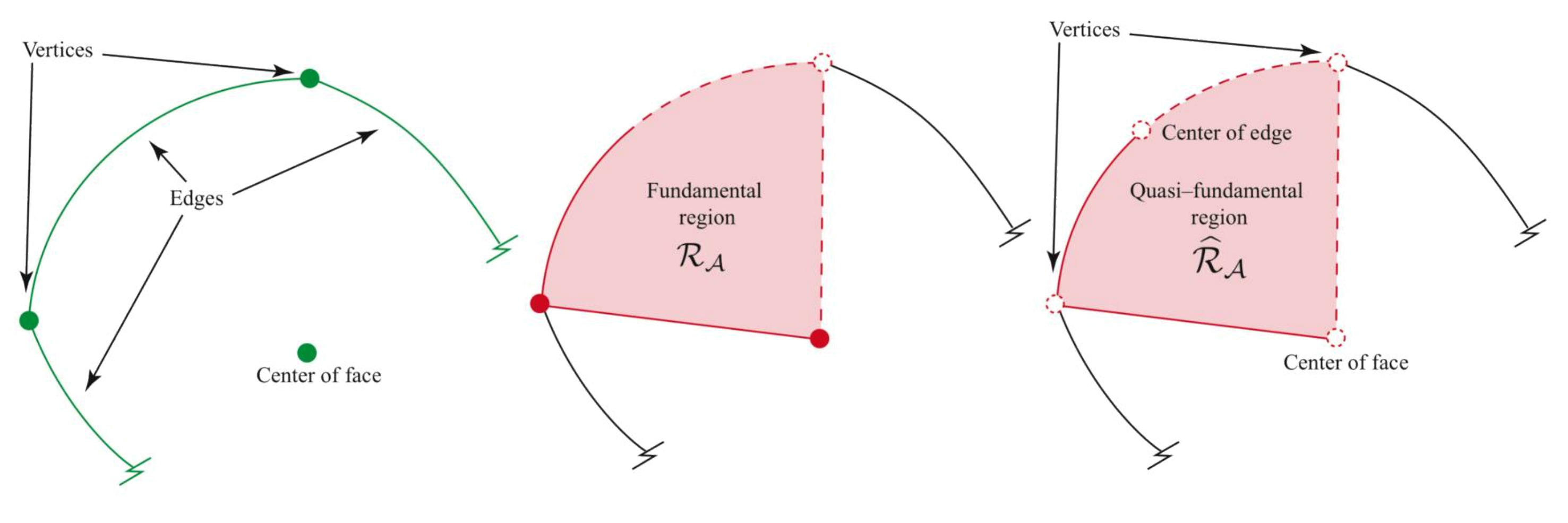}
\caption{\emph{Fundamental and quasi--fundamental regions}. In (a) we have 1 complete edge, a center of 
face and the two vertices of the edge.
In (b) the construction of the fundamental region for $\A$ is exemplified: the fundamental region 
$\mathcal{R}_{\A}$ is 
the interior of the triangle formed by the center of a face and the two vertices of an edge of the same face; 
half of the interior of said edge; 
the segment that joins one of the vertices of the edge to the center of the face; 
one of the two vertices of the edge and the center of the face.
In (c) the quasi--fundamental region 
$\widehat{\mathcal{R}}_{\A}=\mathcal{R}_{\A}\backslash\{V(\A) \cup E(\A) \cup F(\A)\}$ is shown.
}
\label{regionfundamental}
\end{center}
\end{figure}

\begin{remark}
It follows that if $\A$ is a M\"obius polyhedra and $G$ leaves invariant $\A$, then $\mathcal{R}_{\A}$ is a 
fundamental region for the action of $G$. In other words $\mathcal{R}_{\A}$ is one of many possible 
$\mathcal{R}_{G}$.
\end{remark}

\begin{lemma}\label{ordenRegionQuasi}
Let $\A$ be a M\"obius polyhedra and $G$ the isotropy group of $\A$, then for 
$a\in\widehat{\mathcal{R}}_{\A}$, $Card(\mathcal{O}(a))=Card(G)$.
\end{lemma}
\begin{proof}
Let $x\in \A$ such that $Card(\mathcal{O}(x))\neq Card(G)$, then by Lemma \ref{afirmacion0} there exists 
non--trivial $g\in G$ elliptic and $x$ is a fixed point of $g$. 
By Lemma \ref{ptosfijossonverticesaristascaras} 
it follows that if 
$$a\in\CW_{\A}:=\CW\backslash\big( V(\A)\cup E(\A) \cup F(\A)\big)$$ 
then $Card(\mathcal{O}(a))=Card(G)$.
\end{proof}

The table on page 18 of \cite{Toth}, shows the order of the subgroups $S$ that leave invariant 
$V(\A)$, $E(\A)$ and $F(\A)$.
An appropriate interpretation of the aforementioned table (or a straightforward counting argument) leads to our 
Table \ref{tablaBuena}, which will be useful in what follows.
\begin{table}[htp]
\caption{Cardinality of $V(\A)$, $E(\A)$ and $F(\A)$ for the M\"obius polyhedra $\A$ associated to the finite 
isotropy groups $G\subset\PSL$ (excluding $\ZZ_{n}$).
Recall that $v=Card(V(\A))$, $e=Card(E(\A))$ and $f=Card(F((\A))$.}
\begin{center}
\begin{tabular}{|c|c|c|c|c|c|c|c|}
\hline
group $G$  & $A_{4}$ & $S_{4}$ & $S_{4}$ & $A_{5}$ & $A_{5}$ & $\DD_{ n}$ & $\DD_{ n}$ \\
\hline
M\"obius  & Tetra- & Cube & Octa- & Icosa- & Dodeca- & Di- & Hoso- \\
polyhedra  & hedron & & hedron & hedron & hedron & hedron & hedron \\
 \hline
 \hline
 $v$
 & 4 & 8 & 6 & 12 & 20 & n & 2\\
 \hline
 $e$ 
 & 6 & 12 & 12 & 30 & 30 & n & n\\
 \hline
 $f$
  & 4 & 6 & 8 & 20 & 12 & 2 & n\\
 \hline
 $Card(G)$ & 12 & 24 & 24 & 60 & 60 & 2n & 2n \\
 \hline

\end{tabular}
\end{center}
\label{tablaBuena}
\end{table}%

\medskip
As mentioned before, condition (3) of Theorem \ref{caracterizacion1} is automatically satisfied for 
$G\cong A_{5}$. This is the content of the next result.
\begin{proposition}[Characterization of rational 1--forms with isotropy $A_{5}$]\label{caracterizacionA5}
\hfill\\
Let $G$ be a finite subgroup of $\PSL$ isomorphic to $A_{5}$ and  let $\eta$ be a $1$-form with simple poles 
and zeros. 
\\
The 1--form $\eta$ 
has isotropy group $G$
if and only if the following two conditions are met:
\begin{enumerate}
\item[1)] $\Po_{\eta}$ is $G$--invariant.
\item[2)] $\Z_{\eta}$ is $G$--invariant.
\end{enumerate}
\end{proposition}
\begin{proof}
($\Rightarrow$) This is immediate.

\medskip
\noindent
($\Leftarrow$) Since $G\cong A_{5}$, there is a dodecahedron $\A\subset\CW$ whose isotropy group is $G$.

\noindent
Let 
$l_1\in\{0,1\}$ be the number of poles on $f\in F(\A)$.

\noindent
Let
$k_1\in\{0,1\}$ be the number of zeros on $f\in F(\A)$.

\noindent
Since $F(\A)$ is an orbit of $G$, then $l_1+k_1\in \{0,1\}$.

\smallskip
\noindent
Let
$l_2\in\{0,1\}$ be the number of poles on $e\in E(\A)$.

\noindent
$k_2\in\{0,1\}$ be the number of zeros on $e\in E(\A)$.

\noindent
Since $E(\A)$ is an orbit of $G$, thus $l_2+k_2\in \{0,1\}$.

\smallskip
\noindent
Let
$l_3\in\{0,1\}$ be the number of poles on $v\in V(\A)$.

\noindent
Let
$k_3\in\{0,1\}$ be the number of zeros on $v\in V(\A)$.

\noindent
Once again, since $V(\A)$ is an orbit of $G$, then $l_3+k_3\in \{0,1\}$.

\smallskip
\noindent
Let
$l_4$ be the number of poles in the quasi--fundamental region $\widehat{\R}_{\A}$.

\noindent
Let
$k_4$ be the number of zeros in the quasi--fundamental region $\widehat{\R}_{\A}$.

\noindent
In this case, since $a\in\widehat{\R}_{\A}$ satisfies $\mathcal{O}(a)\cap\mathcal{R}_{G}=\{a\}$, then $l_4,k_4\in\NN\cup\{0\}$.

\smallskip
By Gauss--Bonet, and/or observing Table \ref{tablaBuena}, we have:

\begin{equation}\label{GBparaDode}
12 l_1+30 l_2+ 20 l_3+ 60 l_4-12 k_1-30 k_2-20 k_3-60 k_4=2.
\end{equation}

\noindent 
Hence it follows that
$$5|(-12 l_1+12 k_1+ 2),$$
which implies that $l_1=1$ and $k_1=0$.
Therefore, upon substitution into equation \eqref{GBparaDode} and dividing by 10 we obtain:
\begin{equation}\label{GB2}
3 l_2+ 2 l_3+6 l_4-3 k_2- 2 k_3- 6 k_4=-1,
\end{equation}
so it follows that
$3|(-1-2 l_3+2 k_3)$
and hence $l_3=1$ and $k_3=0$.

\noindent
Upon substitution in equation \eqref{GB2} we have
$$l_2+ 2 l_4-k_2- 2 k_4=-1,$$
so $2|(-l_2+k_2-1)$ and we obtain two cases: 
\begin{enumerate}
\item[a)] $l_2=1$ and $k_2=0$ or
\item[b)] $l_2=0$ and $k_2=1$.
\end{enumerate}

\noindent
Summarizing we have:
$F(\A)\cup V(\A)\subset\Po_{\eta}$ and either
\begin{enumerate}
\item[a)] $E(\A)\subset\Po_{\eta}$ or
\item[b)] $E(\A)\subset\Z_{\eta}$.
\end{enumerate}
In any case, condition (3) of Theorem \ref{caracterizacion1} is true.
So $\eta$ is $A_5$ invariant.

\smallskip
\noindent
Finally note that the minimality condition of Theorem \ref{caracterizacion1} is automatically met since 
there are no finite subgroups of $\PSL$ that contain as a proper subgroup a group isomorphic to $A_{5}$.
\end{proof}

\begin{definition}
We will say that $G<\PSL$ is a \emph{platonic} subgroup if it is a finite subgroup not isomorphic to a cyclic or a 
dihedric ({\it i.e.} it is isomorphic to $A_{4}$, $S_{4}$ or $A_{5}$).
\end{definition}

The following result classifies the rational 1--forms with simple poles (and zeros) whose isotropy groups are 
platonic.

\begin{theorem}[Classification of 1--forms with simple poles and zeros having isotropy a platonic subgroup]
\label{main2}
Let $G<\PSL$ be a platonic subgroup. 
Let $\eta$ be a 1--form with simple poles and simple zeros. 
\\
Then the 1--form $\eta$, with $k$ poles and $k-2$ zeros, 
is  $G$--invariant 
if and only if 
there is a platonic polyhedra $\A$ conformally embedded in $\CW$ with isotropy group $G$ such that 
\begin{enumerate}
\item[$\bullet$] $V(\A)\cup F(\A)\subset\Po_{\eta}$, 

\item[$\bullet$] 
either
\begin{enumerate}
\item 
$E(\A)\subset\Z_{\eta}$: In which case there are $\ell$ poles and $\ell$ zeros in the 
quasi--fundamental region $\widehat{\mathcal{R}}_{\A}$, 
for some non negative $\ell$ satisfying $$k=\ell\times Card(G)+v+f, \text{ or}$$

\item $E(\A)\subset\Po_{\eta}$: In which case there are $\ell$ poles and $\ell+1$ zeros in the 
quasi--fundamental region $\widehat{\mathcal{R}}_{\A}$,  
for some non negative $\ell$ satisfying $$k=\ell\times Card(G)+v+f+e.$$ 
\end{enumerate}
\end{enumerate}
Moreover $G$ is the maximal group, as a subgroup of $\PSL$, satisfying the above conditions if and only if 
$\eta$ has isotropy group $G$.\\
In the case that $G$ is isomorphic to $A_5$ or $S_4$, the maximality condition is automatically satisfied.
\end{theorem}
\begin{proof}
($\Rightarrow$) By Proposition \ref{embedding} there exists a spherical polyhedra $\A$ such that 
$G$ is its isotropy group.

\noindent
By Lemma \ref{VerticesAristasCaras}, $V(\A)\cup E(\A)\cup F(\A)\subset \Po_\eta\cup\Z_\eta$.

\noindent
Since $\A$ is a platonic polyhedra conformally embedded in $\CW$, 
the only fixed points of $\A$ of order 2 are on $E(\A)$. 
Hence by Theorem \ref{caracterizacion1}.3.b, $V(\A)\cup F(\A)\subset\Po_\eta$.

\noindent
Since $E(\A)=\mathcal{O}(e)$ for $e\in E(\A)$ then $E(\A)$ is either entirely contained in $\Z_\eta$ or entirely 
contained in $\Po_\eta$ which give rise to conditions (a) and (b) respectively.

\noindent
To finish the proof we need to examine how many zeros and poles are in the quasi--fundamental region.

\noindent
By Lemma \ref{ordenRegionQuasi} if $a\in\widehat{\mathcal{R}}_{\A}$ then $Card(\mathcal{O}(a))=Card(G)$. 
Hence the corresponding formula for the number $k$ of poles follows immediately.

\smallskip
\noindent
($\Leftarrow$)
Assuming $V(\A)\cup F(\A)\subset\Po_{\eta}$ and either (a) or (b) above, the conditions (1)--(3) of 
Theorem \ref{caracterizacion1} are satisfied. Hence the 1--form $\eta$ is $G$--invariant.
\end{proof}

\begin{remark} When constructing the 1--form the following choices are to be made: 
\begin{enumerate}
\item Either (a) or (b) can occur (but not both). This choice determines the non negative integer $\ell$, 
that satisfies the corresponding relation with the number of poles $k$ of $\eta$.
\item The placement of the $\ell$ poles (and the corresponding zeros) inside 
$\widehat{\mathcal{R}}_{\A}$ is arbitrary, each one giving rise to a $G$--invariant 1--form $\eta$.
\item Case (a) with $\ell=0$ corresponds to the examples in \S\ref{ejemplos}. 
\end{enumerate}
\end{remark}

Since the dihedron is the dual of the hosohedron, the following theorems for the dihedric case will be stated 
for the dihedron $\mathfrak{D}_{n}$, leaving the case of the dual $\mathfrak{H}_{n}$ for the interested reader.
\begin{theorem}[Classification of 1--forms with simple poles and zeros having isotropy a dihedric subgroup]
\label{teoremaDihedrico}
Let $G<\PSL$ be a subgroup isomorphic to $\DD_{ n}$ with $n\geq 3$. 
Let $\eta$ be a 1--form with simple poles and simple zeros. 
\\
Then the 1--form $\eta$, with $k$ poles and $k-2$ zeros, 
is  $G$--invariant 
if and only if 
there is a dihedron $\A=\mathfrak{D}_{n}$ with isotropy group $G$  
such that one of the following cases is true

\begin{enumerate}

\item[A)]
	\begin{enumerate}
	\item[$\bullet$] $V(\A)\cup F(\A)\subset\Po_{\eta}$, 

	\item[$\bullet$] 
	either
		\begin{enumerate}
		\item[a)]
		$E(\A)\subset\Z_{\eta}$: In which case there are $\ell$ poles and $\ell$ zeros in the quasi--
		fundamental region $\widehat{\mathcal{R}}_{\A}$, 
		for some non negative $\ell$ satisfying $$k=\ell\times Card(G)+v+f, \text{ or}$$

		\item[b)] $E(\A)\subset\Po_{\eta}$: In which case there are $\ell$ poles and $\ell+1$ zeros in the quasi--
		fundamental region $\widehat{\mathcal{R}}_{\A}$,  
		for some non negative $\ell$ satisfying $$k=\ell\times Card(G)+v+f+e.$$ 
	\end{enumerate}
\end{enumerate}

\item[B)]
		\begin{enumerate}
		\item[$\bullet$] $V(\A)\cup E(\A)\subset\Z_{\eta}$ and $F(\A)\subset\Po_{\eta}$

		\item[$\bullet$] There are $\ell$ poles and $\ell-1$ zeros in the quasi--fundamental region 
		$\widehat{\mathcal{R}}_{\A}$,  
		for some non negative $\ell$ satisfying $$k=\ell\times Card(G)+f.$$
	\end{enumerate}
\end{enumerate}
Moreover $G$ is the maximal group, as a subgroup of $\PSL$, satisfying either (A) or (B) if and only if $\eta$ 
has isotropy group $G$.
\end{theorem}

\begin{proof}
($\Rightarrow$) By Proposition \ref{embedding} there is a dihedron $\A=\mathfrak{D}_{n}$ such that $G$ is its 
isotropy group.

\noindent
By Lemma \ref{VerticesAristasCaras}, $V(\A)\cup E(\A)\cup F(\A)\subset \Po_\eta\cup\Z_\eta$.

\noindent
Since $\A$ is a dihedron, $F(\A)=\{x,y\}$ are the fixed points of the order $n$ elements in 
$\DD_{ n}$. 
Thus, since $n\geq 3$, Theorem \ref{caracterizacion1} requires that $F(\A)\subset\Po_\eta$. 

\noindent
Without loss of generality we can assume that the dihedron $\A$ has $F(\A)=\{0,\infty\}\subset\CW$, hence in 
particular the order $n$ elements of $G$ will be rotations by $2\pi/n$.

\noindent
If $V(\A)\subset\Po_\eta$ we have condition $V(\A)\cup F(\A)\subset\Po_{\eta}$. In which case either 
$E(\A)\subset\Z_{\eta}$ or 
$E(\A)\subset\Po_{\eta}$ that is conditions (A.a) and (A.b) respectively.

\noindent
If $V(\A)\subset\Z_\eta$ we have two cases:
either $E(\A)\subset\Z_{\eta}$ giving rise to condition (B), or $E(\A)\subset\Po_{\eta}$ which is equivalent to 
conditions (A.a) with a different dihedron $\A'$ 
which can be obtained from the original $\A$ by rotating by an angle of $\pi/n$, 
around the fixed points $F(\A)$.

\noindent
To finish the proof we need to examine how many zeros and poles are in the quasi--fundamental region.

\noindent
By Lemma \ref{ordenRegionQuasi} if $a\in\widehat{\mathcal{R}}_{\A}$ then $Card(\mathcal{O}(a))=Card(G)$. Hence the corresponding formula for the number $k$ of poles follows immediately for each case.

\smallskip
\noindent
($\Leftarrow$)
Once again, 
given any of the corresponding cases of Theorem \ref{teoremaDihedrico}, the conditions (1)--(3) of Theorem \ref{caracterizacion1} are satisfied. Hence the 1--form $\eta$ is $G$--invariant.
\end{proof}

\begin{theorem}[Case for $\DD_{ 2}$]\label{D22}\hfill\\
Let $G<\PSL$ be a subgroup isomorphic to $\DD_{ 2}$.
Let $\eta$ be a 1--form with simple poles and simple zeros.
\\
Then the 1--form $\eta$, with $k$ poles and $k-2$ zeros, is $G$--invariant if and only if there is a dihedron 
$\mathfrak{D}_{2}$ such that
\begin{enumerate}

\item[A)]
	\begin{enumerate}
	\item[$\bullet$] $V(\A)\cup F(\A)\subset\Po_{\eta}$, 

	\item[$\bullet$] 
	either
		\begin{enumerate}
		\item[a)] 
		$E(\A)\subset\Z_{\eta}$: In which case there are $\ell$ poles and $\ell$ zeros in the 
		quasi--fundamental region $\widehat{\mathcal{R}}_{\A}$, 
		for some non negative $\ell$ satisfying $$k=\ell\times Card(G)+v+f, \text{ or}$$

		\item[b)] $E(\A)\subset\Po_{\eta}$: In which case there are $\ell$ poles and $\ell+1$ zeros in the quasi--
		fundamental region $\widehat{\mathcal{R}}_{\A}$,  
		for some non negative $\ell$ satisfying $$k=\ell\times Card(G)+v+f+e.$$ 
		\end{enumerate}
	\end{enumerate}

\item[B)] 
	\begin{enumerate}
		\item[$\bullet$] $V(\A)\cup E(\A)\subset\Z_{\eta}$ and $F(\A)\subset\Po_{\eta}$

		\item[$\bullet$] There are $\ell$ poles and $\ell-1$ zeros in the quasi--fundamental region 
		$\widehat{\mathcal{R}}_{\A}$,  
		for some non negative $\ell$ satisfying $$k=\ell\times Card(G)+f.$$
	\end{enumerate}
	
\item[C)] 
	\begin{enumerate}
		\item[$\bullet$] $V(\A)\cup E(\A)\cup F(\A)\subset\Z_{\eta}$, 

		\item[$\bullet$] There are $\ell$ poles and $\ell-2$ zeros in the quasi--fundamental region 
		$\widehat{\mathcal{R}}_{\A}$,
		for some non--negative $\ell$ satisfying $$k=(\ell)\times Card(G)=4\times\ell.$$
	\end{enumerate}

\end{enumerate}
Moreover $G$ is the maximal subgroup of $\PSL$ satisfying one of (A)--(C) if and only if $\eta$ has isotropy 
group $G$.
\end{theorem}
\begin{proof}
($\Rightarrow$)
By Proposition \ref{embedding} there is a dihedron $\A=\mathfrak{D}_{2}$ such that $G=\DD_{2}$ is its 
isotropy group.

\noindent
By Lemma \ref{VerticesAristasCaras}, $V(\A)\cup E(\A)\cup F(\A)\subset \Po_\eta\cup\Z_\eta$.

\noindent
However, since all the non--trivial elements of $G$ have order 2, {\it a--priori} there is no way to know which of 
the sets $V(\A)$, $E(\A)$, $F(\A)$ are subsets of $\Z_{\eta}$. Thus we have to consider all the possible cases:
\begin{itemize}
\item None of $V(\A)$, $E(\A)$, $F(\A)$ are subsets of $\Z_{\eta}$. This is case (A.b).

\item Only one of $V(\A)$, $E(\A)$, $F(\A)$ is a subset of $\Z_{\eta}$. Because of the high symmetry of the 
action of $G=\DD_{2}$ on $\A=\mathfrak{D}_{2}$ all 3 possible cases are the same, so we assume without 
loss of generality that $E(\A)\subset\Z_{\eta}$, this is case (A.a).

\item Exactly two of $V(\A)$, $E(\A)$, $F(\A)$ are subsets of $\Z_{\eta}$. Once again all 3 possible cases are 
the same so without loss of generality we assume that $V(\A) \cup E(\A) \subset \Z_{\eta}$, this is case (B).

\item $V(\A) \cup E(\A) \cup F(\A) \subset \Z_{\eta}$. This is case (C).
\end{itemize}

\noindent
The rest of the proof is as in the previous cases.
\end{proof}

\begin{example}\label{ejemploTetra}
Let 
$$\eta(z)
=\frac{ z^3-\frac{1}{\sqrt{8}} }{ z^6-\sqrt{50} z^3-1 }\, dz
.$$ 
The phase portrait of the vector field associated to $\eta$ can be seen in Figure \ref{TetraContra1}.
It can be readily seen that the poles and zeros are invariant under the isotropy group $G\cong A_{4}$ of a 
Tetrahedron $\A$, but $\eta$ is not invariant under $G$, see Figure \ref{TetraContra1}.a. 

\noindent
In fact, it's isotropy group is $\DD_{2}$ in accordance with Theorem \ref{D22} case (A.b) with 
$\A=\mathfrak{D}_{2}$ and $\ell=0$, see Figure \ref{TetraContra1}.b.
\end{example}

\begin{figure}[htbp]
\begin{center}
\includegraphics[height=0.45\textwidth]{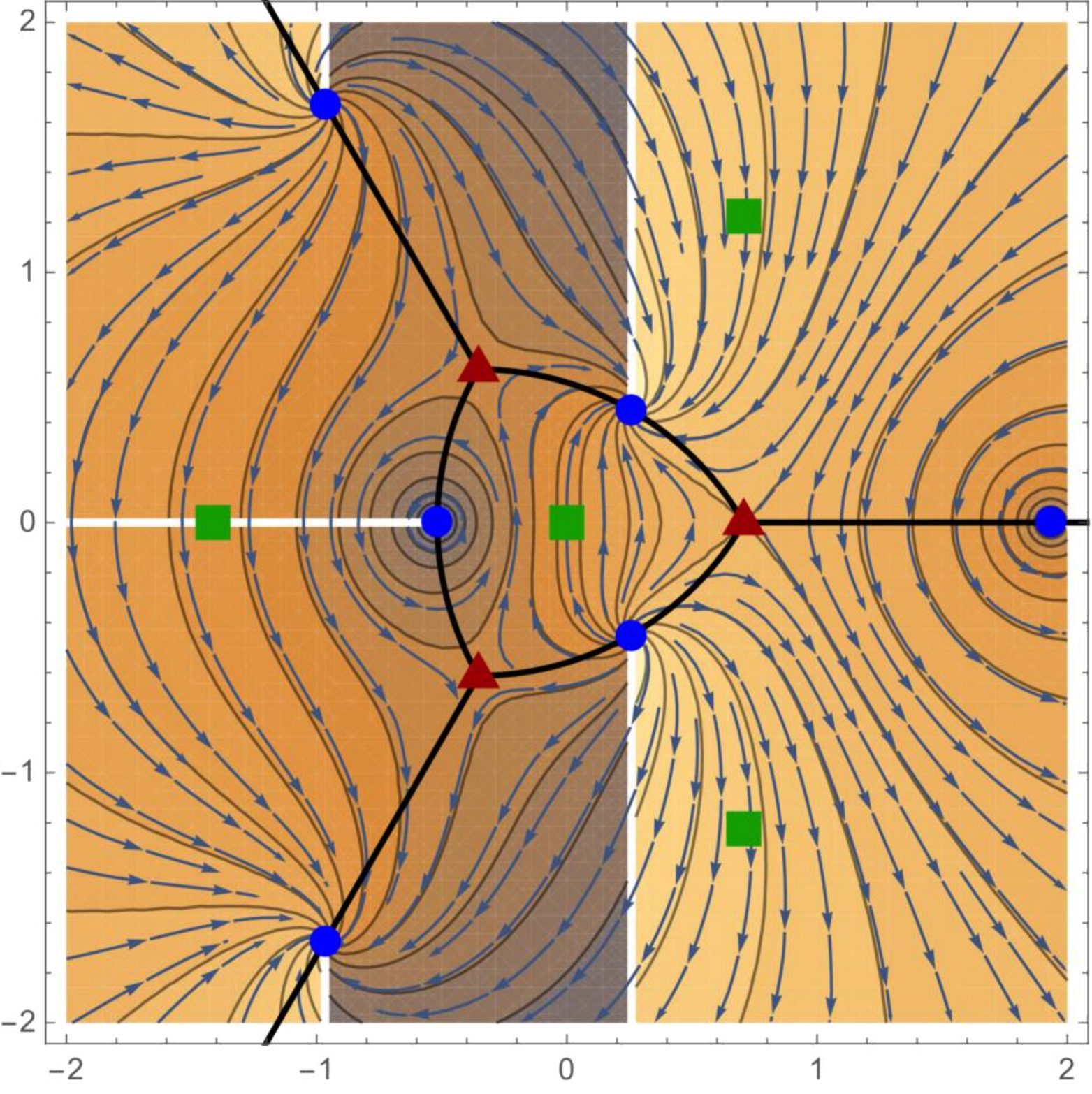}
\hskip 10pt
\includegraphics[height=0.45\textwidth]{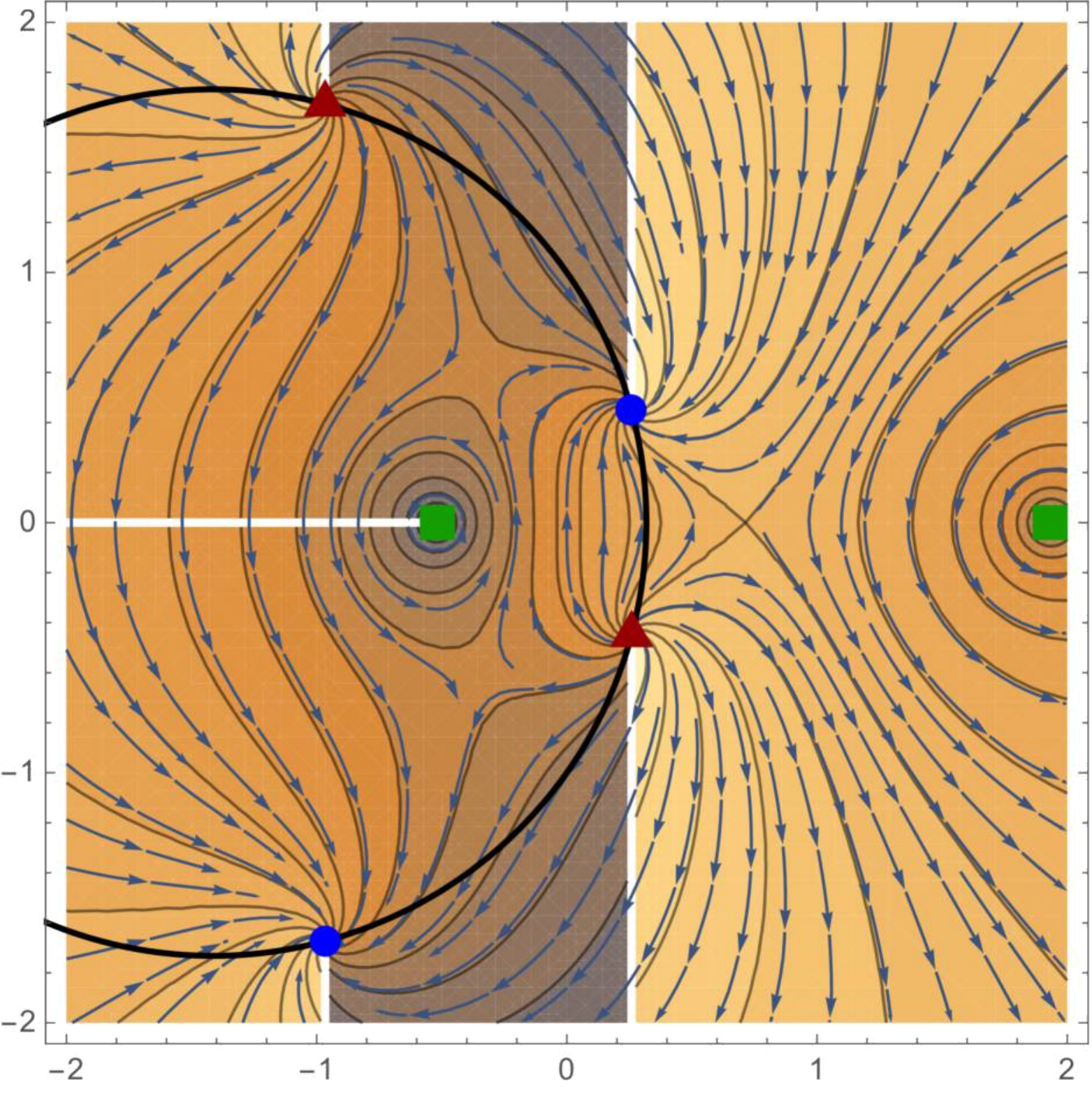}
\caption{Phase portrait of $\eta$ as in Example \ref{ejemploTetra}.
(a) Note that the poles and zeros are invariant under the isotropy group $G\cong A_{4}$ of a tetrahedron. 
(b) However the isotropy group of $\eta$ is in fact $\DD_{2}$. In both cases of the spherical polyhedra, vertices are represented by (red) triangles, centers of edges by (blue) dots and centers of faces by (green) squares.
}
\label{TetraContra1}
\end{center}
\end{figure}

\subsubsection{The case of $G$ isomorphic to the cyclic group $\ZZ_{n}$ for $n\geq 2$}\label{ciclico}
Since there is no spherical polyhedra $\A$ whose isotropy group is isomorphic to $\ZZ_{n}$ for $n\geq2$, we 
can not apply the techniques developed in the previous section to obtain a characterization of the 1--forms 
$\eta$ with isotropy groups isomorphic to $\ZZ_{n}$.

However, when $G\cong\ZZ_{n}$, with $n\geq2$, is a subgroup of $\PSL$, we can recall 
Definition \ref{defRA}.3 of a fundamental region $\R_{\ZZ_{n}}$ and define a \emph{quasi--fundamental region 
for $G$} as 

\centerline{
$\widehat{\R}_{\ZZ_n}=\R_{\ZZ_n}\backslash\{x,y\}$
}

\noindent
where $\{x,y\}\subset\CW$ are the fixed points of $G$ 
(if $g\in G\cong\ZZ_{n}$ is a generator, then $g$ is an order $n$ elliptic element that fixes $\{x,y\}\subset\CW$; 
in fact $\{x,y\}\subset\CW$ are the fixed points of $G$).

It will be useful to note that even though the hosohedra $\A=\mathfrak{H}_{n}$ is $G$--invariant with 
$G\cong\ZZ_{n}$ the fundamental and quasi--fundamental region of $\A$ do not agree with the fundamental 
and quasi--fundamental region of $G$. 
With this in mind we will use the hosohedra $\A=\mathfrak{H}_{n}$ and the quasi--fundamental region 
$\widehat{\R}_{\ZZ_n}$ of $G$ in the statements of the theorems in this section.

\smallskip
\noindent
For the case $n=2$ we have.
\begin{theorem}[Classification of 1--forms with simple poles and simple zeros having isotropy $G\cong\ZZ_2$]
\label{Z2}\hfill\\
Let $G<\PSL$ be a subgroup isomorphic to $\ZZ_{2}$. 
Let $\eta$ be a 1--form with simple poles and simple zeros. 

\noindent
Then the 1--form $\eta$, with $k$ poles and $k-2$ zeros, 
is  $G$--invariant 
if and only if 
there is a $G$--invariant hosohedra $\A=\mathfrak{H}_{2}$ such that
one of the following cases is true.
\begin{enumerate}
\item[A)]
	\begin{enumerate}
	\item[$\bullet$] $V(\A)\subset\Po_{\eta}$.
	\item[$\bullet$] There are $\ell$ poles and $\ell$ zeros in a quasi--fundamental region $\widehat{\R}_{\ZZ_{2}}$ for 
	some positive 
	$\ell\geq1$,
	satisfying 
	$$k=2\times\ell+2.$$ 
	\end{enumerate}
\item[B)]
	\begin{enumerate}
	\item[$\bullet$] $V(\A)\subset\Z_{\eta}$.
	\item[$\bullet$] There are $\ell$ poles and $\ell-2$ zeros in a quasi--fundamental region $\widehat{\R}_{\ZZ_{2}}$ 
	for some positive 
	$\ell\geq2$,
	satisfying 
	$$k=2\times\ell.$$ 
	\end{enumerate}
\item[C)]
	\begin{enumerate}
	\item[$\bullet$] $V(\A)=\{x,y\}$, $x\in\Po_{\eta}$, $y\in\Z_{\eta}$.
	\item[$\bullet$] There are $\ell$ poles and $\ell-1$ zeros in a quasi--fundamental region $\widehat{\R}_{\ZZ_{2}}$ 
	for some positive 
	$\ell\geq1$,
	satisfying 
	$$k=2\times\ell+1.$$ 
	\end{enumerate}
\end{enumerate}
Moreover $G$ is the maximal subgroup of $\PSL$ satisfying one of (A)--(C) if and only if $\eta$ has isotropy 
group $G$.
\end{theorem}
\begin{proof}
($\Rightarrow$)
The existence of the hosohedra $\A=\mathfrak{H}_{2}$ is assured by placing the vertices of $\A$ at the fixed 
points $\{x,y\}\subset\CW$ of $G$; 
for the edges consider a circle on $\CW$ containing the vertices $\{x,y\}$; 
the faces then are the complement, in $\CW$, of the vertices and the edges.

\noindent
From Theorem \ref{caracterizacion1}.3.a, conditions (A), (B) and (C) on the vertices $V(\A)$ follow, moreover 
since $V(\A)$ consists of exactly two points, these are the only possibilities for $V(\A)$.

\noindent
Finally by a direct application of Gauss--Bonnet the conditions on the quasi--fundamental regions 
$\widehat{\R}_{\ZZ_{2}}$ for (A), (B) and (C) follow immediately.

\smallskip
\noindent
($\Leftarrow$)
This implication is a direct consequence of Theorem \ref{caracterizacion1}.3.b, the action of $G$ on $\eta$, the 
action of $G$ on $\CW$ and the definition of $\widehat{\R}_{\ZZ_n}$.
\end{proof}

\smallskip
The case of $G\cong\ZZ_{n}$ with $n\geq 3$ now follows immediately.
\begin{theorem}[Classification of 1--forms with simple poles and simple zeros having isotropy $G\cong\ZZ_{n}$ 
with $n\geq 3$]\label{main3b}\hfil\\
Let $G<\PSL$ be a subgroup isomorphic to $\ZZ_{n}$ with $n\geq 3$. 
Let $\eta$ be a 1--form with simple poles and simple zeros and let $\ell\geq1$.

\noindent
Then the 1--form $\eta$, with $k=n \ell+2$ poles and $k-2=n \ell$ zeros, 
is  $G$--invariant if and only if 
there is a $G$--invariant hosohedra $\A=\mathfrak{H}_{n}$ such that
\begin{enumerate}
\item[$\bullet$] $V(\A)\subset\Po_{\eta}$.
\item[$\bullet$] There are exactly $\ell$ poles and $\ell$ zeros in a quasi--fundamental region $\widehat{\R}_{\ZZ_n}$.

\end{enumerate}
Moreover $G$ is the maximal subgroup of $\PSL$ satisfying the above conditions if and only if $\eta$ has 
isotropy group $G$.
\end{theorem}
\begin{proof}
Once again the existence of the hosohedra $\A=\mathfrak{H}_{n}$ is assured as in the case $n=2$ by placing 
the vertices of $\A$ on the unique fixed points $\{x,y\}\subset\CW$ of $G$; the edges being a circle containing 
the vertices; the faces being the complement of vertices and edges.

\noindent
The result now follows as an immediate consequence of Theorem \ref{caracterizacion1}.3 by noticing that 
since the order of any generator of $G$ is $n\geq 3$ then cases (B) and (C) of Theorem \ref{Z2} can not occur.
\end{proof}

\begin{remark}
Noting that the fixed points $\{x,y\}\subset\CW$ of $G\cong\ZZ_{n}$ provide us with the family of hosohedra 
$\{\A=\mathfrak{H}_{n}\}$ we can restate the above theorems in terms of the fixed points as follows.
\end{remark}
\smallskip

\begin{theorem}[Case $\ZZ_2$ revisited]
\hfill \\
Let $G<\PSL$ be a subgroup isomorphic to $\ZZ_{2}$. 
Let $\eta$ be a 1--form with simple poles and simple zeros. 

\noindent
Then the 1--form $\eta$, with $k$ poles and $k-2$ zeros, 
is  $G$--invariant 
if and only if 
one of the following cases is true.
\begin{enumerate}
\item[A)]
	\begin{enumerate}
	\item[$\bullet$] The fixed points $\{x,y\}\subset\CW$ of $G$ are poles.
	\item[$\bullet$] There are $\ell$ poles and $\ell$ zeros in a quasi--fundamental region $\widehat{\R}_{\ZZ_{2}}$ for 
	some positive 
	$\ell\geq1$,
	satisfying 
	$$k=2\times\ell+2.$$ 
	\end{enumerate}
\item[B)]
	\begin{enumerate}
	\item[$\bullet$] The fixed points $\{x,y\}\subset\CW$ of $G$ are zeros.
	\item[$\bullet$] There are $\ell$ poles and $\ell-2$ zeros in a quasi--fundamental region $\widehat{\R}_{\ZZ_{2}}$ 
	for some positive 
	$\ell\geq2$,
	satisfying 
	$$k=2\times\ell.$$ 
	\end{enumerate}
\item[C)]
	\begin{enumerate}
	\item[$\bullet$] The fixed points $\{x,y\}\subset\CW$ of $G$ are exactly a pole and a zero.
	\item[$\bullet$] There are $\ell$ poles and $\ell-1$ zeros in a quasi--fundamental region $\widehat{\R}_{\ZZ_{2}}$ 
	for some positive 
	$\ell\geq1$,
	satisfying 
	$$k=2\times\ell+1.$$ 
	\end{enumerate}
\end{enumerate}
Moreover $G$ is the maximal subgroup of $\PSL$ satisfying one of (A)--(C) if and only if $\eta$ has isotropy 
group $G$.
\end{theorem}

\begin{theorem}[Case $G\cong\ZZ_{n}$ with $n\geq 3$ revisited]\hfill\\ 
Let $G<\PSL$ be a subgroup isomorphic to $\ZZ_{n}$ with $n\geq 3$. 
Let $\eta$ be a 1--form with simple poles and simple zeros and let $r\geq1$.

\noindent
Then the 1--form $\eta$, with $k=n r+2$ poles and $k-2=n r$ zeros, 
is  $G$--invariant if and only if
$\eta$ has 
\begin{enumerate}
\item[$\bullet$]  two poles at the fixed points $\{x,y\}\subset\CW$ of 
$G\cong\ZZ_n$,
\item[$\bullet$]  exactly $r$ poles and $r$ zeros in a quasi--fundamental region $\widehat{\R}_{\ZZ_n}$.

\end{enumerate}
Moreover $G$ is the maximal subgroup of $\PSL$ satisfying the above conditions if and only if $\eta$ has 
isotropy group $G$.
\end{theorem}

Notice that Theorem \ref{Z2}.C provides the smallest example (in terms of the least number of poles) when 
$G\cong\ZZ_{2}$.

\begin{example}[The simplest cyclic: case C of Theorem \ref{Z2}] 
Let $p_1, p_2, x\in\CW$ be three points and  $T\in\PSL$ an elliptic transformation 
such that $T(p_1)=p_2$, $T(p_2)= p_1$ and $T(x)= x$. Let $y$ be the other fixed point of $T$.

\begin{enumerate}
\item Then 
$$\eta=\lambda\frac{(z-y)}{(z-p_1)(z-p_2)(z-x)} \ dz,
\quad\text{for }\lambda\in\CC^{*},
$$
is the simplest 1-form with isotropy group $\ZZ_2$ with exactly 3 poles.

\item There is a quasi--fundamental region $\widehat{\R}_{\ZZ_{2}}$, of the group $G$ generated by $T$, 
containing $\{p_{1},p_{2}\}$ but not containing $\{x,y\}$.  Add $\ell-1$ poles $\{p'_{i}\}$ and $\ell-1$ zeros 
$\{q'_{i}\}$ to $\widehat{\mathcal{R}}_{\ZZ_{2}}$. 
Then the 1--form  
$$\eta= \lambda\frac{(z-y)}{(z-p_1)(z-p_2)(z-x)}\,
\frac{\prod\limits_{i=1}^{2\ell-2}(z-q'_i)}{\prod\limits_{i=1}^{2\ell-2}(z-p'_i)}\ dz,
\quad\text{for }\lambda\in\CC^{*},
$$  
has isotropy group $\ZZ_2$ and has exactly $2\ell$ zeros and $2\ell+1$ poles. 
\end{enumerate}

\end{example}

\subsection{Global--geometric characterization of rational 1--forms with finite isotropy}\label{AnalClass}

Summarizing Theorems \ref{main2}, \ref{teoremaDihedrico}, \ref{D22}, \ref{Z2} and \ref{main3b}, 
we immediately obtain the following general classification result for non--trivial finite isotropy:

\begin{corollary}\label{genclass}
Let $G<\PSL$ be a non--trivial finite subgroup of $\PSL$ 
and $k\geq3$.
Then 
$$
\eta=\lambda \ \frac{\ \prod\limits_{j=1}^{k-2}(z-q_{j})\ }{\prod\limits_{\iota=1}^{k}(z-p_{\iota})}\ dz, 
\quad \lambda\in\CC^{*}, 
\quad q_{j}\in\Z_\eta, 
\quad p_{\iota}\in\Po_\eta
$$
is a 1--form on $\CW$ with exactly $k-2$ simple zeros, $k$ simple poles and with isotropy group $G$
if and only if
\begin{enumerate}
\item[1)] we can place $k-\ell_{2}\abs{G}$ poles and $k-2-\ell_{1}\abs{G}$ zeros on the vertices $V(\A)$, centers 
of edges $E(\A)$ and centers of faces $F(\A)$, of the corresponding M\"obius polyhedra $\A$, as in 
Table \ref{tablaK},
\item[2)] we can place exactly $\ell_{1}$ zeros and $\ell_{2}$ poles in a quasi--fundamental region 
$\widehat{\R}_{G}$ ($\widehat{\R}_{\ZZ_{n}}$ in the case of the cyclic groups), 
where the number of poles $k=k(\ell_{1},\ell_{2})$ is given by a simple formula that depends on the difference 
$dif=\ell_{1}-\ell_{2}$, as in Table \ref{tablaK}.
\end{enumerate}
\end{corollary}

\begin{table}[htp]
\caption{Formula for the number of poles $k$, and what to place on $V(\A)$, $E(\A)$ and $F(\A)$, in terms of 
the difference $dif=\ell_{1}-\ell_{2}$.}
\begin{center}
\begin{tabular}{|c|c|c|c||c|c|}
 \hline
 $G$ & $A_{4}$, $S_{4}$, $A_{5}$ & $\DD_{n}$ & $\DD_{2}$ & 
 $\ZZ_{n}$ 
 & 
 $\ZZ_{2}$ 
 \\ 
 \hline
 & & & & & \\[-8pt]
 $\abs{G}$ & $12$, $24$, $60$ & $2 n$ & $4$ & $n\geq 3$ & $2$ \\[3pt]
 \hline
 $\A$ & Platonic & Dihedra & Dihedra & Hosohedra & Hosohedra \\
 \hline
 \multicolumn{6}{}{} \\
 \hline
  & \multicolumn{5}{|c|}{} \\
 $dif$ & \multicolumn{5}{|c|}{Formula for $k$} \\
 & \multicolumn{5}{|c|}{} \\
 \hline
 & & & & & \\[-8pt]
 $-2$ & & & $\ell_{2}\abs{G}$ & & $\ell_{2}\abs{G}$ \\[3pt]
 \hline
 & & \multicolumn{2}{|c||}{} & \multicolumn{2}{|c|}{} \\[-8pt]
 $-1$ & & \multicolumn{2}{|c||}{$\ell_{2}\abs{G}+f$} & \multicolumn{2}{|c|}{$\ell_{2}\abs{G}+1$} \\[3pt]
 \hline
 & \multicolumn{3}{|c||}{} & \multicolumn{2}{|c|}{} \\[-8pt]
 $0$ & \multicolumn{3}{|c||}{$\ell_{2}\abs{G}+v+f$} & \multicolumn{2}{|c|}{$\ell_{2}\abs{G}+2$} \\[3pt]
\hline
 & \multicolumn{3}{|c||}{} & & \\[-8pt]
$1$ & \multicolumn{3}{|c||}{$\ell_{2}\abs{G}+v+f+e$} & &  \\[3pt]
\hline
 \multicolumn{6}{}{} \\
\hline
 & \multicolumn{5}{|c|}{} \\
$dif$ & \multicolumn{5}{|c|}{What to place on $V(\A)$, $E(\A)$ and $F(\A)$} \\
 & \multicolumn{5}{|c|}{} \\
\hline
$-2$ &  &  & $V(\A)\cup E(\A)$ &  & $V(\A)\subset\Z_{\eta}$ \\ 
& & & $\cup F(\A)\subset\Z_{\eta}$ & & \\
\hline
$-1$ &  & \multicolumn{2}{|c||}{$V(\A)\cup E(\A)\subset\Z_{\eta}$} &  & $V(\A)=\{x,y\}$ \\
 &  & \multicolumn{2}{|c||}{$F(\A)\subset\Po_{\eta}$} &  & $x\in\Po_{\eta}$, $y\in\Z_{\eta}$ \\
\hline
$0$ & \multicolumn{3}{|c||}{$V(\A)\cup F(\A)\subset\Po_{\eta}$} & \multicolumn{2}{|c|}{$V(\A)\subset\Po_{\eta}$}  \\
& \multicolumn{3}{|c||}{$E(\A)\subset\Z_{\eta}$} & \multicolumn{2}{|c|}{} \\
\hline
$1$ & \multicolumn{3}{|c||}{$V(\A)\cup E(\A)\cup F(\A)\subset\Po_{\eta}$} &  &  \\
\hline
\end{tabular}
\end{center}
\label{tablaK}
\end{table}%

\subsection{Main result}
We can now state the main theorem.
\begin{theorem}[Classification of rational 1--form with simple poles and simple zeros according to their isotropy group]\label{mainresult}
Let $\eta$ be a rational 1--form on $\CW$ with simple poles and simple zeros. Let $k\geq 2$ denote the number of poles of $\eta$.
\begin{enumerate}
\item When $k=2$, $\eta$ is conjugate to $\widehat{\eta}=\frac{\lambda}{z}\, dz$ for $\lambda\in\CC^{*}$, it's isotropy group is $\CC^{*}=\{z\mapsto az\ \vert\ a\in\CC^{*}\}$.

\item When $k\geq 3$, 
$$
\eta=\lambda \ \frac{\ \prod\limits_{j=1}^{k-2}(z-q_{j})\ }{\prod\limits_{\iota=1}^{k}(z-p_{\iota})}\ dz, 
\quad \lambda\in\CC^{*}, 
\quad q_{j}\in\Z_\eta, 
\quad p_{\iota}\in\Po_\eta
$$
and it has finite isotropy group $G$ as in Corollary \ref{genclass}, or $G=Id$.
\end{enumerate}

\end{theorem}
\begin{proof}
Follows directly from Remark \ref{separacionviak} and \S\ref{AnalClass}.
\end{proof}

\section{Other related results}\label{sec4}

\subsection{Bundle structure for 1--forms with finite isotropy}\label{sec:variedad}

Recall that we are studying 1--forms on $\CW$ that only have simple zeros and poles, 
from Corollary \ref{genclass}, it is natural to consider the following 
\begin{multline}
\mathcal{M}(G,\ell_{1},\ell_{2})=\Big\{ \eta \ |\ \text{Isotropy}(\eta)=G,  \\
\#(\Z_{\eta}\cap\widehat{\R}_{G})= \ell_{1}, \qquad
\#(\Po_{\eta}\cap\widehat{\R}_{G})= \ell_{2} \Big\}.
\end{multline}
This is the set of 1--forms with isotropy group $G$ and exactly $\ell_{1}$ zeros and $\ell_{2}$ poles in a quasi--
fundamental region $\widehat{\R}_{G}$.

In \cite{Mucino2} the authors prove that the space of all 1-forms  up to degree $-s$, denoted by 
$\Omega^1(-s)$, is biholomorphic to a nontrivial line bundle over $\CP^{s}\times\CP^{s-2}$. 

In a similar vein, we begin by proving the following

\begin{theorem}\label{variedad} Let $G<\PSL$ be a 
finite
subgroup. Then 
\begin{enumerate}

\item[1)] $\M(G,\ell_{1},\ell_{2})$ is a holomorphic $\big(\frac{\PSL}{G}\times\CC^*\big)$--bundle over 
$$\left(\frac{(\widehat{\R}_{G})^{\ell_1}\times(\widehat{\R}_{G})^{\ell_2}-\Delta}{S_{\ell_1}\times S_{\ell_2}}\right)$$ 
where $S_{\ell_i}$ is the symmetric group of $\ell_i$ elements and 
$\Delta\subset (\widehat{\R}_{G})^{\ell_1}\times (\widehat{\R}_{G})^{\ell_2} $ is the set of diagonals.
\item[2)] $\M(G,\ell_{1},\ell_{2})$ is a complex analytic sub--manifold of $\Omega^1(-k)$, 
of dimension $dim\big(\M(G,\ell_{1},\ell_{2})\big)=\ell_{1}+\ell_{2}+4$, 
where $k=k(\ell_{1},\ell_{2})$ is as in Table \ref{tablaK}.
\item[3)] $\M(G,\ell_{1},\ell_{2})$ is arc--connected; that is, if $\eta_1$, $\eta_2\in \M(G,\ell_{1},\ell_{2})$, then 
there exists a differential function $F:[0,1]\to \M(G,\ell_{1},\ell_{2})$ such that $F(0)=\eta_1$ and 
$F(1)=\eta_2$.
\end{enumerate}
\end{theorem}

\begin{proof} 
For (1) consider Corollary \ref{genclass}.
Since $S_{\ell_1}\times S_{\ell_2}$ acts on $(\widehat{\R}_{G})^{\ell_1}\times(\widehat{\R}_{G})^{\ell_2}-\Delta$ 
by stripping the order of the placement of the $\ell_{1}$ zeros and $\ell_{2}$ poles on the quasi--fundamental 
region $\R_{G}$, the action of $S_{\ell_1}\times S_{\ell_2}$ is holomorphic and free; 
thus
$$E=\left(\frac{(\widehat{\R}_{G})^{\ell_1}\times(\widehat{\R}_{G})^{\ell_2}-\Delta}
{S_{\ell_1}\times S_{\ell_2}}\right)$$ 
is a holomorphic manifold of (complex) dimension $\ell_{1}+\ell_{2}$.
Let $ \{V_ \alpha \}_ {\alpha \in A} $ be an atlas for $E$, that is 
a collection of open sets in $E$, such that $V_ \alpha $ is biholomorphic to a subset of 
$ \CC ^ {\ell_1 + \ell_2} $ and $ \cup _ {\alpha \in A} V_ \alpha =  E$. 
The push--forward of $\eta$ by $\PSL$ provides a 1--form in 
$\M(G,\ell_1,\ell_2)$ with isotropy $G$ up to homothecy provided by the main coefficient  
$\lambda \in \CC^*$. Thus $\{V_ \alpha \times \frac{\PSL}{G}\times\CC^*  \} _ {\alpha \in A}$ is a holomorphic 
atlas for $\M(G,\ell_1,\ell_2)$.

For (2), first note that the fact that $\M(G,\ell_{1},\ell_{2})$ is a complex analytic manifold follows directly 
from (1). 
To show that $\M(G,\ell_{1},\ell_{2})$ is a sub--manifold of 
$\Omega^1(-k)$, note that $Id:\M(G,\ell_{1},\ell_{2})\hookrightarrow\Omega^1_{S}(-k)$ is a 
submersion
into $\Omega^1_{S}(-k)\subset\Omega^1(-k)$, where $\Omega^1_{S}(-k)$ are the 1--forms of degree $-k$ 
with simple poles and zeros (which, by the way, is dense in $\Omega^1(-k)$).
That $Id:\M(G,\ell_{1},\ell_{2})\hookrightarrow\Omega^1_{S}(-k)$ is a submersion follows directly 
by using the coordinate system comprised of the principal coefficient $\lambda$, the poles 
$\Po_{\eta}$ and zeros $\Z_{\eta}$.
To relate to the coordinate system provided by the principal coefficient $\lambda$ and the coefficients of $\eta$ 
considered as a quotient of monic polynomials, use the Vi\`ete map 
$V:\Omega^1(-k)\rightarrow\Omega^1(-k)$, see \cite{Katz}, and note that $V$ is bi--rational/non--singular on 
$\Omega_{S}^1(-k)$.

To prove (3), first note that since $(\widehat{\R}_{G})^{\ell_1}\times(\widehat{\R}_{G})^{\ell_2}-\Delta$ is 
arc--connected, the base space $E$ is arc--connected. 
Moreover, each fiber is clearly arc--connected since $\PSL\times\CC^{*}$ is arc--connected. 
Because of the local cartesian product structure of $\M(G,\ell_{1},\ell_{2})$ the result follows.
\end{proof}

\begin{remark}
Of course, as shown in Theorem \ref{variedad}.2 $\M(G,\ell_{1},\ell_{2})\subset\Omega^1(-k)$.
However, by considering the fibers, it is clear that $\M(G,\ell_{1},\ell_{2})$ is not a sub--bundle of 
$\Omega^1(-k)$.
\end{remark}

\subsection{Sufficient geometric conditions for isochronicity}
\label{secIsocron}
Recall that an iso\-chro\-nous 1--form can be characterized by requiring that all its residues be purely 
imaginary.
In regards to which of the invariant 1--forms are isochronous, we have this nice geometric result 
(recall that a circle passing through $\infty\in\CW$ is a line in $\CC$).

\begin{theorem} [Sufficient geometric conditions for isochronous 1--forms with simple poles and zeros having 
finite non--trivial isotropy]
\label{sufficientIsochronous}
Let $\eta$ be a 1--form with finite non--trivial isotropy group $G$. Let $E\subset\CW$ be a circle
such that the reflection $\rho_{E}$ along $E$ satisfies that 
for all $p\in\Po_{\eta}$ and all $q\in\Z_{\eta}$
\begin{enumerate}
\item[1)] $\rho_{E}(p)\in\mathcal{O}(p)\subset\Po_{\eta}$,
\item[2)] $\rho_{E}(q)\in\mathcal{O}(q)\subset\Z_{\eta}$.
\end{enumerate}
Then there exists $\theta\in\RR$ such that $\e^{i\theta} \eta$ is isochronous.
\end{theorem}
\begin{proof}
Clearly there is a $T\in\PSL$ such that $T(E)\subset\RR$ and thus it follows that $(T\circ \rho_{E}\circ T^{-1})(z)=\zbar$ for $z\in\CC$.
\\
Let $\widehat{\eta}=T_{*}\eta$, thus 
$$\widehat{\eta}(z)=\lambda \frac{Q(z)}{P(z)}$$ 
with $\lambda\in\CC^{*}$, $Q(z),P(z)\in\CC[z]$ 
monic polynomials.
\\
Then, from conditions (1) and (2), for each pole $\widetilde{p}_{j}$ of $\widehat{\eta}$ and for each zero $\widetilde{q}_{\iota}$ of $\widehat{\eta}$ one has that $\bar{\widetilde{p}_{j}}= g_1 \widetilde{p}_{j}$ and $\bar{\widetilde{q}_{\iota}}=g_2 \widetilde{q}_{\iota}$ for some $g_1,g_2\in T\circ G\circ T^{-1}$; in other words $\bar{\widetilde{p}_{j}}$ and $\bar{\widetilde{q}_{\iota}}$ are also a pole and a cero, respectively, of $\widehat{\eta}$.  
Hence it follows that both $Q(z)$ and $P(z)$ have real coefficients. 
\\
Hence, for any $\widetilde{p}_{j}\in\Po_{\widehat{\eta}}$ 
$$\widehat{\eta}
= \lambda \frac{Q(z) dz}{\Big(z-\widetilde{p}_j\Big)\Big(z-\bar{\widetilde{p}_j}\Big) P_{j}(z)}
= \lambda \frac{Q(z) dz}{\Big(z^{2}-2\Re{\widetilde{p}_j} z+\abs{\widetilde{p}_j}^{2}\Big) P_{j}(z)},
$$ 
with $Q(z)$ and $P_{j}(z)$ being monic polynomials with real coefficients.
\\
And since $\bar{\widetilde{p}_j}$ and $\widetilde{p}_j$ are in the same orbit, their residues are the same, so
$$\lambda \frac{Q(\widetilde{p}_j)}{\Big(\widetilde{p}_j-\bar{\widetilde{p}_j}\Big) P_{j}(\widetilde{p}_j)}=
Res(\widehat{\eta},\widetilde{p}_j)=Res(\widehat{\eta},\bar{\widetilde{p}_j})=-\lambda \frac{Q\Big(\bar{\widetilde{p}_j}\Big)}{\Big(\widetilde{p}_j-\bar{\widetilde{p}_j}\Big) P_{j}\Big(\bar{\widetilde{p}_j}\Big)}.$$
Thus the residue $Res(\widehat{\eta},\widetilde{p}_j)$ is real multiple of $\lambda$ for each pole $\widetilde{p}_j$ of $\widehat{\eta}$. 
Since $\lambda=\abs{\lambda}\e^{i\arg{(\lambda)}}$, let $\theta=\arg{(z)}\pm\pi/2$ to obtain that
$\e^{i\theta}\widehat{\eta}$ is isochronous. 
Finally since $T$ leaves the residues invariant we 
conclude that $\e^{i\theta}\eta$ is isochronous.
\end{proof}
\begin{remark}
Note that the case when $\eta$ has only two poles requires that 

\centerline{
$Res(\eta,p_{1})=-Res(\eta,p_{2})$
}
\noindent
hence in order to extend Theorem \ref{sufficientIsochronous} to this case would require that $Res(\eta,p_{1})=Res(\eta,p_{2})=0$.
\end{remark}
\begin{example}
Because of the high symmetry of the M\"obius polyhedra $\A$ and since the examples of rational 1--forms 
$\eta$ constructed in \S\ref{ejemplos} only have poles or zeros on $V(\A) \cup E(\A) \cup F(\A)$, then it is easy 
to see that the conditions of Theorem \ref{sufficientIsochronous} are satisfied. This provides an alternate proof 
that the examples presented in \S\ref{ejemplos} are isochronous.
\end{example}

The next example shows that the conditions of Theorem \ref{sufficientIsochronous} are sufficient but 
not necessary. 
\begin{example}
Let 
\begin{multline}
\eta(z)= 
i
\left[
\frac{1}{z}
+\frac{1}{z-1}
+\frac{1}{z+1}
+\frac{1}{z-i}
+\frac{1}{z+i}
\right. \\
+\frac{1}{z-\left(\frac{1}{2}-i\right)}
+\frac{1}{z+\left(\frac{1}{2}-i\right)}
+\frac{1}{z-\left(1+\frac{i}{2}\right)}
+\frac{1}{z+\left(1+\frac{i}{2}\right)}
\\
\left.
-\frac{1}{z-\frac{1}{2}}
-\frac{1}{z+\frac{1}{2}}
-\frac{1}{z-\frac{i}{2}}
-\frac{1}{z+\frac{i}{2}}
\right] dz,
\\
=\frac{(z^{4}-a^{4}) (z^{4}-b^{4}) (z^{4}-c^{4})}
{z (z^{4}-(\frac{1}{2})^{4}) (z^{4}-1) (z^{4}-(1+\frac{i}{2})^{4})}dz.
\end{multline}
with $a, b, c \in\CC$ determined by the partial fraction expansion. 
\\
By inspection it is clear that there are 3 different residues and that they are a real multiple of each other. 
Moreover, 
note that $\eta$ has isotropy group $G\cong\ZZ_{4}$.
\\
The fixed points of $G$ are $\{0,\infty\}\subset\CW$ and they are poles of $\eta$ with residue $i$ and $-5i$ 
respectively.
\\
The orbits of $1/2$, $1$, and $1+ i/2$ are also poles with residues $-i$, $i$ and $i$ respectively.
\\
The orbits of $a$, $b$, and $c$ are zeros.
\\
Hence $\eta$ has a total of 14 poles and 12 zeros 
and is an isochronous rational 1--form.
See figure \ref{contraejemploisocrona} for the phase portrait of $\eta$.

We want to see whether there is a circle $E\subset\CW$ satisfying conditions (1) and (2) of 
Theorem \ref{sufficientIsochronous}. 
Since $0$ and $\infty$ are fixed points $\mathcal{O}(0)=\{0\}$ and $\mathcal{O}(\infty)=\{\infty\}$, so by 
condition (1), $E$ must pass through $0$ and $\infty$, {\it i.e.} $E$ is a straight line through the origin. 
Letting $\rho_{E}\in\PSL$ be the reflection through $E$, it is clear that because of the 4--fold symmetry of the 
poles and zeros, there is no straight line $E$ passing through the origin that satisfies conditions (1) and (2) of 
Theorem \ref{sufficientIsochronous}.

\begin{figure}[htbp]
\begin{center}
\hskip 10pt
\includegraphics[height=0.45\textwidth]{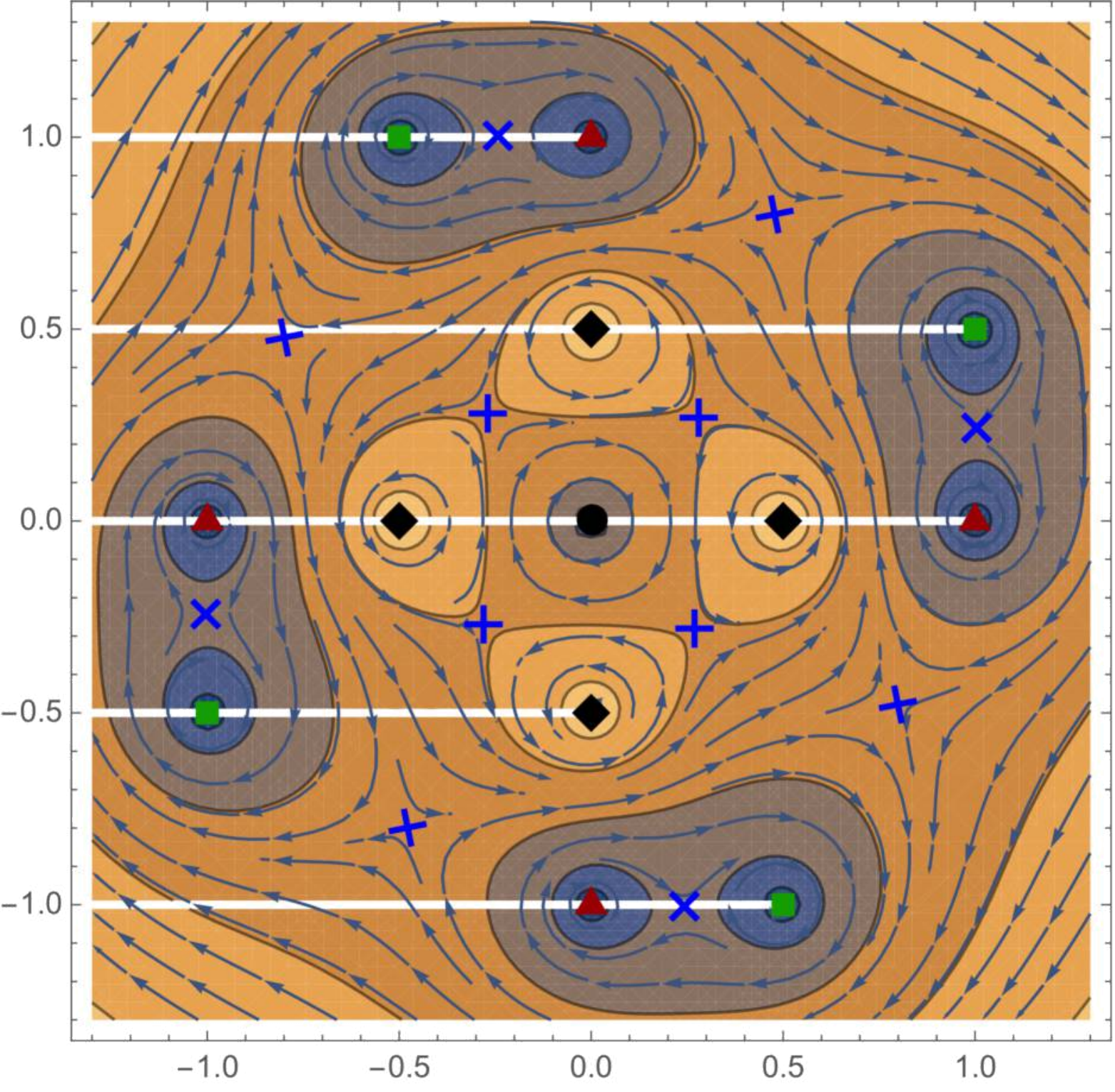}
\caption{
Example of an isochronous 1--form $\eta$. Note that there is no circle $E\subset\CW$ satisfying conditions (1) 
and (2) of Theorem \ref{sufficientIsochronous}. 
The zeros are saddles and are presented as (blue) crosses, the poles are centers and appear as (red) triangles, (green) squares, (black) diamonds and (black) dots according to their orbits (the other element of the  orbit of the origin is $\infty\in\CW$).
}
\label{contraejemploisocrona}
\end{center}
\end{figure}

\end{example}

\subsection{Langer's question}\label{secLanger}
In \cite{Langer} J.~C.~Langer studies quadratic differentials 

\centerline{
$\mathcal{F}=f(z)\,dz^{2}$, 
}
\noindent
for a rational function $f(z)$ on $\CW_{z}$, and presents a way to plot the polyhedral geometry of 
$\mathcal{F}$ using the phase portrait of the 1--form $\eta=\sqrt{f(z)}dz$. 

Since he is interested in \emph{`computational strategies for numerically plotting edges and other geodesics 
for such polyhedral geometries'}, 
J.~C.~Langer first considers the non--compact metric space $(\mathcal{F}_{fin},d)$, where $\mathcal{F}_{fin}$ 
is the finite points (consisting of regular points, zeros, and simple poles of $\mathcal{F}$); and $d(z_{1},z_{2})$ 
is the distance obtained using the metric $g=\abs{\mathcal{F}}$ associated to $\mathcal{F}$.
He then defines the polyhedral geometry of $\mathcal{F}$ as follows: the \emph{vertices}, 
$\mathcal{F}_{vert}$, are to be the finite critical points of $\mathcal{F}$; the \emph{edges}, 
$\mathcal{F}_{edge}$, are the union of the critical trajectories (which are the trajectories which tends to a finite 
limit point (necessarily a zero or simple pole) in one or both directions); and the edges in turn divide 
$\mathcal{F}_{face}=\mathcal{F}_{fin}\backslash\left( \mathcal{F}_{vert}\cup\mathcal{F}_{edge}\right)$ into $n$ 
connected components (the faces) $\mathcal{F}_k$ of a few standard types, including half planes, infinite 
strips, finite or semi-infinite cylinders.

\noindent
J.~C.~Langer procedes to show some examples of the above and asks the question:
\emph{``for which rational functions $f(z)$ does the corresponding polyhedral geometry of 
$\mathcal{F} = f(z)\,dz^2$ embed isometrically into $\RR^3$?''}

Related to this question, we can show a partial result. 
For this denote by $\mathcal{RI}\Omega^{1}\{\CW\}\subset\Omega^{1}\{\CW\}$ the isochronous rational 1--
forms on $\CW$. 
\begin{proposition}\label{geomPolyIsocrono}
The set of quadratic differentials 
$$
\mathcal{PG}=\Big\{\mathcal{F}=(f(z))^{2}\,dz^{2}\ \vert\ \mathcal{F} \text{ has polyhedral geometry} \Big\}
$$
is precisely 
$$
\mathcal{PG}=\Big\{ \mathcal{F}=\eta\otimes\eta \ \vert\ \eta\in\mathcal{RI}\Omega^{1}\{\CW\} \Big\}.
$$
\end{proposition}
\begin{proof}
Of course the description $\mathcal{F}=(f(z))^{2}\,dz^{2}$ is equivalent to $\mathcal{F}=\eta\otimes\eta$ for 
$\eta=f(z)\, dz$.
Moreover, by definition, the trajectories of $\eta$ correspond to the trajectories of $\mathcal{F}$. 

\noindent
From the definition of polyhedral geometry for $\mathcal{F}$, it is required that the edges, 
$\mathcal{F}_{edge}$, be the union of critical trajectories of $\mathcal{F}$. In particular if $\mathcal{F}$ is to 
have polyhedral geometry then the union of critical trajectories of $\mathcal{F}$ must be the union of the 
edges of a spherical polyhedra. 
Thus $\eta\in\mathcal{RI}\Omega^{1}\{\CW\}$.

\noindent
The other inclusion is obvious.
\end{proof}

Recall that Theorem \ref{sufficientIsochronous} provides sufficient conditions that show when a rational 1--form 
$\eta= f(z)\, dz$ with simple poles and simple zeros is isochronous. 
Of course by choosing $\theta\in(0,\pi)$ we obtain examples of quadratic differentials 
$\mathcal{F}=(\e^{i\theta} f(z))^{2}dz^{2}$, for rational $f(z)$ which do not have a polyhedral geometry, even 
though the isotropy group $G_{\eta}$ of the associated 1--form $\eta=f(z)\, dz$ is a platonic group. 
We also have examples of 1--forms $\eta=f(z)\, dz$ invariant under a platonic group $G$ that can not be made 
isochronous, see Figure \ref{tetranoisocrono} for an example with $G_{\eta}=A_{4}$,
hence \emph{can not have a polyhedral geometry, yet its corresponding quadratic differential is rational}.

\begin{figure}[htbp]
\begin{center}
\includegraphics[height=0.5\textwidth]{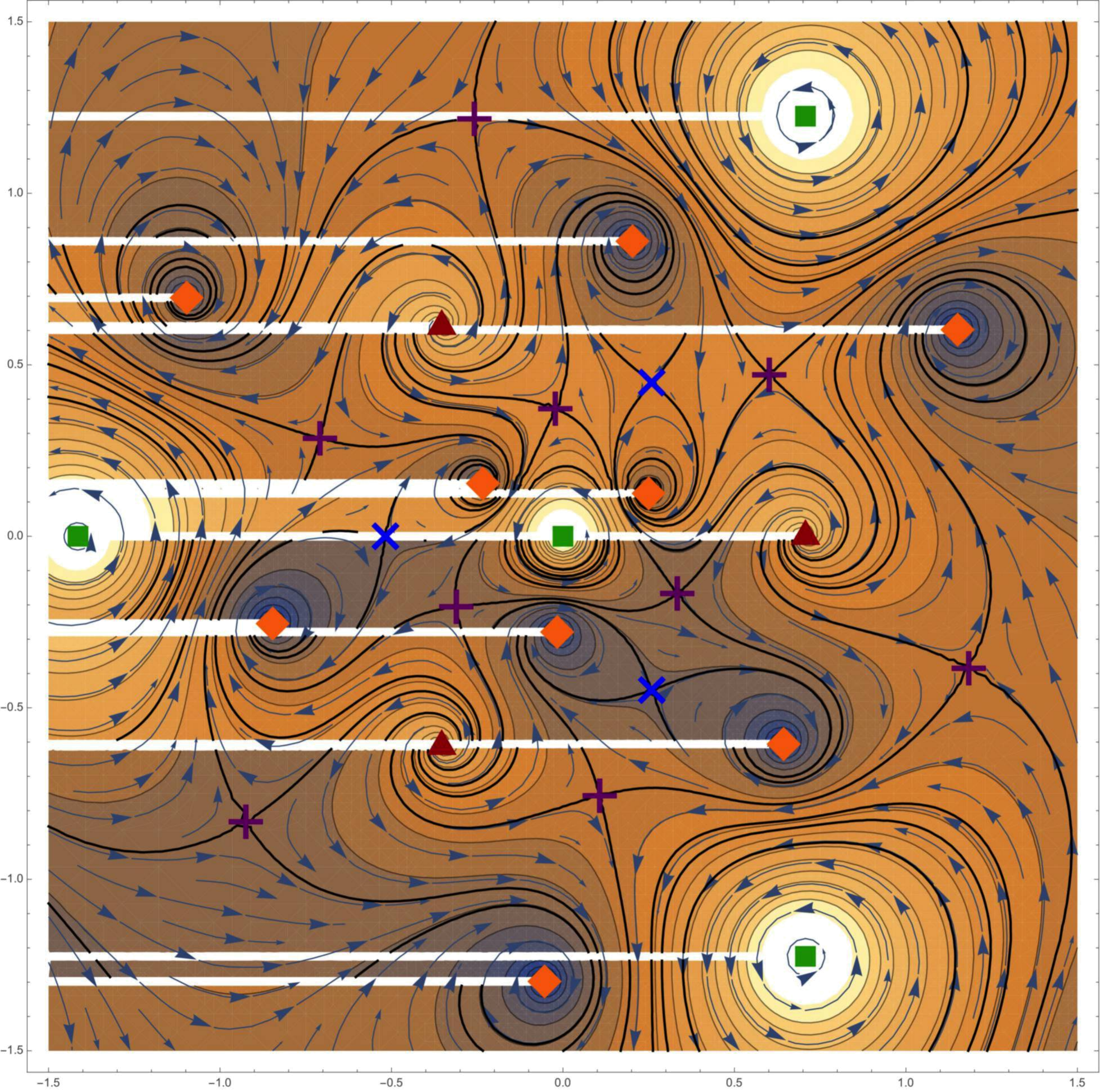}
\caption{Rational 1--form with isotropy group $A_{4}$ that is not isochronous, thus its QD is also rational but 
does not have polyhedral geometry as defined by J.~C.~Langer.
Poles appear as (green) squares, (orange) diamonds and (red) triangles, each orbit distinguished with a different symbol. 
Zeros appear as (red and blue) crosses, in this case each orbit appears with a different color. 
Note that in this picture we can only see the complete orbit for the (green) squares and the (orange) diamonds. 
}
\label{tetranoisocrono}
\end{center}
\end{figure}

\begin{example}\label{ejemploLangerOcta}
Let 
$$\eta(z)= \frac{z (1 - z^4)}{z^8 + 14 z^4 + 1} dz.$$ 
The phase portrait of the vector field associated to $\eta$ can be seen in Figure \ref{LangerOcta}.
J.C.~Langer \cite{Langer} shows that the quadratic differential $\eta\otimes\eta$ has polyhedral geometry of 
the octahedra (whose isotropy group is isomorphic to $S_{4}$).
Thus $\eta$ satisfies conditions (1) and (2) of Theorem \ref{caracterizacion1} with $G\cong S_{4}$, but its 
isotropy group is not $G$.
In fact, it's isotropy group is $G\cong A_{4}$ and it falls in case (a) of Theorem \ref{main2} with $\ell=0$ and 
$k=8$.

\noindent
A realization of the M\"obius polyhedra $\A=Tetrahedron$ is as follows: the vertices are 
$V(\A)=\{\frac{-1+\sqrt{3}}{2}e^\frac{i\pi}{4},-\frac{-1+\sqrt{3}}{2}e^\frac{i\pi}{4},\frac{1+\sqrt{3}}{2}e^\frac{-i\pi}{4},-\frac{1+\sqrt{3}}{2}e^\frac{-i\pi}{4}\}\subset\CW$, the centers of the edges are 
$E(\A)=\{0,-1,1, i, -i, \infty\}$ and the centers of faces are $F(\A)=\{\frac{-1+\sqrt{3}}{2}e^\frac{-i\pi}{4},-\frac{-1+\sqrt{3}}{2}e^\frac{-i\pi}{4},\frac{1+\sqrt{3}}{2}e^\frac{i\pi}{4},-\frac{1+\sqrt{3}}{2}e^\frac{i\pi}{4}\}$, once again 
see Figure \ref{LangerOcta}.
\end{example}

\begin{figure}[htbp]
\begin{center}
\includegraphics[height=0.45\textwidth]{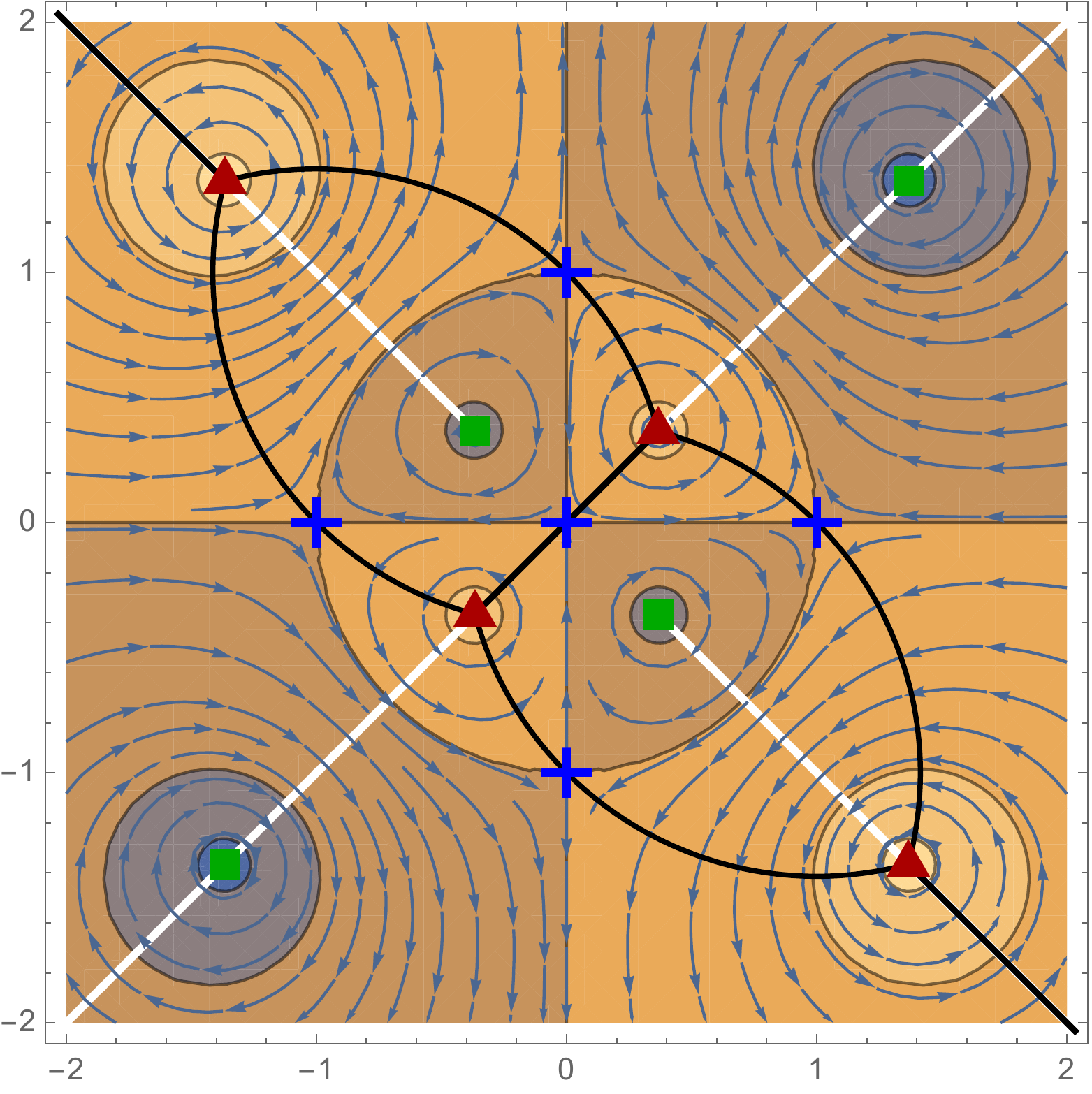}
\caption{Phase portrait of $\eta(z)= \frac{z (1 - z^4)}{z^8 + 14 z^4 + 1} dz$. The set of poles and zeros is 
invariant under $G\cong S_{4}$ the isotropy group of a octahedron. However the isotropy group of $\eta$ is 
$G\cong A_{4}$, the isotropy group of a tetrahedron. The vertices of the tetrahedron are (red) triangles, the centers of edges are (blue) crosses ($\infty\in\CW$ is also a center of an edge), and the centers of faces are (green) squares.
}
\label{LangerOcta}
\end{center}
\end{figure}

\begin{remark}
In our work we search for the symmetries of the rational 1--forms, hence in Figure \ref{LangerOcta} we observe 
a tetrahedron. However J.~C.~Langer observes an octahedron since he is searching for polyhedral 
geometries, where the edges of the polyhedron are the critical trajectories.
\end{remark}


\section{Examples} \label{ejemplos}
In this section we show a example of $1$-form with isotropy $G$, for platonic subgroups and some dihedral 
and cyclic subgroups of $\PSL$. 
\subsection{The case of $A_{4}$}
\label{a4}
Consider the tetrahedron $\A_{1}$ together with its isometry group $\G_{1}\cong A_{4}$ from 
Lemma \ref{classFinitos}. For simplicity we shall use $\A$ and $G$ instead of $\A_{1}$ and $G_{1}$.

To construct the 1--form we set four simple poles at the vertices of the tetrahedron $\A$, another four simple 
poles at the centers of faces and six simple zeros at the midpoints of the edges. See figure \ref{figuraa4}.
\\
With this construction the 1--form is:
\begin{equation}\label{tetra}
\eta=f(z)\:dz=\lambda\ \frac{4 z^6-20 \sqrt{2} z^3-4}{4 z^7+7 \sqrt{2} z^4-4 z} \:dz.
\end{equation}

\begin{figure}[htbp]
\begin{center}
\includegraphics[width=.4\textwidth]{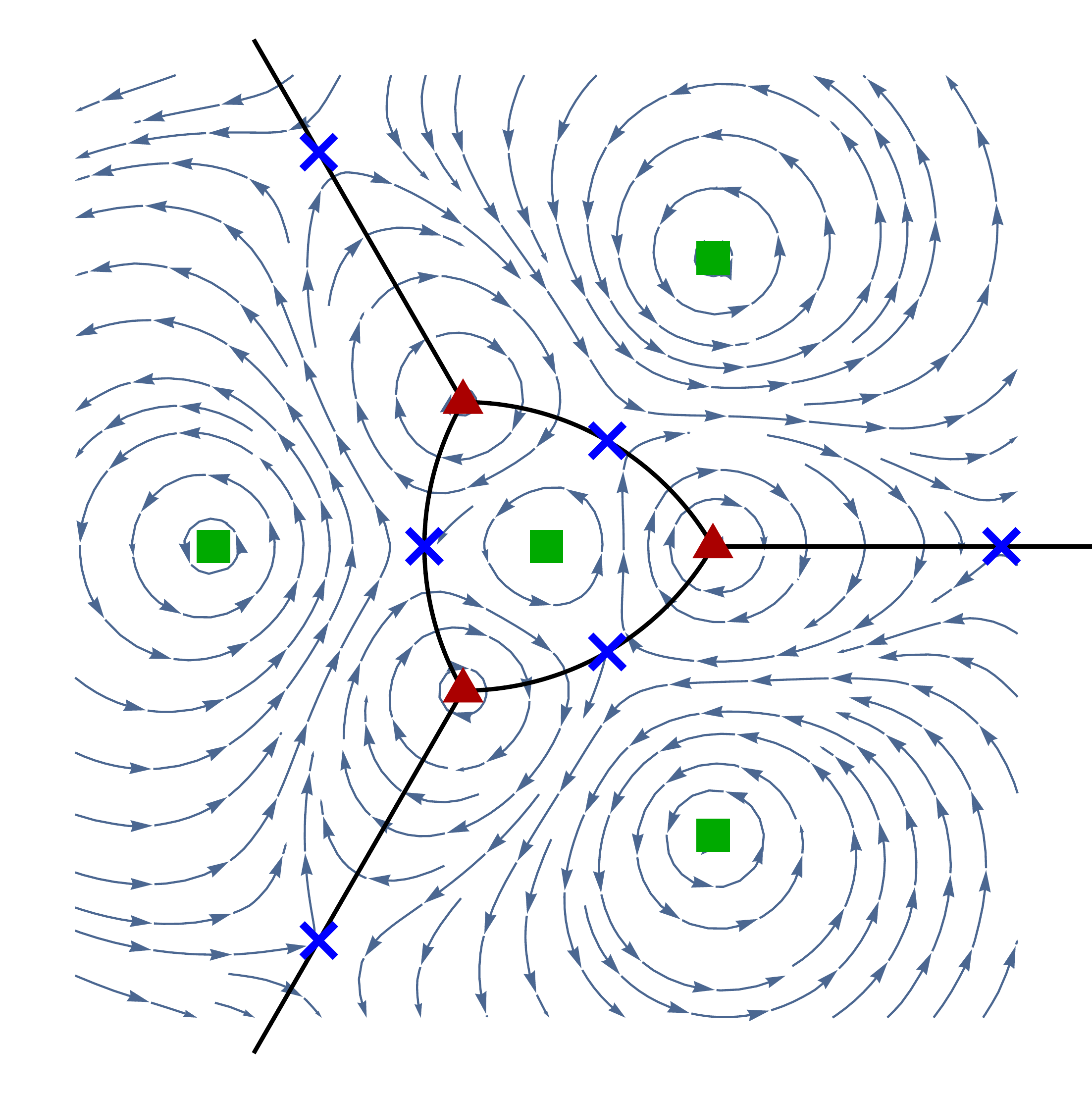}
\hskip 10pt
\includegraphics[height=.4\textwidth]{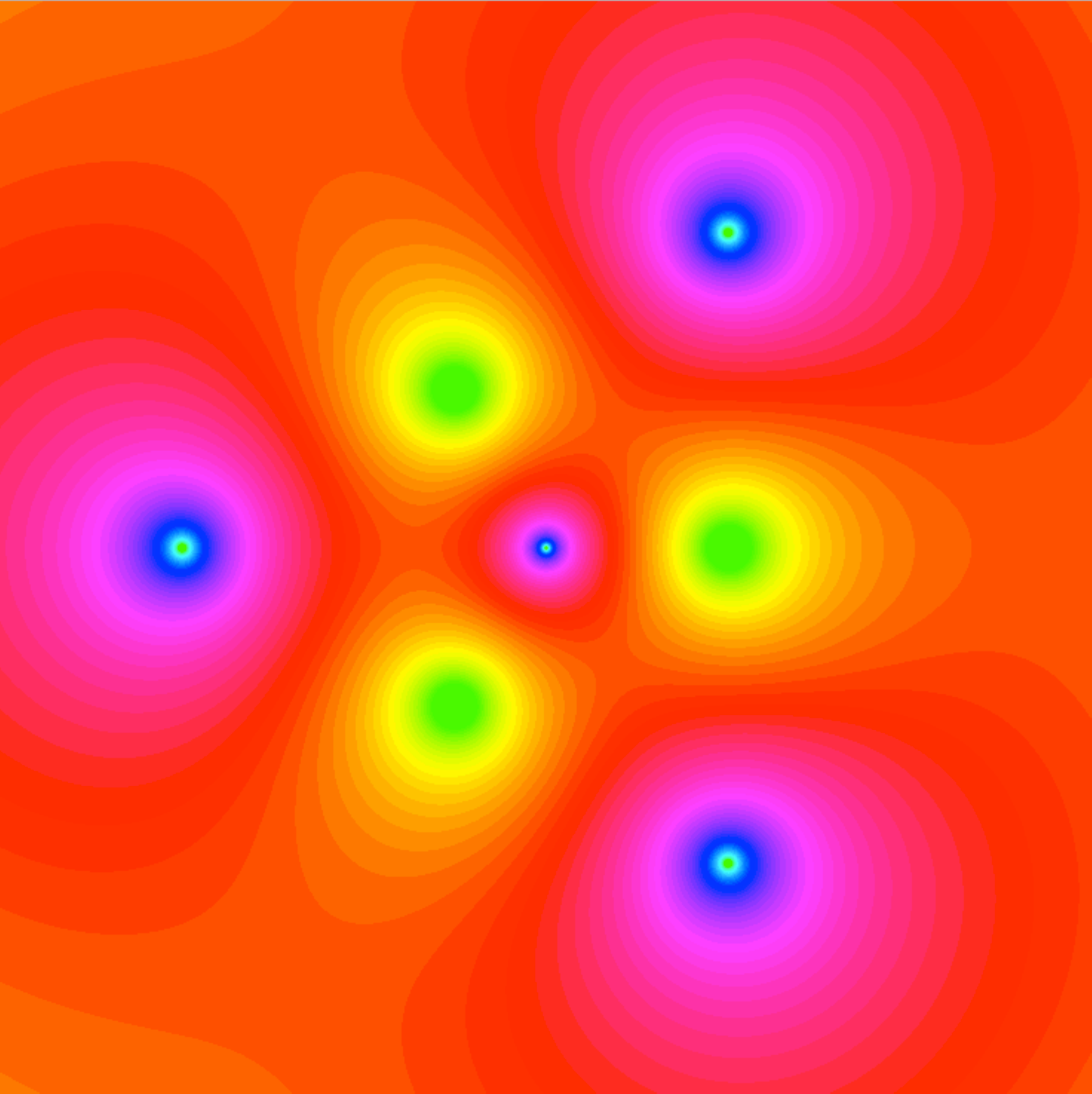}
\\
\includegraphics[width=.4\textwidth]{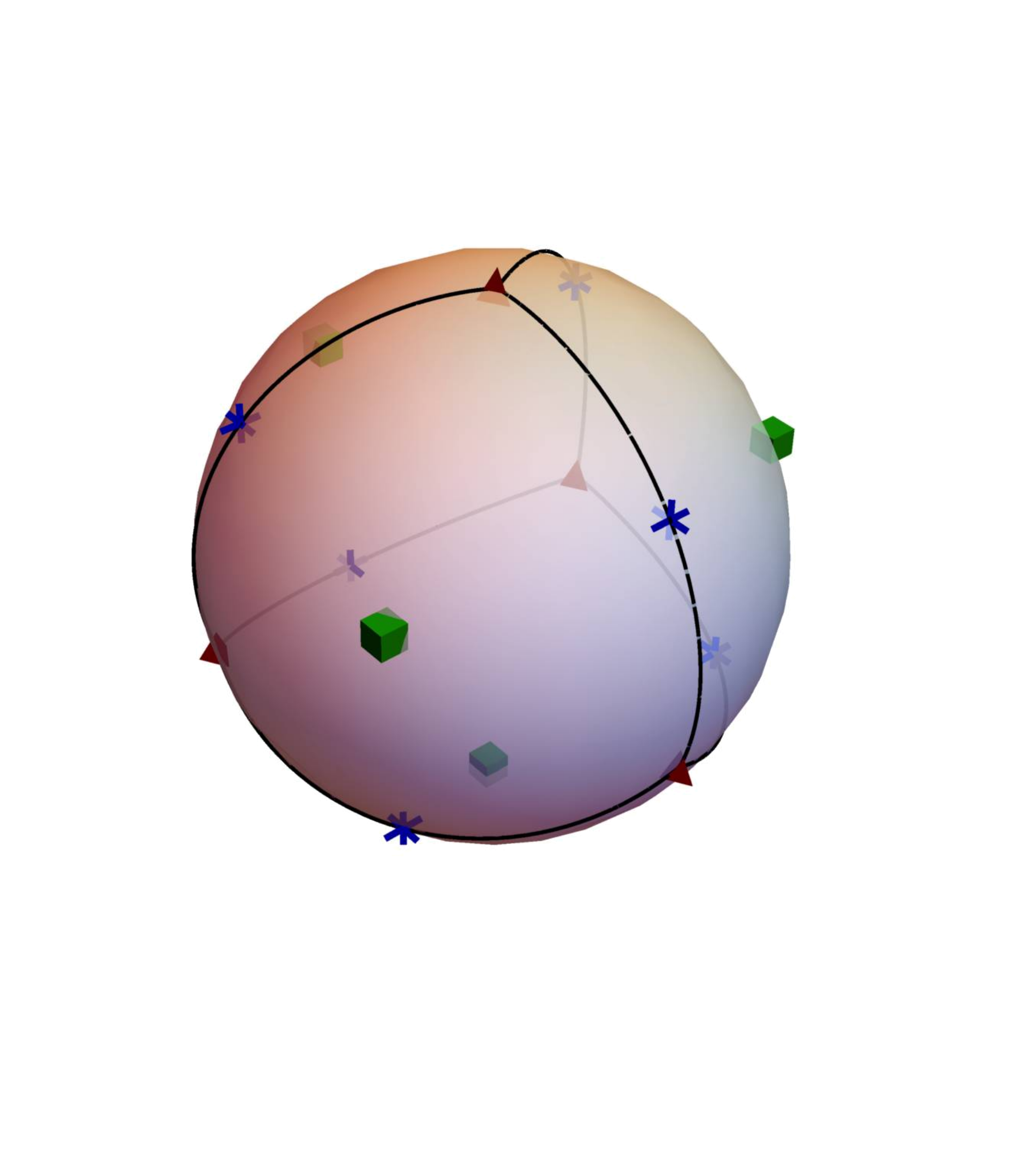}
\hskip 10pt
\includegraphics[height=.38\textwidth]{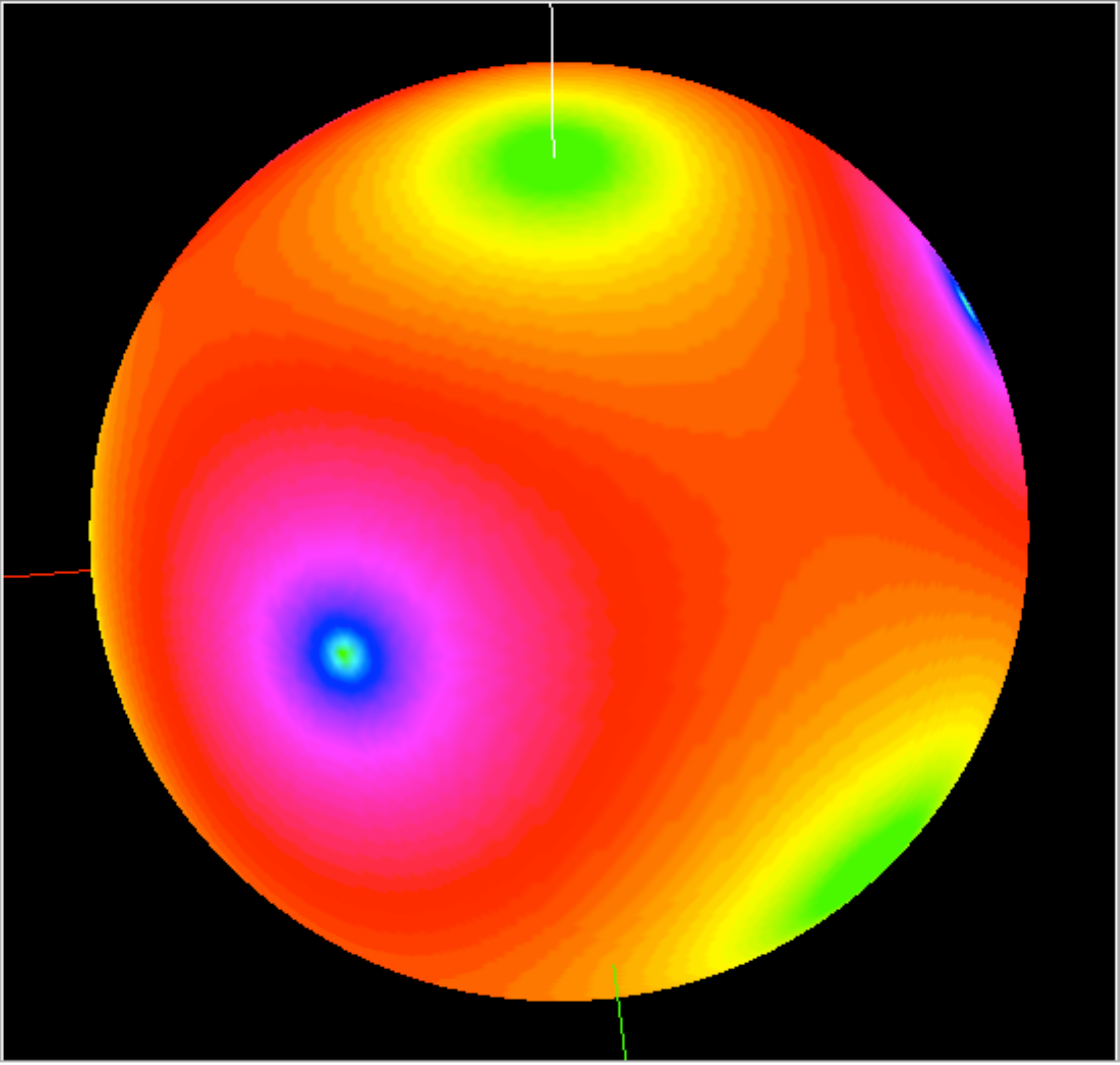}
\caption{Phase portrait of the field associated to the 1--form \eqref{tetra}. This corresponds to the tetrahedron 
and has isometry group isomorphic to $A_{4}$. In this figure we set $\lambda=-i$ so that the 
poles (zeros of the corresponding field) are centers. 
On the right hand side, the poles placed at the vertices appear as the center of the large light colored disks (green surrounded by 
yellow), while the poles placed at the centers of the faces appear as the center of the small disks surrounded 
by dark circles (blue surrounded by purple).
On the left hand, side vertices appear as (red) triangles, centers of edges appear as (blue) crosses and centers of faces appear as (green) squares.
}
\label{figuraa4}
\end{center}
\end{figure}
\noindent
By Theorem \ref{main2}.a with $\ell=0$, $\eta$ is $G$--invariant with $G\cong A_{4}$. 

In order to check that the maximality condition of Theorem \ref{main2} is satisfied, 
note that the only finite subgroups of $\PSL$ that could possibly contain $G\cong A_{4}$ are 
subgroups isomorphic to $S_{4}$. 
However, since $\eta$ has exactly 8 simple poles (and 6 simple zeros) there are not enough (simple) poles to 
place one on each of the vertices and centers of faces of an octahedron or a cube 
(see for instance Table \ref{tablaBuena}). 
Hence $\eta$ can not satisfy the requirements of Theorem \ref{main2} for $\A$ an octahedron or a cube. 
Thus $\eta$ is not $G$--invariant for $G\cong S_{4}$

\noindent
Hence the isotropy group of $\eta$ is $G\cong A_4$. 


\subsection{The case of $S_{4}$}
\label{s4}
Consider the octahedron $\A_{2}$ together with its isometry group $\G_{2}\cong S_{4}$ from 
Lemma \ref{classFinitos}.
For simplicity we shall use $\A$ and $G$ instead of $\A_{2}$ and $G_{2}$.

To construct the 1--form, we set the simple poles at the centers of the faces and at the vertices, also we set the 
simple zeros at the midpoints of the edges. See figure \ref{figures4}.
\\
The 1--form thus constructed is:
\begin{equation}\label{octa}
\eta=f(z)\:dz=- \lambda\ \frac{1-33 z^4-33 z^8+z^{12}}{z+13 z^5-13 z^9-z^{13}} \:dz.
\end{equation}
\noindent
By Theorem \ref{main2}.a with $\ell=0$, $\eta$ has isotropy subgroup $G\cong S_{4}$. 
To see that the phase portrait of the associated field is isochronous we verify that all residues of $\eta$ are  
real multiples of $\lambda$. Hence when $\lambda$ is pure imaginary $\eta$ is isochronous.

\begin{figure}[htbp]
\begin{center} 
\includegraphics[width=.4\textwidth]{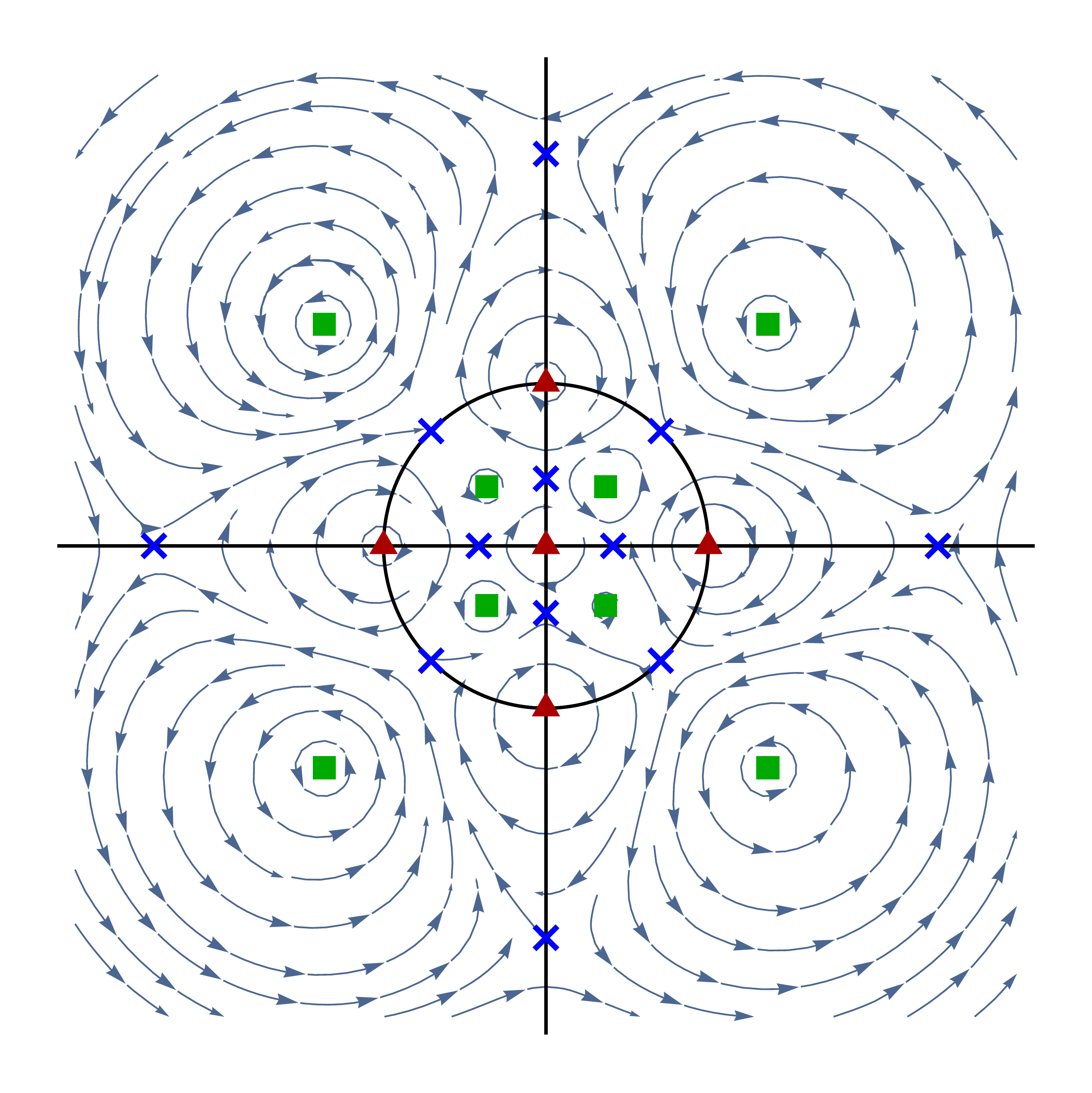}
\hskip 8pt
\includegraphics[width=.4\textwidth]{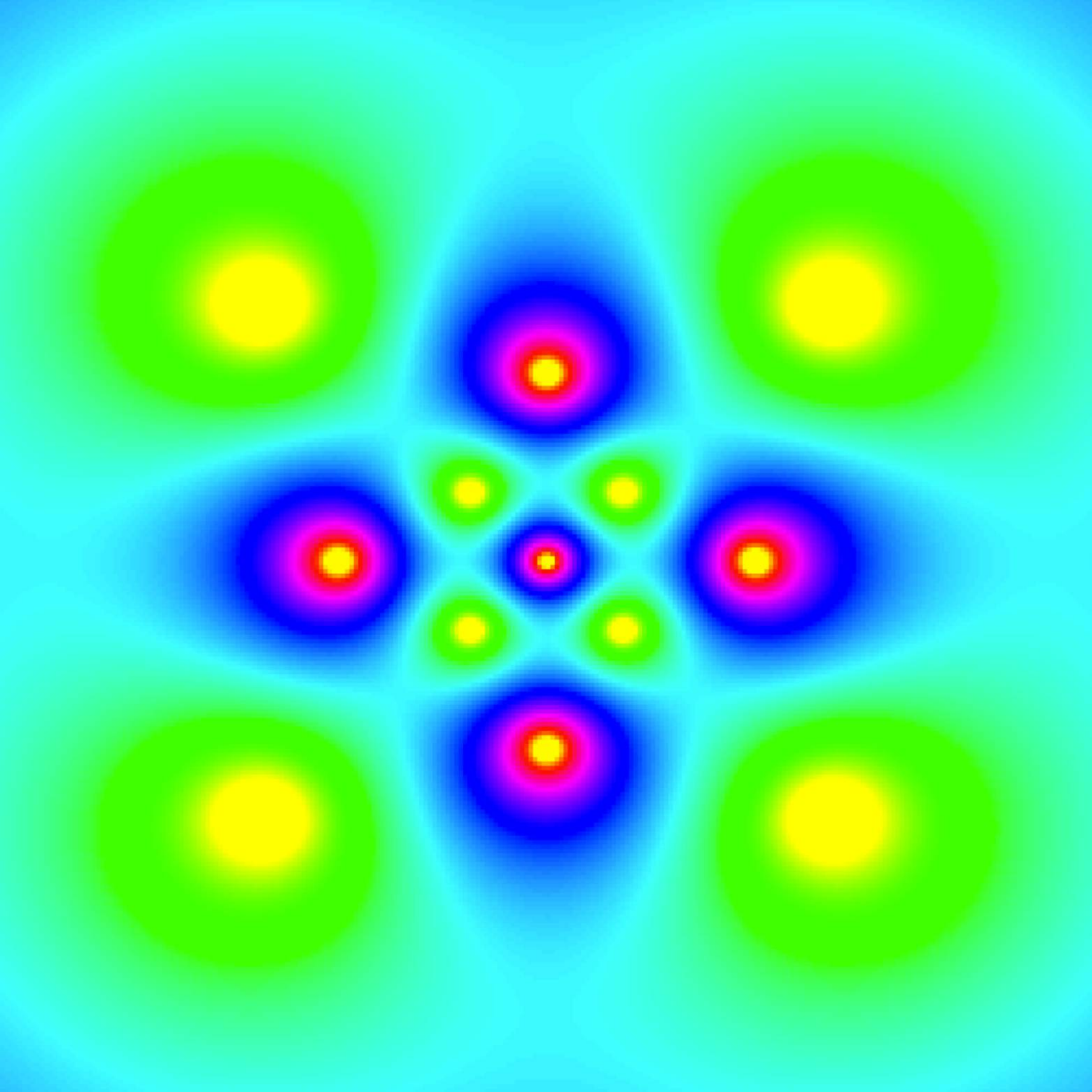}
\\
\includegraphics[width=.375\textwidth]{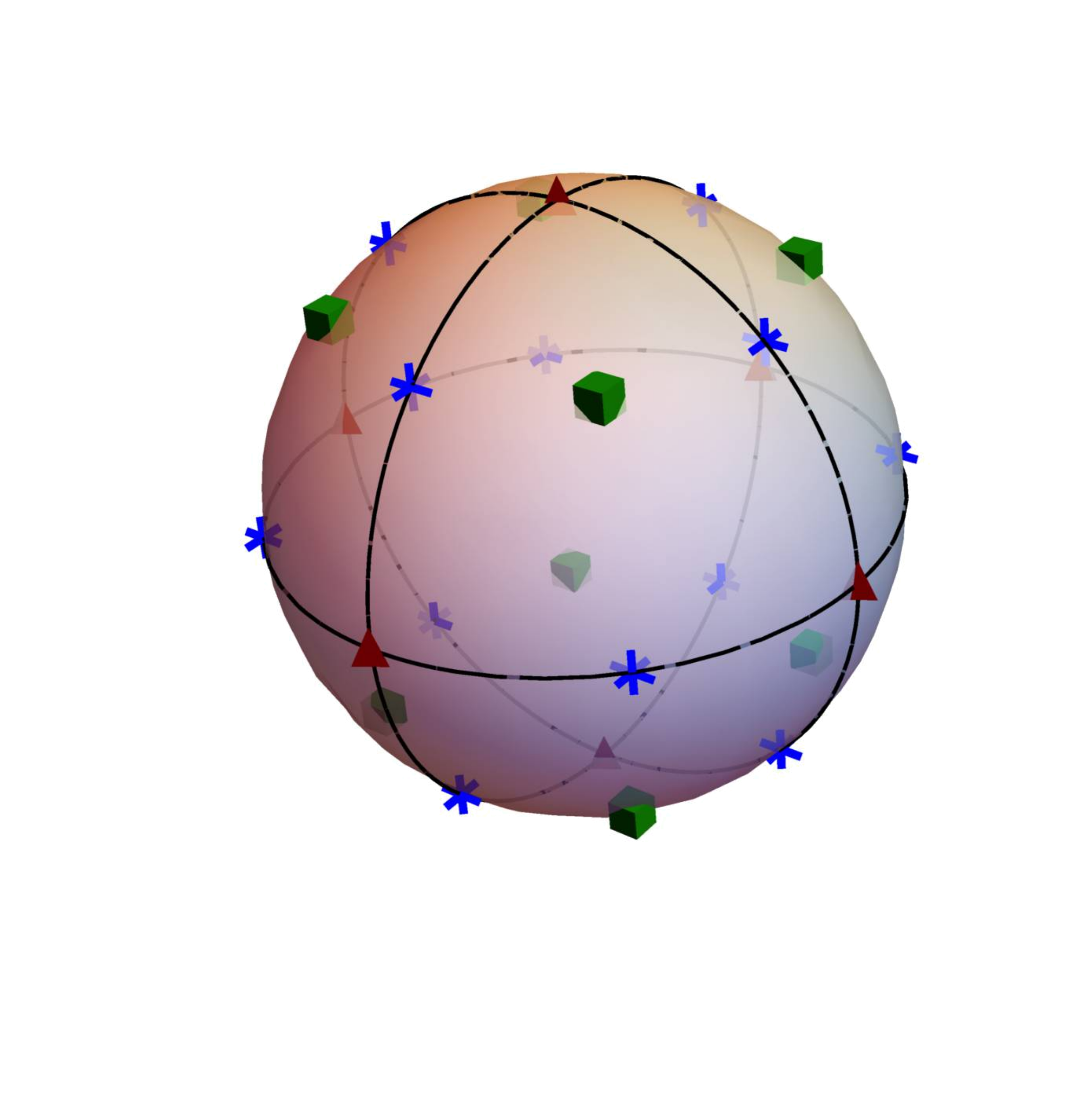}
\hskip 16pt
\includegraphics[width=.4\textwidth]{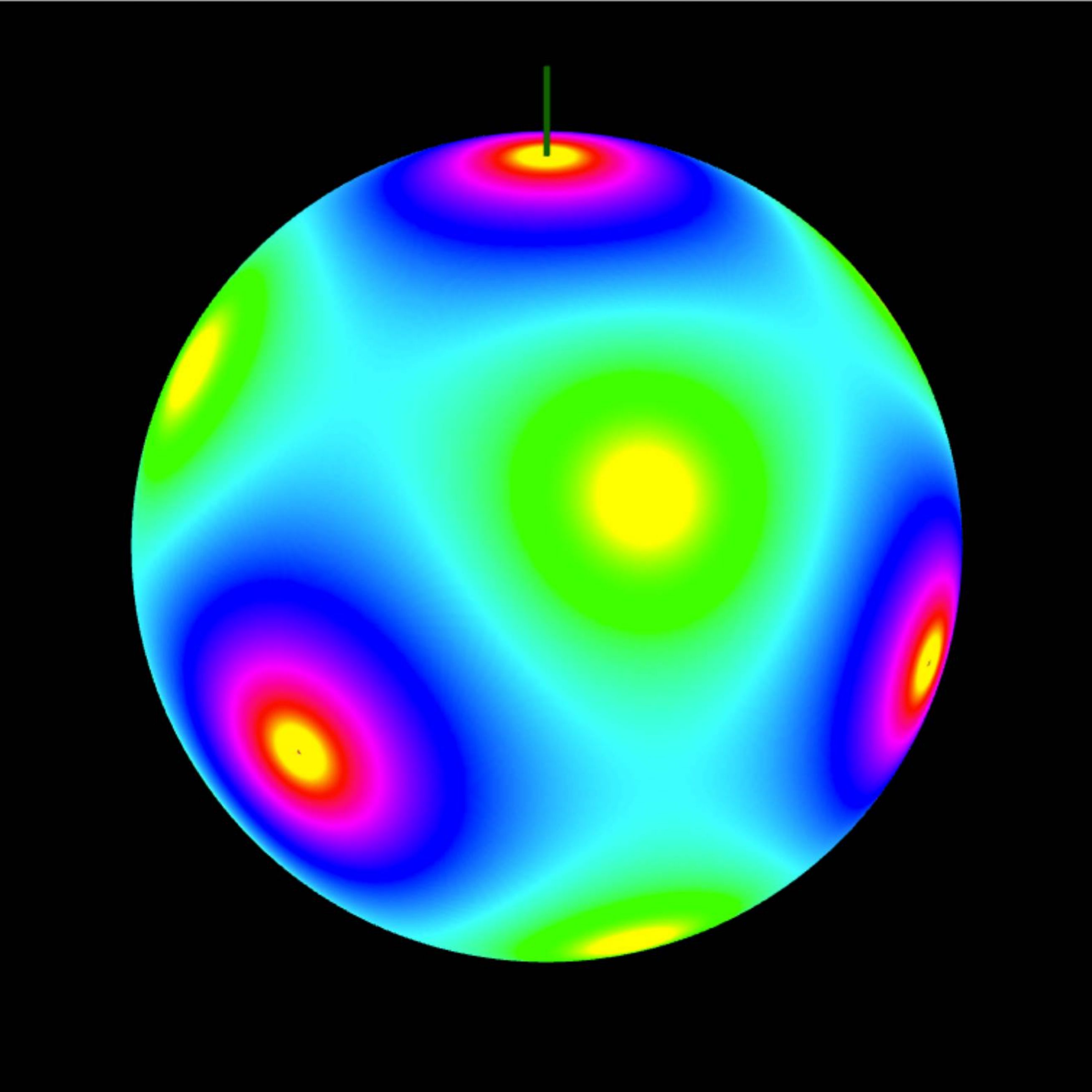}
\caption{Phase portrait of the field associated to the 1--form $\eta$ given by \eqref{octa}. 
Note that $\eta$ is invariant under the isometry group of the octahedron/cube which is isomorphic to
$S_{4}$. In this figure $\lambda= -i$ so that the poles (zeros of the corresponding field) are centers.
On the right hand side, the poles placed at the vertices appear as the centers of the darker concentric anular regions (yellow surrounded by purple and blue), while the poles placed at the 
centers of the faces appear as the centers of the lighter colored disks (yellow surrounded by green).
On the left hand side, vertices appear as (red) triangles, centers of edges appear as (blue) crosses and centers of faces appear as (green) squares.}
\label{figures4}
\end{center}
\end{figure}


\subsection{The case of $A_{5}$}
\label{a5}
Consider the icosahedron $\A_{4}$ together with its isometry group $\G_{4}\cong A_{5}$ from 
Lemma \ref{classFinitos}.
For simplicity we shall use $\A$ and $G$ instead of $\A_{4}$ and $G_{4}$.

Once again we set simple poles at the vertices and at the centers of the faces, and we set simple zeros at the 
midpoints of edges. 
\\
A routine computation then shows that the 1--form thus obtained is
\begin{equation} 
\label{dodecaedro}
\eta = \lambda
\frac{1 - 522 z^5 - 10005 z^{10} - 10005 z^{20} + 522 z^{25} + z^{30}}
{-z - 217 z^6 + 2015 z^{11} + 5890 z^{16} - 2015 z^{21} - 217 z^{26} + z^{31}}
dz.
\end{equation}
By Theorem \ref{main2}.a with $\ell=0$, $\eta$ has isotropy subgroup $G\cong A_{5}$. 
To see that the phase portrait of the associated field is isochronous we verify that all residues of $\eta$ are  
real multiples of $\lambda$. Hence when $\lambda$ is pure imaginary $\eta$ is isochronous.
The phase portrait of the associated field, with $\lambda=-i$, is shown in figure \ref{figurea5}.

\begin{figure}[htbp]
\begin{center}
\includegraphics[width=.35\textwidth]{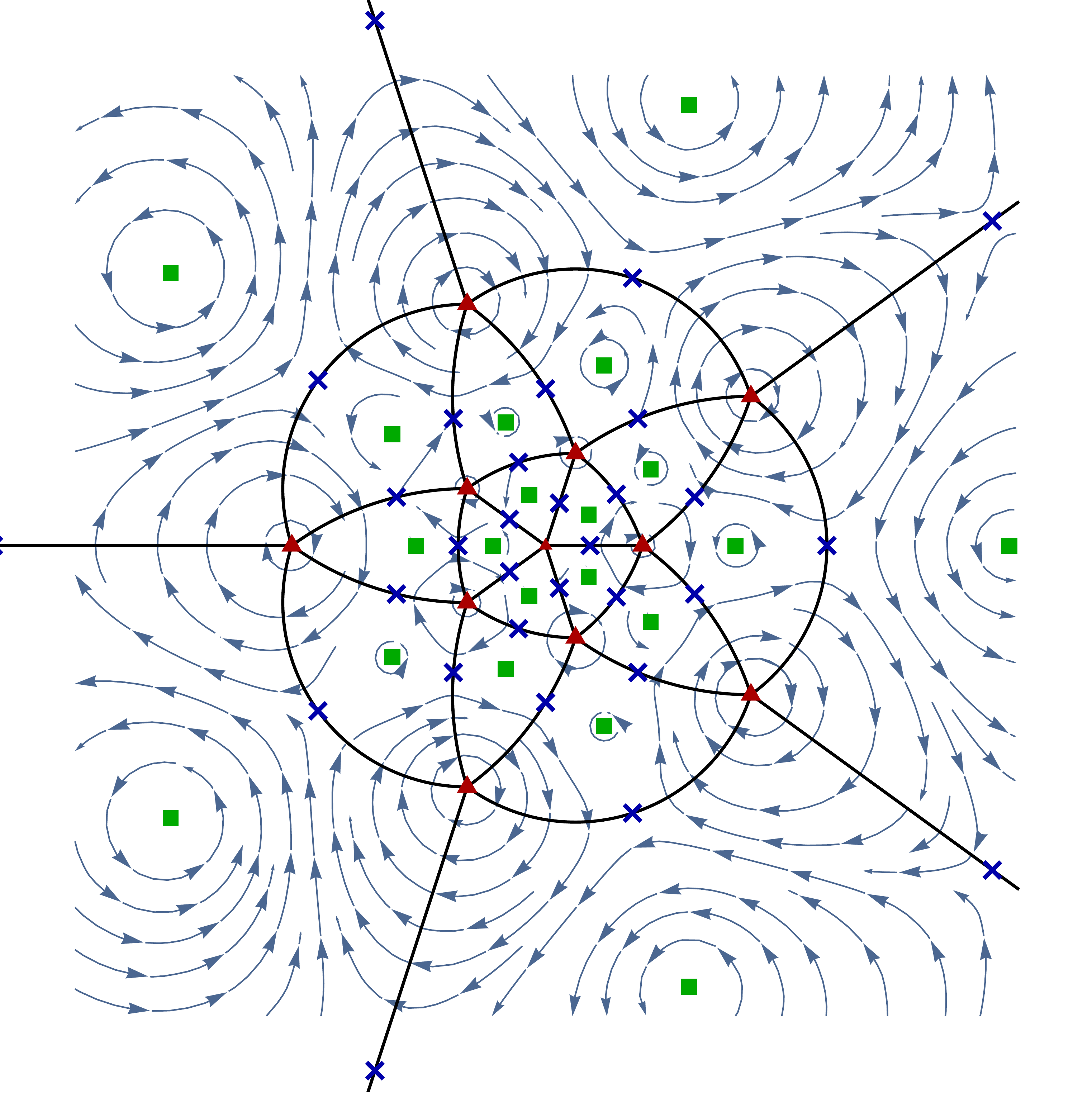}
\hskip 10pt
\includegraphics[width=.35\textwidth]{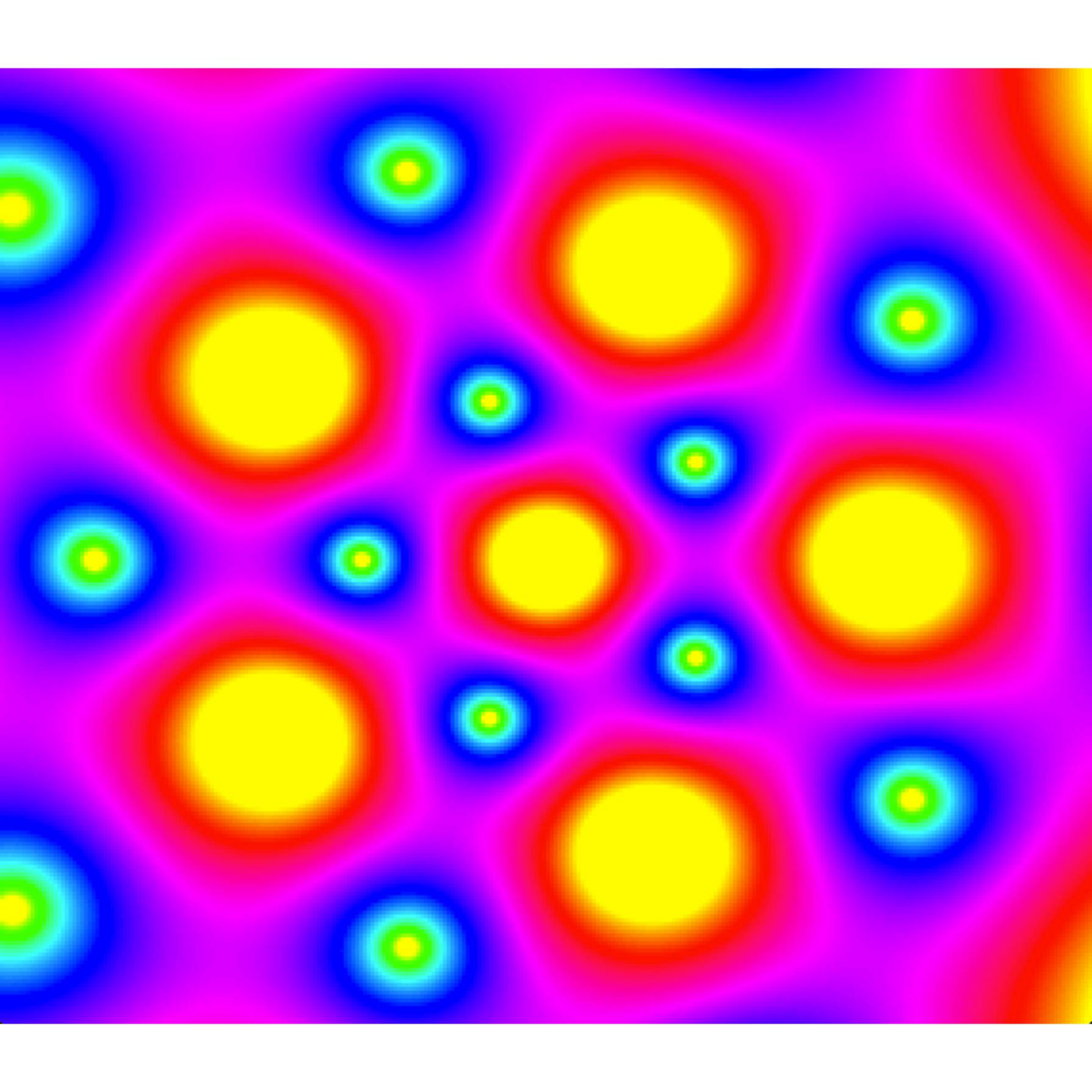}
\\
\includegraphics[width=.35\textwidth]{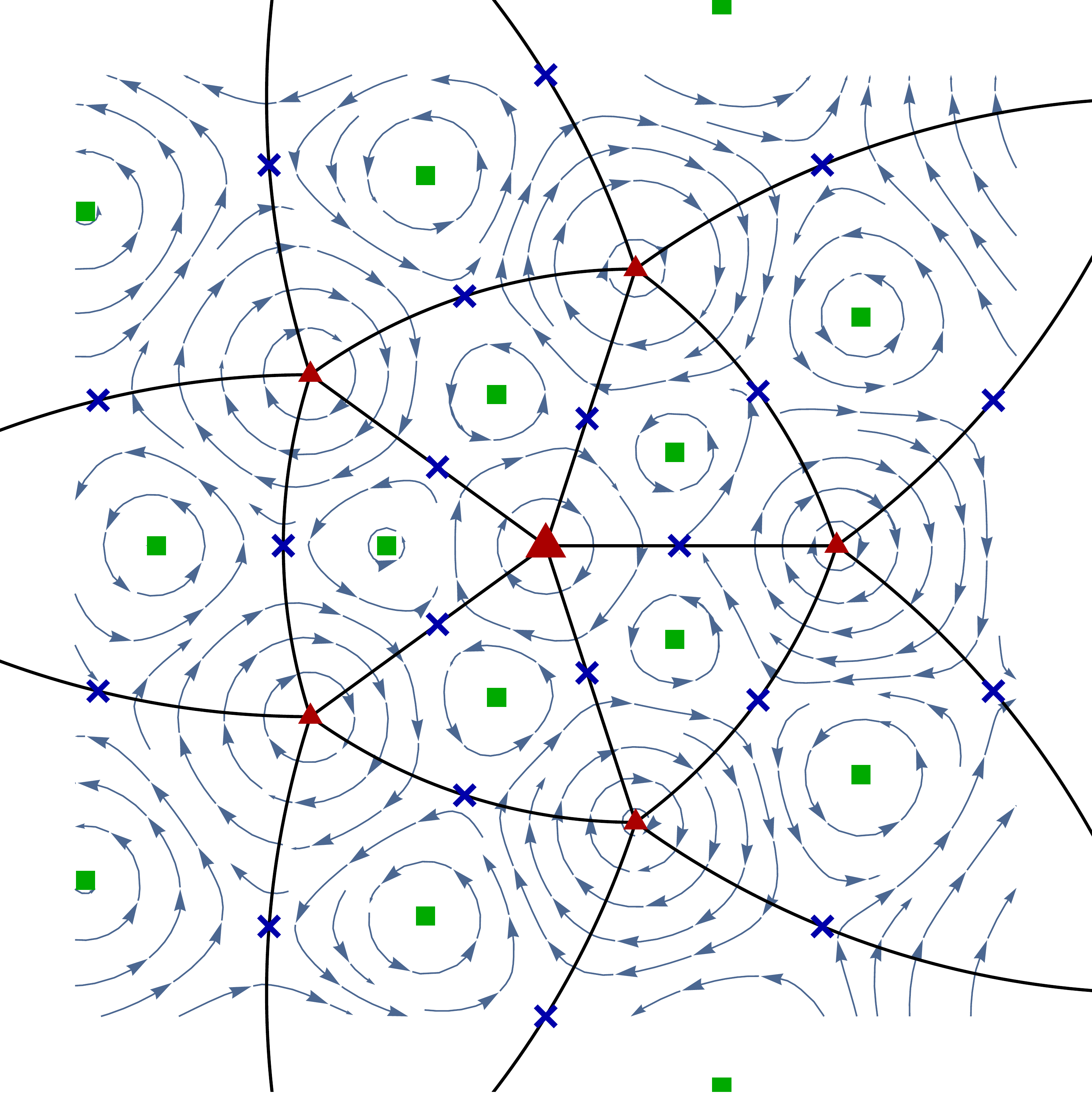}
\hskip 10pt
\includegraphics[width=.35\textwidth]{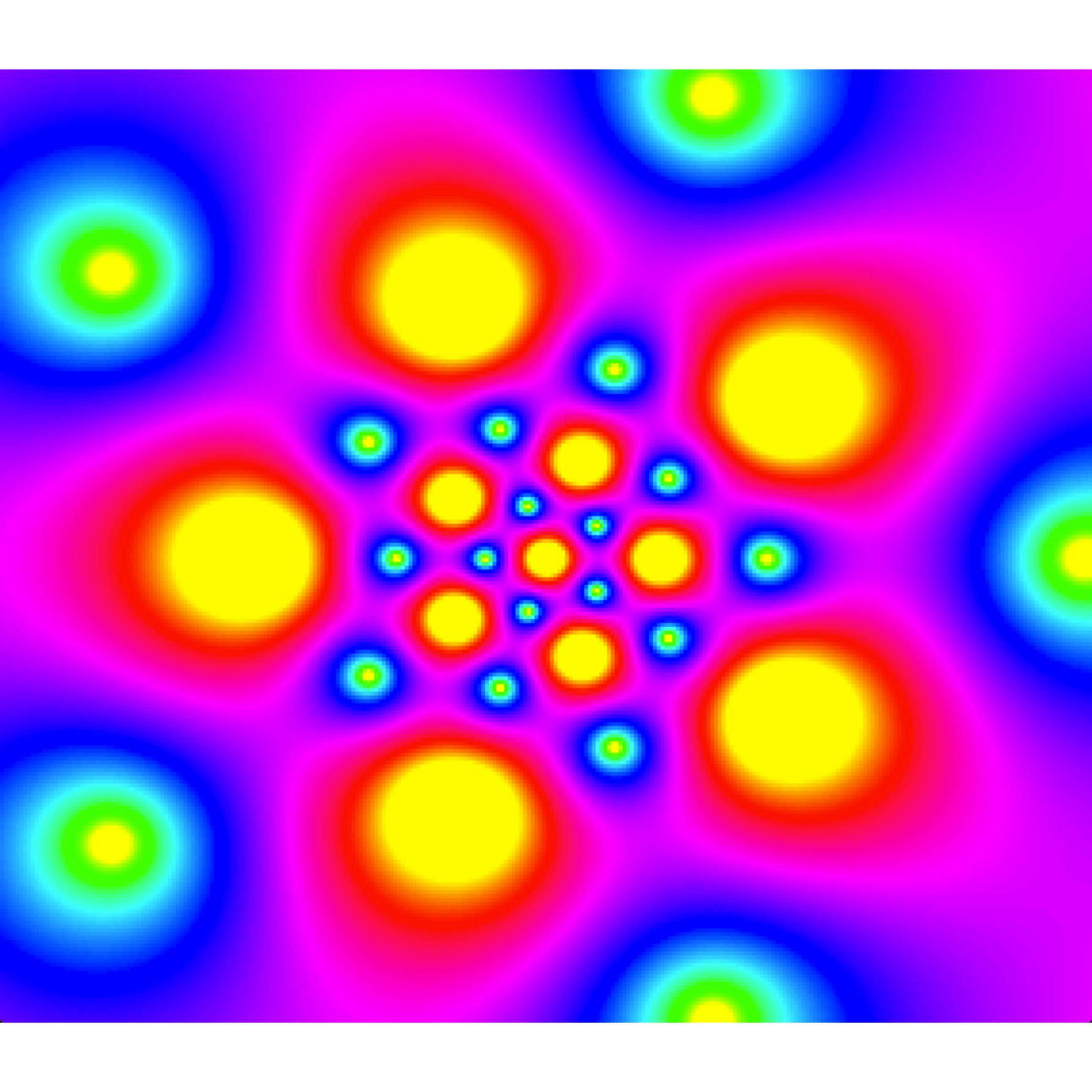}
\\
\includegraphics[width=.35\textwidth]{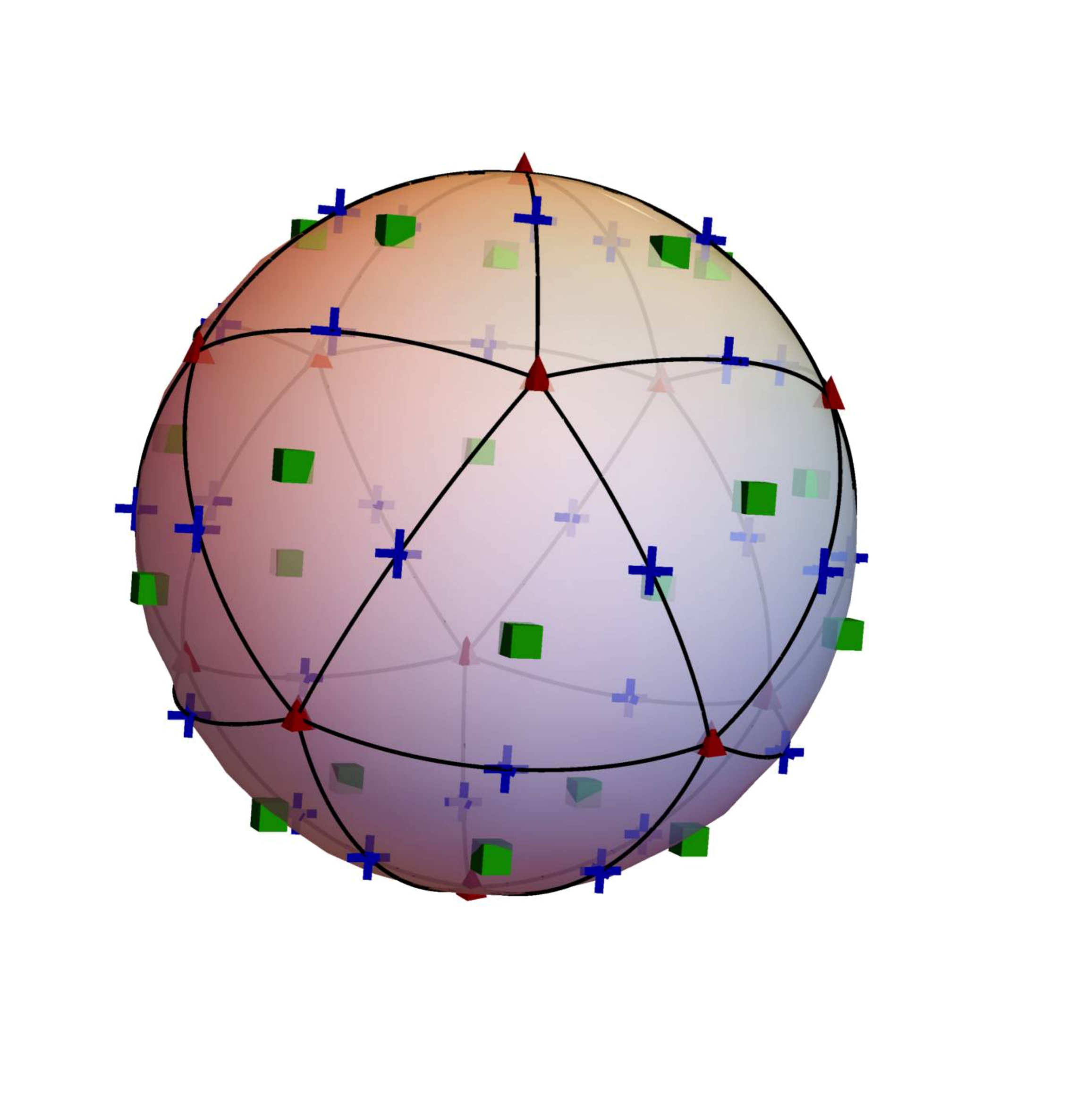}
\hskip 10pt
\includegraphics[width=.35\textwidth]{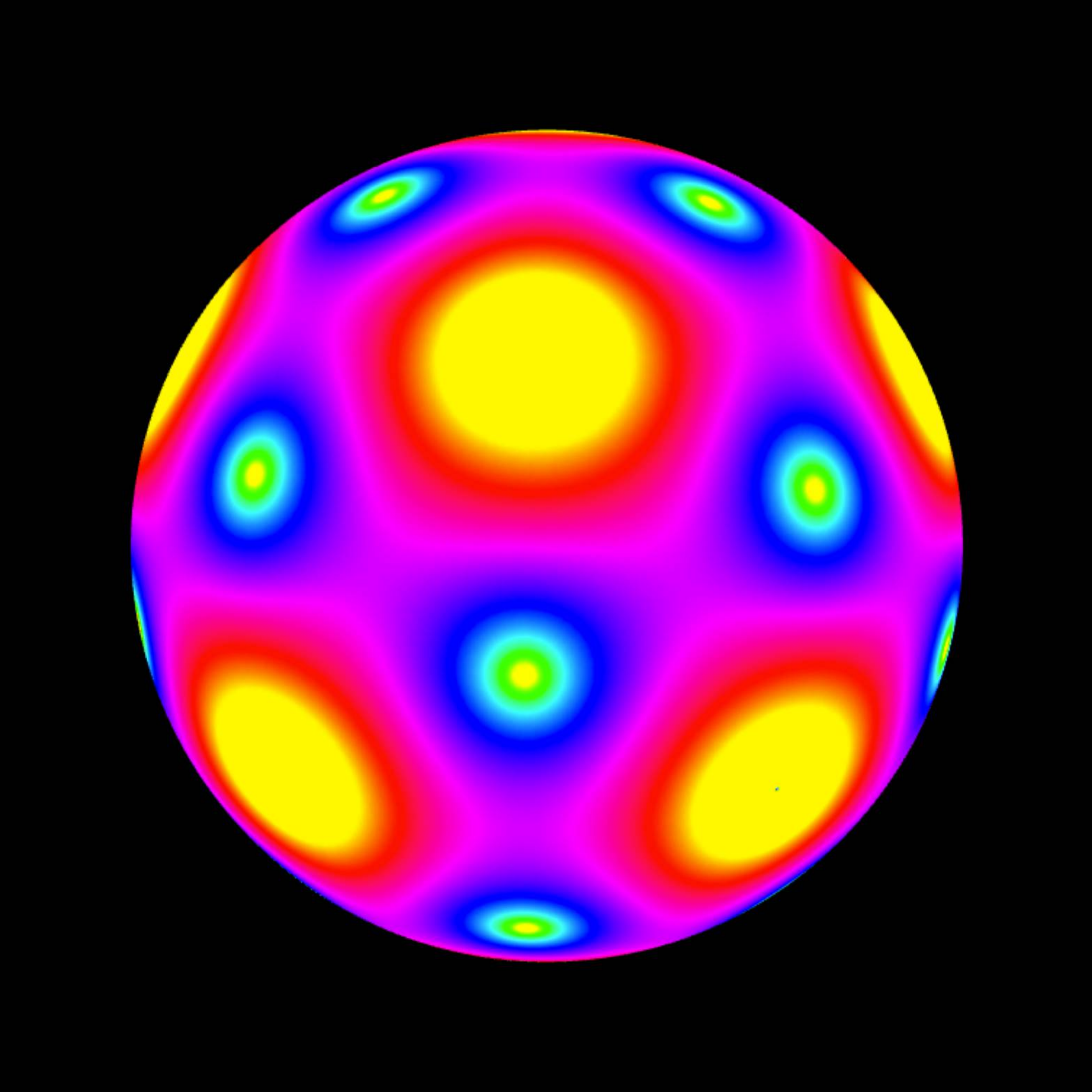}
\caption{Phase portrait of the field associated to the 1--form \eqref{dodecaedro}, corresponding to the 
dodecahedron/icosahedron whose isometry group is isomorphic to $A_{5}$. 
The topmost figures correspond to the field visualized on the rectangle $[-3,3]\times [-3,3]$ while the middle 
figures correspond to the rectangle $[-1,1]\times [-1,1]$. This is done in order to better observe the centers in 
the inner pentagon.
The bottom figures correspond to the field visualized on the Riemann sphere, where one can clearly 
appreciate the symmetries.
To see that the phase portrait of associated field is isochronous we verify that all residues of $\eta$ are real.
Once again $\lambda= -i$ so that the 
poles (zeros of the corresponding field) are centers. 
On the right hand side, the poles placed at the vertices appear as the centers of the lighter colored disks (yellow surrounded by red), 
while the poles placed at the centers of the faces appear as the centers of the darker concentric annular 
regions (yellow surrounded by blue).
On the left hand, side vertices appear as (red) triangles, centers of edges appear as (blue) crosses and centers of faces appear as (green) squares.
}
\label{figurea5}
\end{center}
\end{figure}


\subsection{Dihedral groups $\DD_{n}$}
\label{dihedric}
Consider the dihedron $\A_{6}$ together with its isometry group $\G_{6}\cong \DD_{n}$ from 
Lemma \ref{classFinitos}.
For simplicity we shall use $\A$ and $G$ instead of $\A_{6}$ and $G_{6}$.

We procede in the same way as before, that is we set simple poles on the vertices of the diehdron (the 
$n$--th roots of unity) and on the center of the faces ($0$ and $\infty$), and we set simple zeros on centers of 
the edges (the $n$--th roots of $-1$). In this way we obtain the 1--form:
\begin{equation}\label{diedricoN}
\eta_{n}=f(z)\:dz=\lambda\ \frac{z^n+1}{z(z^n-1)}dz.
\end{equation}

By Theorems \ref{teoremaDihedrico}.A.a (when $n\geq3$) and \ref{D22}.A.a (when $n=2$) with $\ell=0$, 
$\eta_{n}$ is invariant under the group $G\cong \DD_{n}$. 

In order to check that the maximality conditions of Theorems \ref{teoremaDihedrico} and \ref{D22} are 
satisfied, first note that the only finite subgroups of $\PSL$ that could possibly contain 
$G\cong \DD_{n}$ are subgroups $\widehat{G}$ isomorphic to 
\begin{enumerate}
\item $\DD_{m}$ for $n\vert m$, 
\item $A_{4}$ for $n=2$ (see for instance \cite{subA4}), 
\item $S_{4}$ for $n=4$ (see for instance \cite{subS4}), 
\item $A_{5}$ for $n=2,5$ (see for instance \cite{subA5}). 
\end{enumerate}
On the other hand $\eta_{n}$ has exactly $n+2$ simple poles, on $V(\A)\cup F(\A)$, and $n$ simple zeros, 
on $E(\A)$.
From Theorem \ref{teoremaDihedrico} the case $\widehat{G}\cong \DD_{m}$ for $m>n\geq 2$ is not possible. 
From Theorem \ref{main2} the cases $\widehat{G}\cong A_{4}, S_{4}, A_{5}$ are not possible either ($A_{4}$ 
needs 8 poles on $V(\A)\cup F(\A)$ but $\eta_{2}$ only has 4 poles; $S_{4}$ needs 14 poles on 
$V(\A)\cup F(\A)$ but $\eta_{4}$ only has 6 poles; $A_{5}$ needs 32 poles but $\eta_{n}$ has 4 and 7 poles 
respectively for $n=2$ and $5$).

\noindent
Hence we conclude that in fact the isotropy group of $\eta_{n}$ is $\DD_{n}$ for $n\geq 2$.

The phase portrait of the associated field is isochronous since all residues of $\eta$ are real multiples of 
$\lambda$, hence by requiring that $\lambda$ be purely imaginary $\eta$ is isochronous. 
See Figure \ref{diedrico5} for the phase portrait of the associated vector field.

\begin{figure}[htbp]
\begin{center}
\includegraphics[width=.4\textwidth]{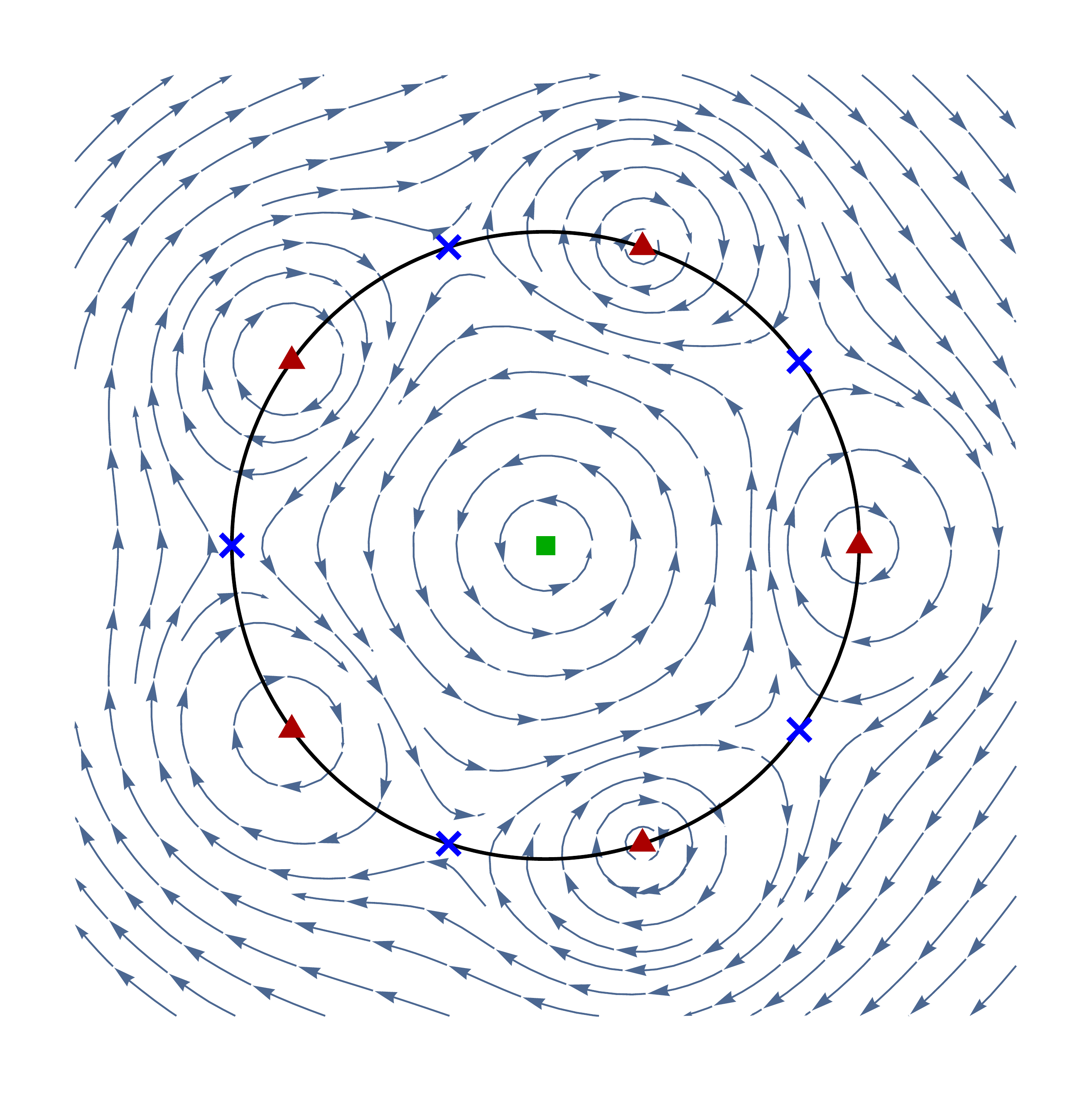}
\hskip 8pt
\includegraphics[width=.4\textwidth]{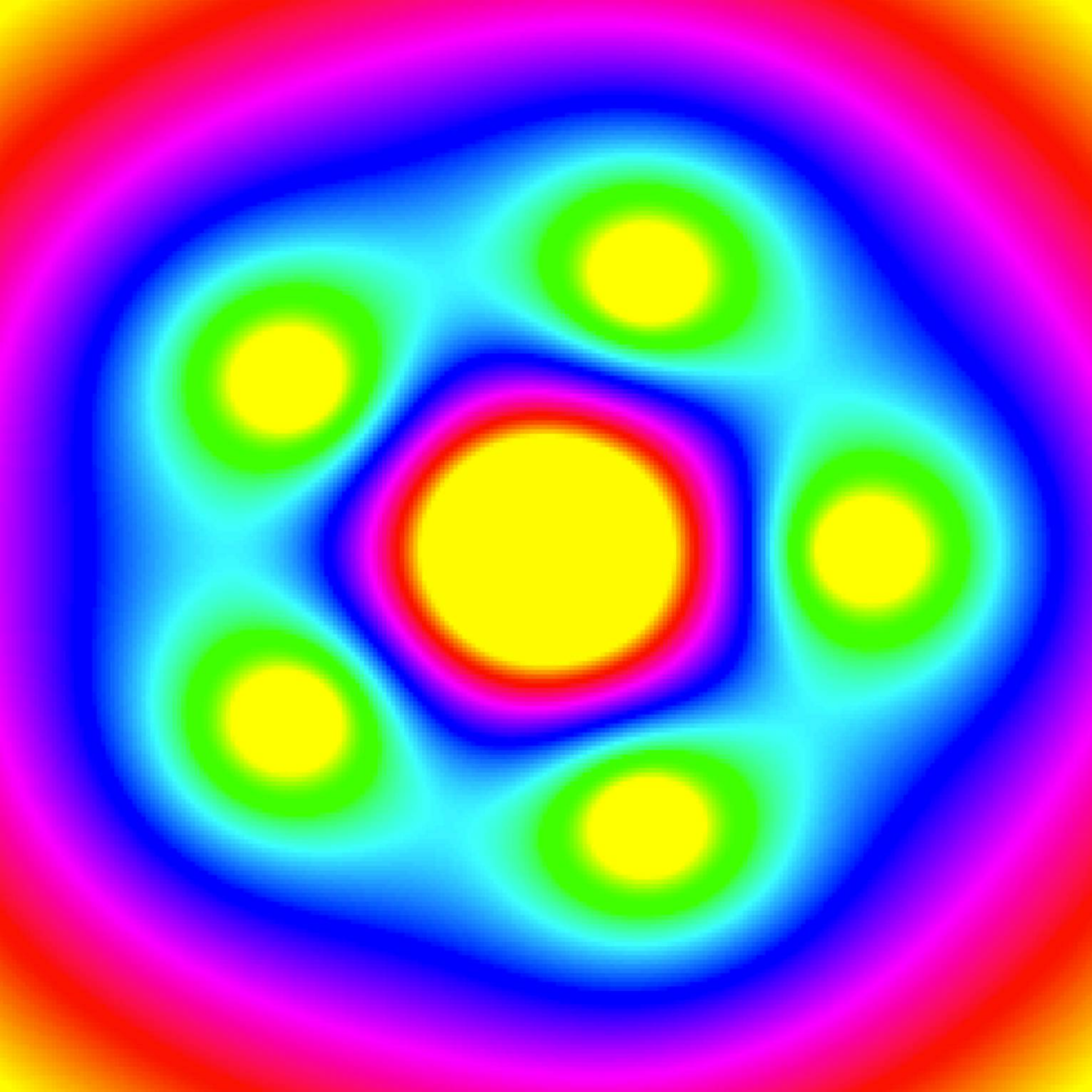}
\\
\includegraphics[width=.375\textwidth]{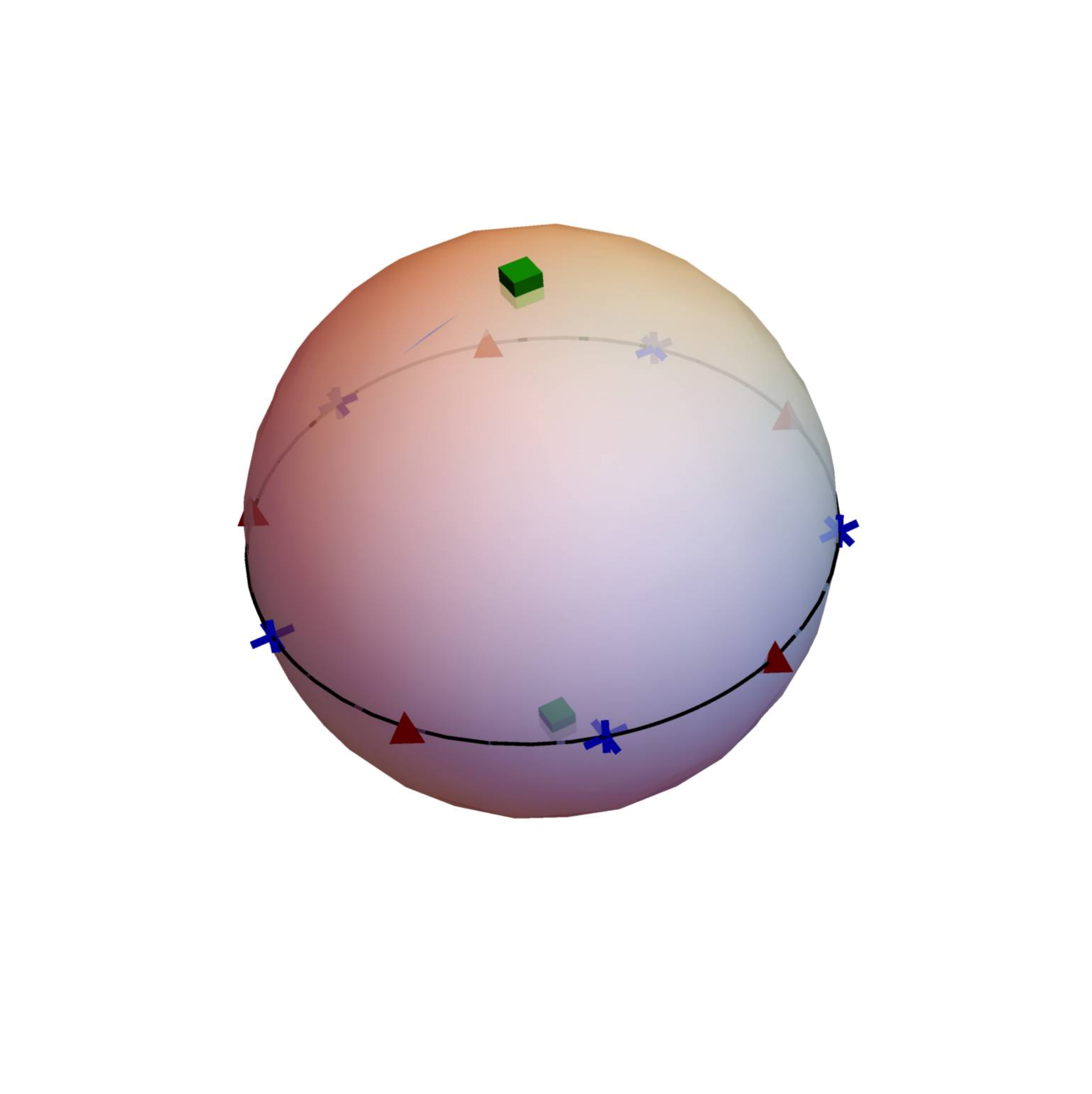}
\hskip 16pt
\includegraphics[width=.4\textwidth]{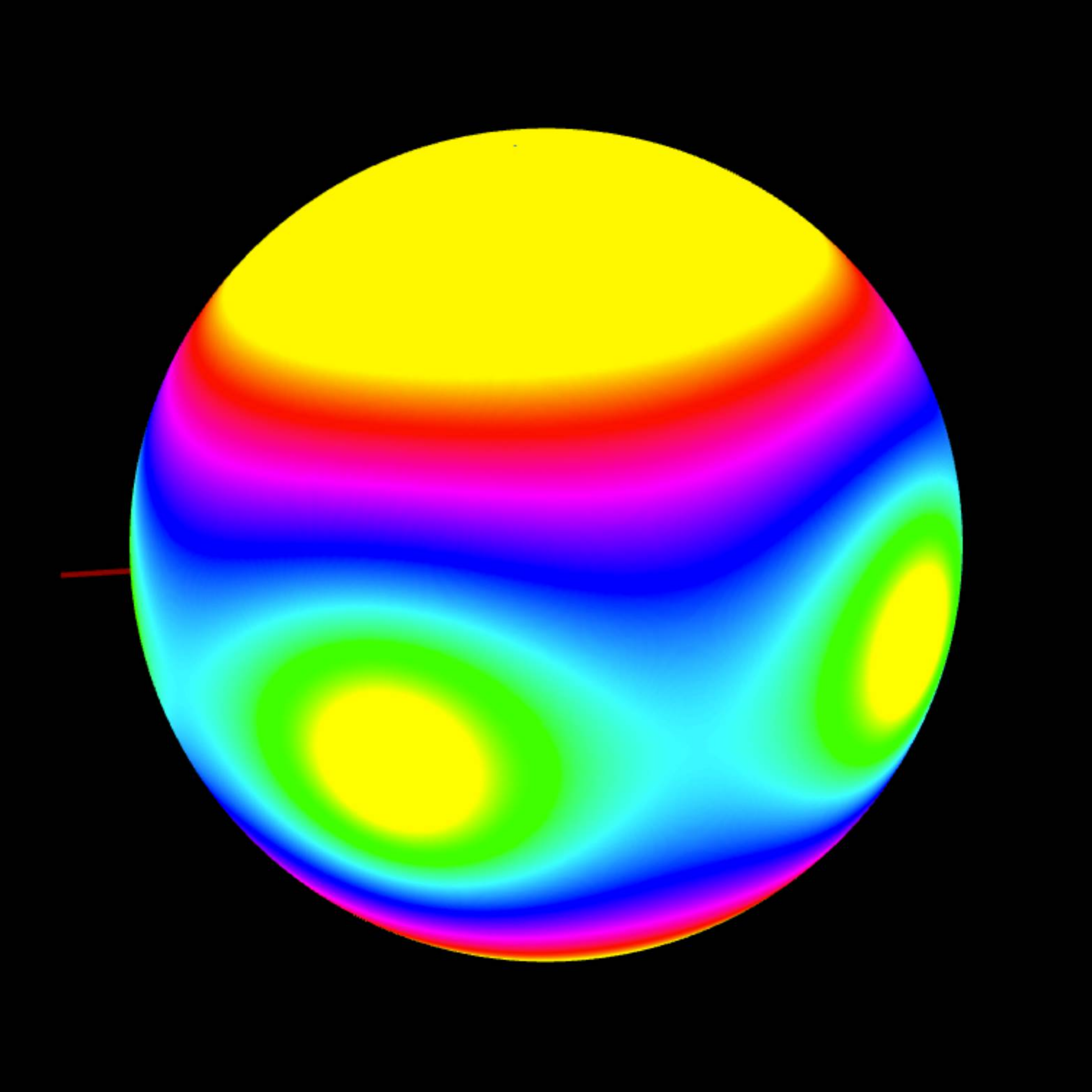}
\caption{Phase portrait of the field associated to the dihedral 1--form \eqref{diedricoN} with $n=5$.
Thus we have the isotropy group isomorphic to $\DD_{5}$. Once again $\lambda= -i$ so that the 
poles (zeros of the corresponding field) are centers. 
On the right hand side, the poles placed at the vertices appear as centers of the smaller light colored disks (yellow surrounded by 
green), while the poles placed at the centers 
of the faces appear as the centers of the large light colored disks surrounded by darker colored annular 
regions (yellow surrounded by red, purple and blue).
On the left hand side, vertices appear as (red) triangles, centers of edges appear as (blue) crosses and centers of faces appear as (green) squares.
}
\label{diedrico5}
\end{center}
\end{figure}


\subsection{Cyclic groups $\ZZ_{n}$}
\label{cyclic}
For the cyclic example, consider 
\begin{equation}\label{EjemploCiclico}
\eta_{n}=\frac{i(1+2 z^n)}{z(z^n-1)}\, dz.
\end{equation}
In this case the poles of $\eta_{n}$ are the $n$--th roots of unity, the origin and $\infty\in\CW$; 
the zeros of $\eta_{n}$ 
are the $n$--th roots of $-1/2$. Hence the conditions of Theorems \ref{Z2}.A and \ref{main3b} are satisfied with 
$\ell=1$, thus $\eta_{n}$ is $\ZZ_{n}$--invariant.

A generator $T\in G\cong\ZZ_{n}$ that leaves invariant $\eta_{n}$ also fixes $0,\infty\in\CW$, hence 
$T(z)=\e^{2 i \pi/n} z$. However, $\eta_{n}$ is not invariant under $T_2(z)=1/z$. 
In fact $T_{2*}\eta_{n}=\frac{i(2+ z^n)}{z(z^n-1)}\neq \eta_{n}$. This shows that $\eta_{n}$  
is not invariant under a group isomorphic to $\DD_{m}$ for $m\in\NN$.
Moreover, it is clear that there is no $\widehat{G}\cong\ZZ_{m}$, $m> n$, such that $\eta_{n}$ is 
$\widehat{G}$--invariant.

\noindent
Finally since the group structure of $A_{5}$, $S_{4}$ and $A_{4}$ does not contain a subgroup isomorphic to 
$\ZZ_{n}$ for $n\neq 2, 4, 5$, then for these values of $n$ the rational 1--form $\eta_{n}$ can not be invariant 
under a group isomorphic to $A_{5}$, $S_{4}$ or $A_{4}$. 

\noindent
In the case of $n=2, 4, 5$, even though each group $\widehat{G}$ isomorphic to $A_{5}, S_{4}$ or $A_{4}$ 
does contain a subgroup isomorphic to $\ZZ_{n}$ 
(for $n=2, 4, 5$) it factors through a dihedric $\DD_{n}$ on its way to $\ZZ_{n}$, so if $\eta_{n}$ is 
$\widehat{G}$--invariant, it would also have to be $\DD_{n}$--invariant which has already been shown to 
not occur for $\eta_{n}$. 

\noindent
Thus $\eta_{n}$ has isotropy group $G\cong\ZZ_{n}$.

\noindent
To see that the phase portrait of the associated field is isochronous we verify that all residues of $\eta$ are  
real multiples of $\lambda$. Hence for $\lambda$ purely imaginary $\eta$ is isochronous.
As examples see Figures \ref{ciclico2} and \ref{ciclico5} that correspond to the cases $n=2$ and $n=5$ 
respectively.

\begin{figure}[htbp]
\begin{center}
\includegraphics[width=.4\textwidth]{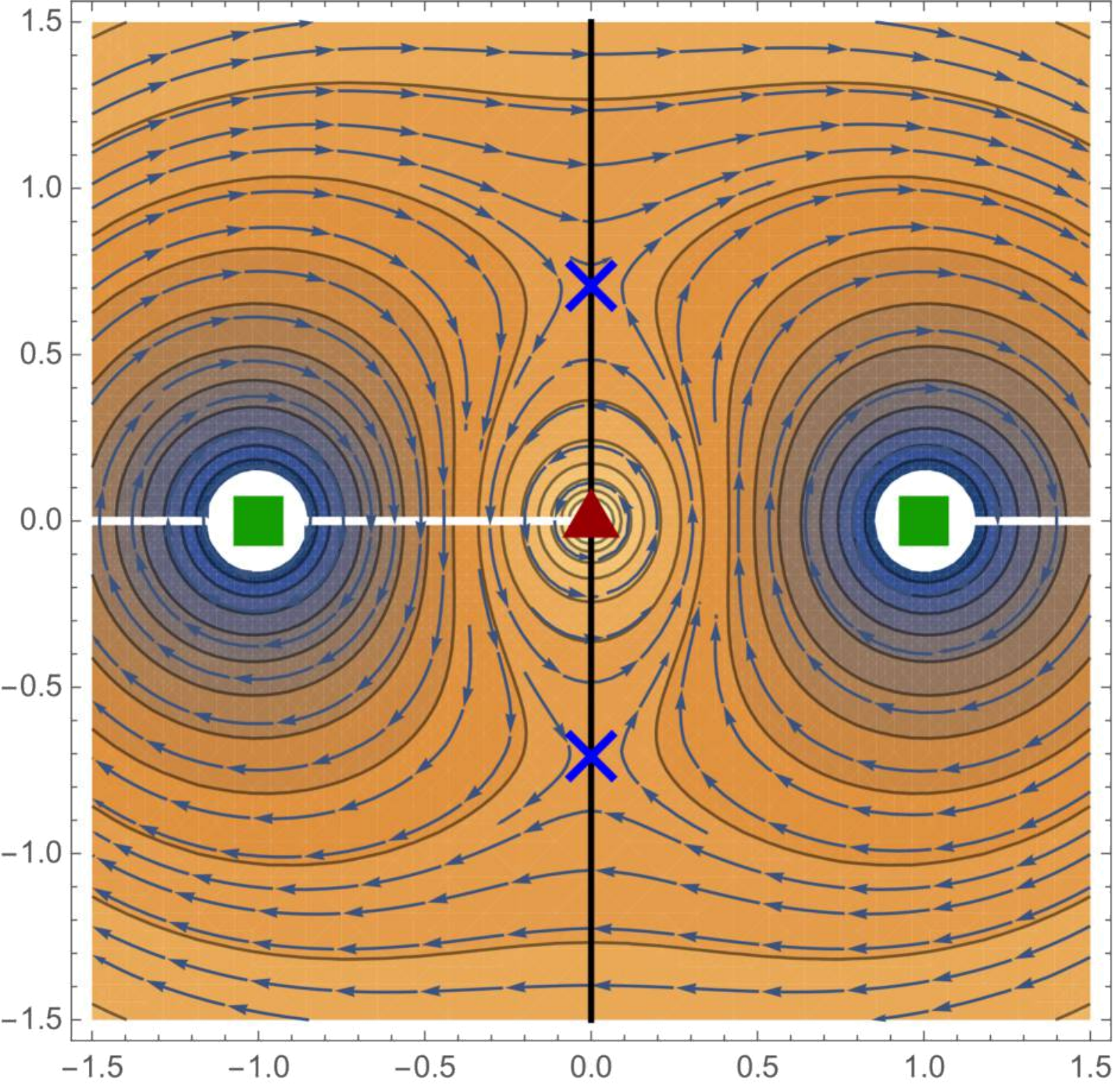}
\\
\includegraphics[width=.45\textwidth]{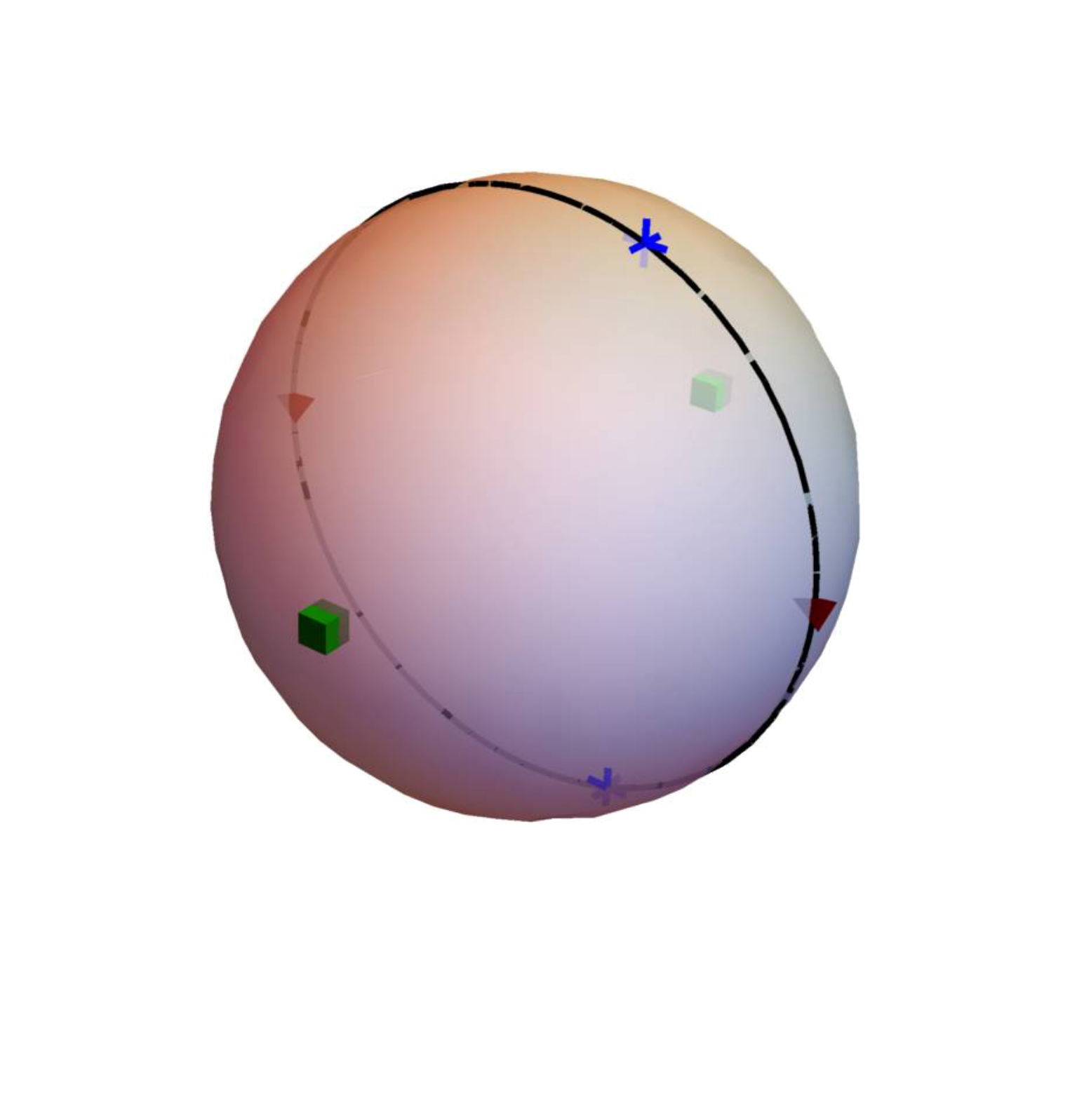}
\hskip 16pt
\includegraphics[width=.4\textwidth]{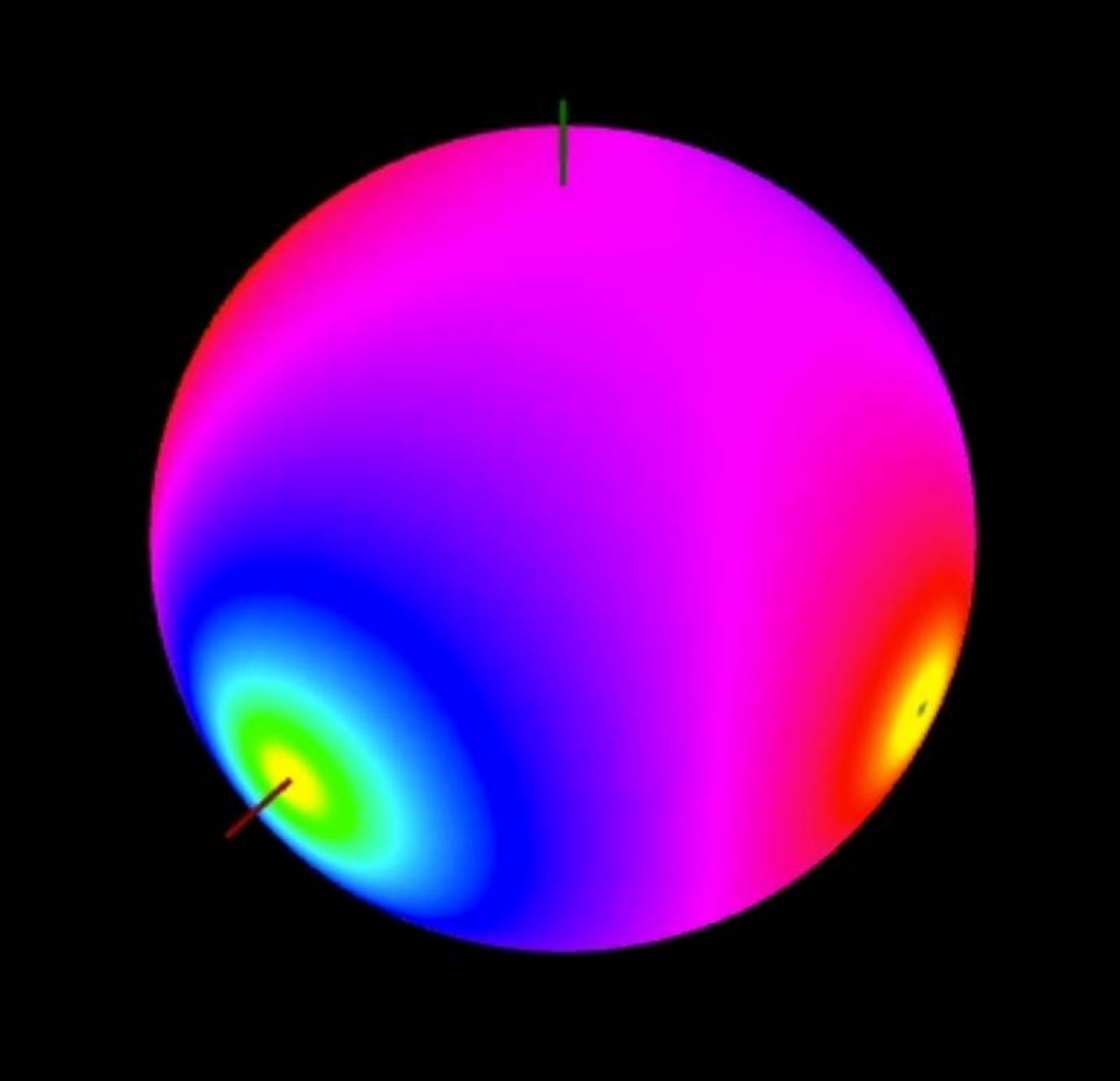}
\caption{Phase portrait of the field associated to the cyclic 1--form \eqref{EjemploCiclico} with $n=2$.
Thus we have the isotropy group isomorphic to $\ZZ_{2}$. In this case $\lambda= i$ so that the 
poles (zeros of the corresponding field) are centers. 
The poles placed at the vertices appear in the top and left bottom figure as (red) triangles at $0, \infty\in\CW$ and in the bottom right hand figure as the center of the small light colored disks (yellow surrounded by red), while the poles placed at the centers of the faces appear in the top and left bottom figure as (green) squares at $-1, 1$ and in the bottom right hand figure as the centers of the large light colored concentric annular regions (yellow surrounded by green and blue). The zeros appear as (blue) crosses on the top and left bottom figures.
The restriction for the placement of the poles and zeros on $\CC^{*}$ is that one can place them on a quasi--fundamental region in such a way that they are not on the same concentric circle centered at the origin.}
\label{ciclico2}
\end{center}
\end{figure}

\begin{figure}[htbp]
\begin{center}
\includegraphics[width=.4\textwidth]{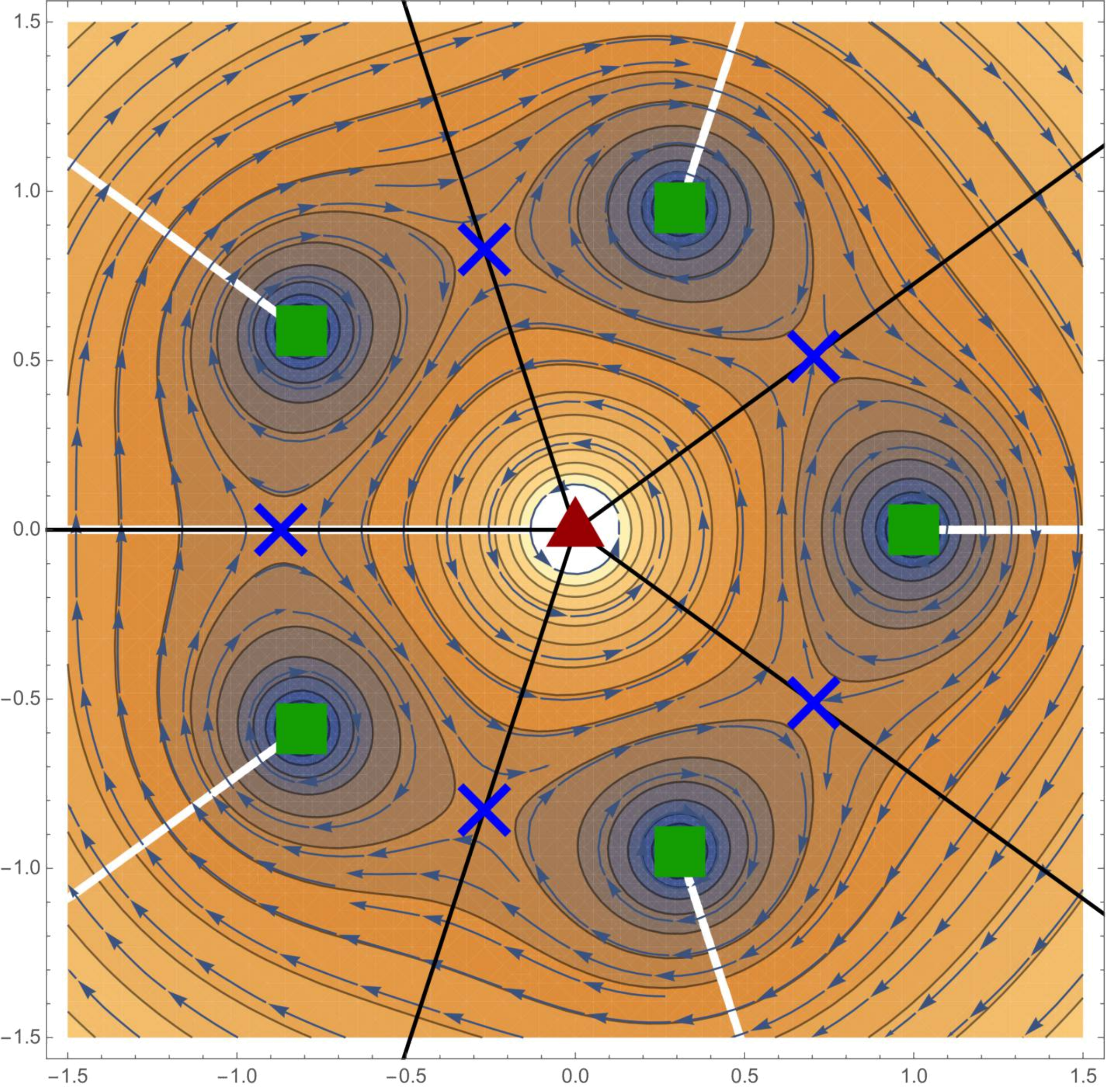}
\\
\includegraphics[width=.4\textwidth]{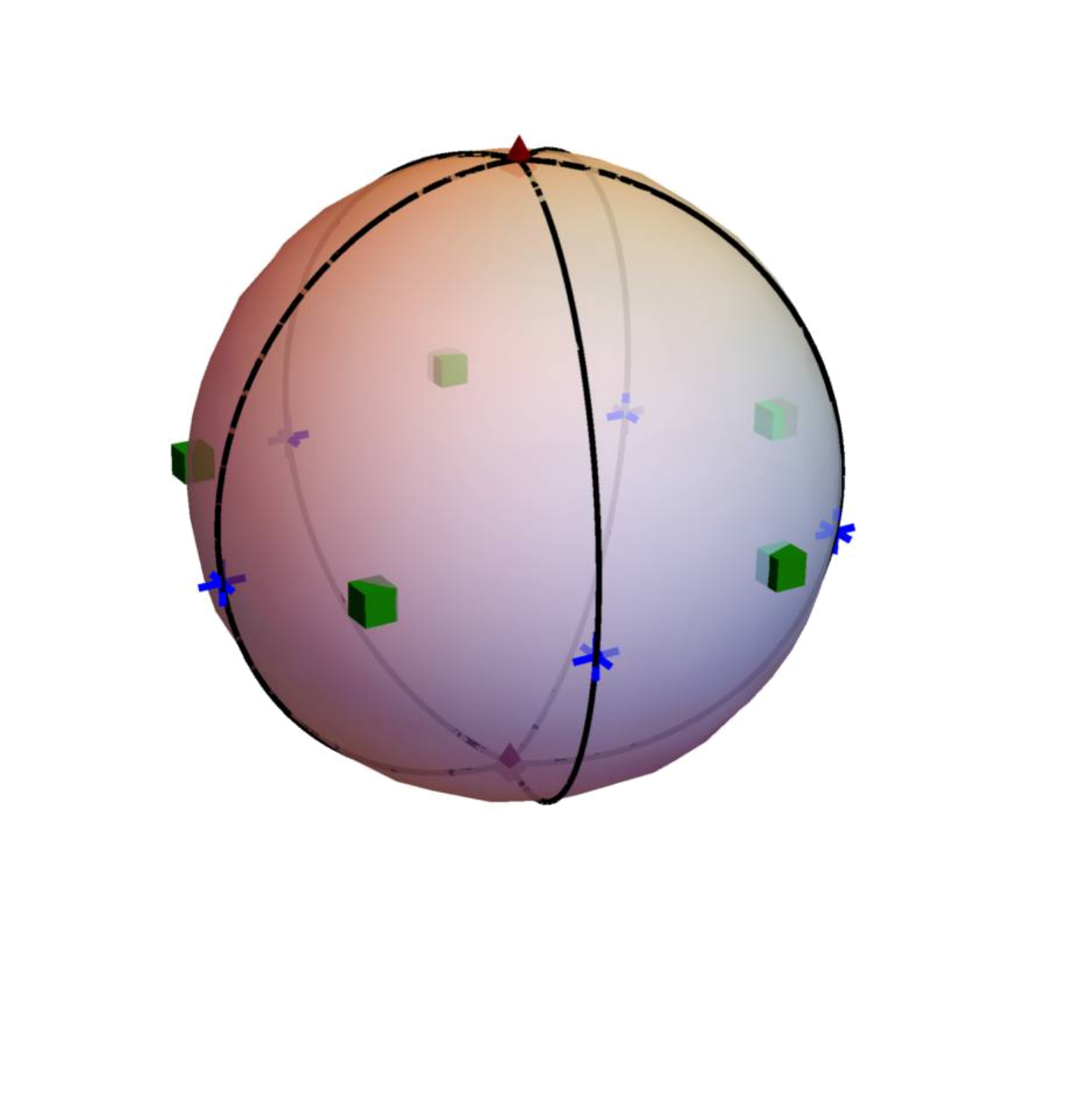}
\hskip 16pt
\includegraphics[width=.4\textwidth]{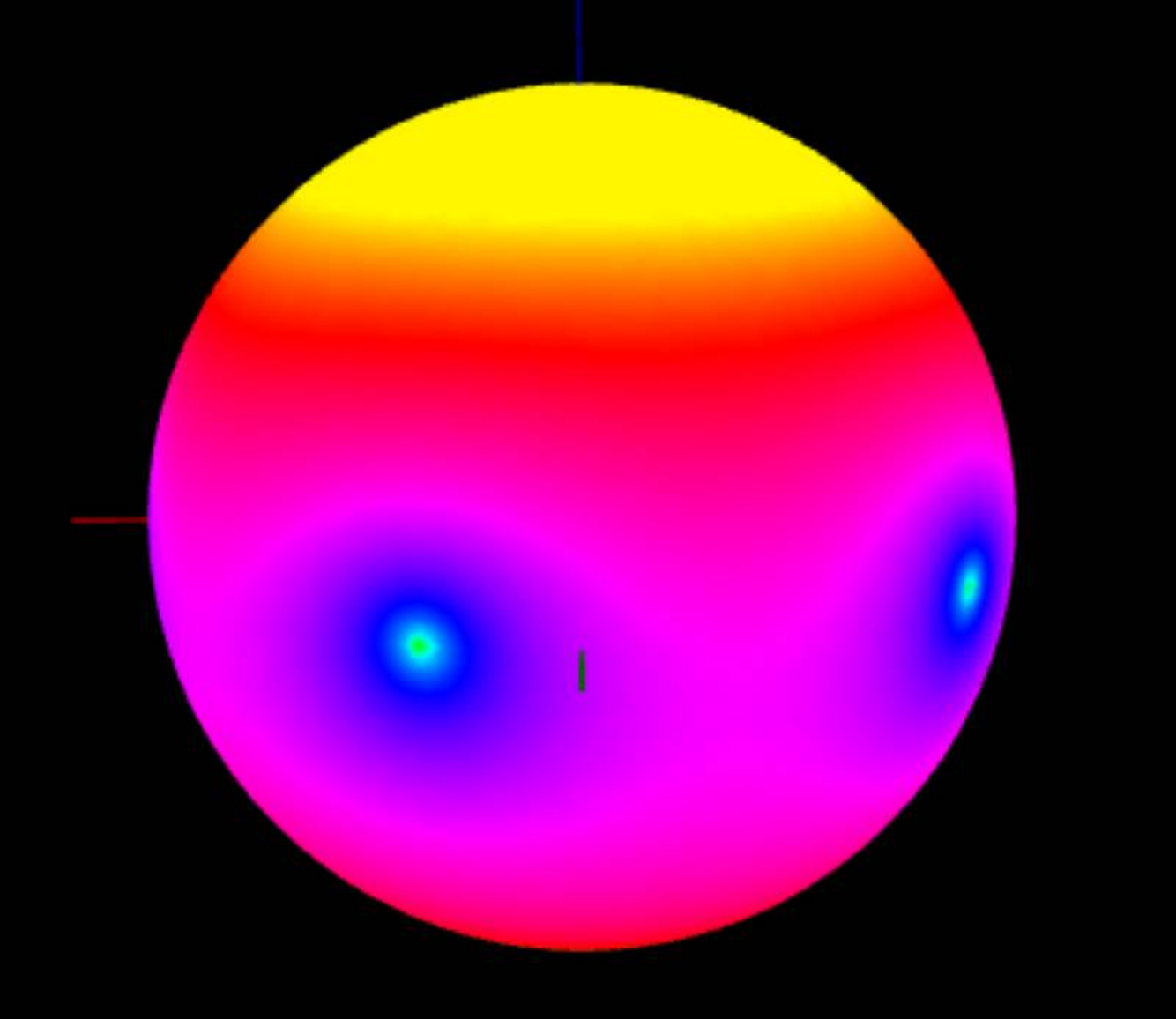}
\caption{Phase portrait of the field associated to the cyclic 1--form \eqref{EjemploCiclico} with $n=5$.
Thus we have the isotropy group isomorphic to $\ZZ_{5}$. Once again $\lambda= -i$ so that the
poles (zeros of the corresponding field) are centers. 
On the bottom right figure, the poles placed at the vertices appear as the centers of the large lighter colored disks (yellow), while the poles placed at the centers of the faces appear as the centers of the small darker colored concentric annular regions (blue) and the zeros are the saddles that are on the edges (but not on the center of the edges).
In the other two figures the vertices are the (red) triangles which correspond to the poles at the origin and $\infty\in\CW$, the (blue) crosses are zeros placed on the edges (but not on the centers of the edges) and the (green) squares are poles placed on the quasi--fundamental region with the restriction that they are not on the same parallel as the zeros (in this case these poles are placed on the center of the faces).}
\label{ciclico5}
\end{center}
\end{figure}


\end{document}